\documentclass{amsart}
\usepackage{amssymb}

\usepackage{amsmath}
\usepackage{txfonts}

\input xy
\xyoption{all}

\usepackage{amsfonts}
\usepackage{mathrsfs}

\usepackage{latexsym}
\usepackage{graphicx}
\usepackage{amscd,amssymb,amsmath,amsbsy,amsthm}
\usepackage[all]{xy}
\usepackage[colorlinks,plainpages,backref,urlcolor=blue]{hyperref}
\usepackage{verbatim}

\usepackage [english]{babel}
\usepackage [autostyle, english = american]{csquotes}
\MakeOuterQuote{"}

\newtheorem{theorem}{Theorem}[section]
\newtheorem{lemma}[theorem]{Lemma}
\newtheorem{corollary}[theorem]{Corollary}
\newtheorem{prop}[theorem]{Proposition}

\theoremstyle{definition}
\newtheorem{definition}[theorem]{Definition}
\newtheorem{example}[theorem]{Example}
\newtheorem{remark}[theorem]{Remark}

\newtheorem{assumption}[theorem]{Assumption}

\newtheorem{thm}{Theorem}

\DeclareMathOperator{\rank}{rank} 
\DeclareMathOperator{\im}{im} \DeclareMathOperator{\coker}{coker}
\DeclareMathOperator{\codim}{codim} 
 \DeclareMathOperator{\Sym}{Sym}
 \DeclareMathOperator{\supp}{supp}
\DeclareMathOperator{\depth}{depth} \DeclareMathOperator{\GL}{GL}
\DeclareMathOperator{\Hom}{{Hom}} 
\DeclareMathOperator{\Ext}{{Ext}}
\DeclareMathOperator{\sHom}{{\mathscr{H}om}}
\DeclareMathOperator{\sExt}{{\mathscr{E}xt}}

\DeclareMathOperator{\proj}{pr} 
\DeclareMathOperator{\Spec}{{Spec}} 
 \DeclareMathOperator{\Rep}{{Rep}}
\DeclareMathOperator{\Grass}{Grass} \DeclareMathOperator{\End}{End}
\DeclareMathOperator{\sEnd}{{\mathscr{E}nd}}
\DeclareMathOperator{\rad}{rad}
\DeclareMathOperator{\catA}{\mathfrak{A}}
\DeclareMathOperator{\Res}{{Res}}
\DeclareMathOperator{\Ind}{{Ind}}

\newcommand{\surj}{\twoheadrightarrow}
\newcommand{\inj}{\hookrightarrow}

\newcommand{\sO}{\mathcal{O}}
\newcommand{\sF}{\mathscr{F}}
\newcommand{\sA}{\mathscr{A}}
\newcommand{\cF}{\mathcal{F}}
\newcommand{\cE}{\mathcal{E}}
\newcommand{\cG}{\mathcal{G}}
\newcommand{\cZ}{\mathcal{Z}}
\newcommand{\cY}{\mathcal{Y}}
\newcommand{\cV}{\mathcal {V}}
\newcommand{\cR}{\mathcal {R}}
\newcommand{\cQ}{\mathcal {Q}}
\newcommand{\Z}{\mathbb{Z}}
\newcommand{\SCH}{\mathbb{L}}
\newcommand{\TIL}{\mathscr{T}il}
\newcommand{\WEY}{\mathbb{K}}
\newcommand{\Gm}{\mathbb{G}_m}
\newcommand{\CC}{\mathbb{C}}

\newcount\cols
{\catcode`,=\active\catcode`|=\active
 \gdef\Young(#1){\hbox{$\vcenter
 {\mathcode`,="8000\mathcode`|="8000
  \def,{\global\advance\cols by 1 &}%
  \def|{\cr
        \multispan{\the\cols}\hrulefill\cr
        &\global\cols=2 }%
  \offinterlineskip\everycr{}\tabskip=0pt
  \dimen0=\ht\strutbox \advance\dimen0 by \dp\strutbox
  \halign
   {\vrule height \ht\strutbox depth \dp\strutbox##
    &&\hbox to \dimen0{\hss$##$\hss}\vrule\cr
    \noalign{\hrule}&\global\cols=2 #1\crcr
    \multispan{\the\cols}\hrulefill\cr%
   }
 }$}}
}

\newcommand{\fp}{\mathfrak{p}}
\newcommand{\calM}{\mathcal{M}}
\newcommand{\calR}{\mathcal{R}}
\newcommand{\calO}{\mathcal{O}}
\newcommand{\calQ}{\mathcal{Q}}

\newcommand{\calS}{\mathcal{S}}
\newcommand{\calT}{\mathcal{T}}

\DeclareMathOperator{\gl}{{gl}}
\DeclareMathOperator{\add}{{add}}
\DeclareMathOperator{\Ref}{{ref}}
\DeclareMathOperator{\Eq}{{Eq}}

\title[Noncommutative desingularization of orbit closures]
{Noncommutative desingularization of orbit closures for some
representations of $GL_n$}

\author[J.~Weyman]{Jerzy~Weyman}
\address{Department of Mathematics,
Northeastern University, Boston, MA 02115, USA}
\email{j.weyman@neu.edu}

\author[G.~Zhao]{Gufang~Zhao}
\address{Department of Mathematics,
Northeastern University, Boston, MA 02115, USA}
\email{zhao.g@husky.neu.edu}

\date{\today}

\begin{document}

\begin{abstract}
We describe noncommutative desingularizations of
determinantal varieties, determinantal varieties defined by minors of generic symmetric matrices, and
pfaffian varieties defined by pfaffians of generic anti-symmetric matrices. For maximal minors of square matrices and symmetric matrices, this gives a non-commutative crepant resolution. Along the way, we describe a method to calculate the quiver with relations for any non-commutative desingularizations coming from exceptional collections over partial flag varieties.
\end{abstract}

\maketitle
\tableofcontents

\section{Introduction}
For a singular variety $X$, a \textit{non-commutative desingularization} (Definition~\ref{def: non-c desing}) is a coherent sheaf of associative algebras  $\sA$ over a proper scheme $Y$ over $X$ birationally equivalent to $X$, such that $\sA$ is generically a sheaf of matrix algebras and has finite homological dimension. This notion arises  from the study of the derived categories of coherent sheaves, Lie theory, maximal Cohen-Macaulay modules, as well as theoretical physics. In the case when $X$ has Gorenstein singularities,  Van den Bergh introduced a notion of \textit{non-commutative crepant resolution} (Definition~\ref{def: nccr}, c.f. also \cite{VDB2}), i.e., a non-commutative desingularization which is maximal Cohen-Macaulay as a coherent sheaf on $X$. This notion is a generalization of an existent notion of crepant resolution, and is introduced with the hope that any two crepant resolutions, commutative or not, are derived equivalent.
\par
The idea of non-commutative desingularization can be traced back to the general theory of derived Morita equivalence due to Rickard et al.. 
Let $k$ be a field. An ordered set of objects $\Delta=\{\Delta_\alpha,\alpha\in I\}$ in a triangulated $k$-linear category
$D$ is called  exceptional if we have
$\Ext^\bullet(\Delta_\alpha,\Delta_\beta ) = 0$ for $\alpha < \beta$ and $\End(\Delta_\alpha)
= k$; it is said to be strongly exceptional if further
$\Ext^n(\Delta_\alpha,\Delta_\beta) = 0$ for $n\neq 0$. An exceptional set is said to be full if it generates
$D$. If $X$ is a projective variety, $(E_0,\cdots,E_n)$ is a full strongly exceptional collection in the derived category of coherent sheaves on $X$. Then there is an equivalence of derived categories$$R\Hom(\oplus_iE_i,-):D^b(X)\to D^b(A\hbox{-}mod),$$ where $A\hbox{-}mod$ is the category of finitely generated left modules over $A:=\End(\oplus_iE_i)$ with the opposite multiplication. Note that by properness of $X$, the algebra $A$ is finite dimensional as vector space over $k$, and can be described as the path algebra of a quiver with relations.
\par
Let $Z$ be a quasi-projective variety, e.g., the total space of a vector bundle over a projective variety $X$, then $Z$ rarely admits an exceptional collection. Nevertheless, if it does admit tilting objects, the category of coherent sheaves on $Z$ is derived equivalent to the endomorphism algebra of a tilting object, as has been explained in \cite{HV} and will be reviewed in Section~\ref{sec: general_discussion}. Here by a tilting object, we mean a perfect complex $T$  in $D^b(Z)$ such that $\Ext^i(T,T)=0$ for all $i\neq0$, and $T$ generates the entire derived category, in the sense that the smallest triangulated subcategory containing $T$ closed under direct summands is the entire derived category. In all cases we are interested in, the tilting object can always be chosen as a vector bundle, considered as a complex of coherent sheaves concentrated on degree 0. If this happens, the assumption that $T$ generates $D^b(Z)$ can be replaced by that $\Ext^\bullet(T,C)=0$ for some complex $C$ implies $C$ is exact (see, e.g. \cite{BH10}). But the fact that $Z$ is non-compact makes the endomorphism algebra more complicated to describe explicitly compared to the case when $Z$ proper.
\par
\par
Let $G$ be a reductive group and $P$ a parabolic subgroup. Let $p:Z\to G/P$ be an equivariant vector bundle. In this paper, we use the inverse image of the exceptional collections on  $G/P$ to $Z$, with special emphasis on the situation when the total space $Z$  is a commutative desingularization of an orbit closure in some representations of $G$. When the vector bundle is the commutative desingularization of generic determinantal variety defined by maximal minors, a non-commutative desingularization of this nature has been studied by R.~Buchweitz, G.~Leuschke, and M.~van den Bergh in \cite{BLV}. In the present paper, we describe the quiver with relations of  noncommutative
desingularizations of higher corank determinantal varieties, symmetric
determinantal varieties, and anti-symmetric determinantal varieties.
\par
The study of exceptional collections on homogeneous spaces was initiated by Beilinson and Kapranov who gave explicit constructions of exceptional collections on projective spaces and Grassmannians. Let $B_{u,v}$ be the set of partitions with no more than $u$ rows and $v$ columns. According to the result of Kapranov, for certain ordering on $B_{r,n-r}$, the set $$\{\SCH_\alpha\cQ\mid\alpha\in B_{r,n-r}\}$$ is an exceptional collection over $\Grass_{n-r}(E)$, the Grassmannian of $(n-r)$-planes in an $n$-dimensional vector space $E$, where $\SCH_\lambda$ is the Schur functor applied to the partition $\lambda$, and $\cQ$ is the tautological rank $r$ quotient bundle over $\Grass_{n-r}(E)$.
\par
Let $p:Z\to G/P$ be an equivariant vector bundle, and let $\Delta(G/P)=\{\Delta_\alpha\mid\alpha\in I\}$ be a full exceptional collection on $G/P$. On the total space $Z$, the inverse image $p^*(\oplus_\alpha\Delta_\alpha)$ can  fail to be a tilting object. Moreover, its endomorphism algebra $\Lambda:=\End_{\sO_Z}(p^*(\oplus_\alpha\Delta_\alpha))$ is usually infinite dimensional as a vector space.
\par
It is easy to show (Proposition~\ref{til-resol criterion}(1)) that for any tilting bundle $\TIL$ on $G/P$, (in particular for $\oplus_\alpha\Delta_\alpha$,) the inverse image $p^*\TIL$ is  a tilting bundle on $Z$ if the  $p^*\sEnd_{\sO_{G/P}}(\TIL)$ has no higher cohomology. Let $R$ be a normal integral domain endowed with a resolution of singularities $q:Z\to \Spec R$ by $Z$. 
In this case, $\Lambda:=\End_{\sO_Z}(p^*\TIL)$ has finite global dimension, and that  $R\Hom_{\sO_Z}(p^*\TIL,-)$ induces a derived equivalence $D^b(Coh(Z))\cong D^b(\Lambda\hbox{-}mod)$. 
Note that $\Lambda$ is coherent sheaf of algebras on $\Spec R$ with finite global dimension, which is generically a matrix algebra, which we call a non-commutative weak desingularization (Definition\ref{def: non-c desing}(1)).

The geometric technique of \cite[\S~5.1.5]{W03} can be easily applied to give a criterion for $q_*p^*\TIL$ and $\End_R(q_*p^*\TIL)$ to be maximal Cohen-Macaulay, which we recall in \S~\ref{subsec: geo tech}.
In \S~\ref{subsec:general NCCR} and \S~\ref{sec:depth}, we discuss intensively some sufficient conditions for $\Lambda$ to be reflexive. 
When $q_*p^*\TIL$ is maximal Cohen-Macaulay and $\Lambda$ is reflexive, we have $\Lambda\cong \End_R(q_*p^*\TIL)$ and hence according to Definition~\ref{def: non-c desing}(2) we say that $\Lambda$ is a non-commutative  desingularization.

We also develop a method to calculate $\Lambda$ explicitly. 
On a homogeneous space $G/P$, an exceptional collection often consists of equivariant vector bundles. In such case the objects in $D^b(Z)$ corresponding to the simples are obtained by pushing forward the dual exceptional collection (Definition~\ref{def: dual excep}) over $G/P$.
Let $\mathfrak{A}$ be a $k$-linear category. Recall that for an exceptional collection $\Delta$, the \textit{dual collection}  $\nabla = \{\nabla_\alpha \mid \alpha\in I \}$ is another subset of objects in
$D^b(\mathfrak{A})$, in bijection with $\Delta$, such that  $\Ext^\bullet (\nabla_\beta,\Delta_\alpha )= 0$ for $\beta > \alpha$,
and there exists an isomorphism $\nabla_\beta\cong \Delta_\beta$ mod
$D_{<\beta}$, where $D_{<\beta}$ is the full triangulated subcategory
generated by $\{\Delta_\alpha\mid \alpha<\beta \}$.
Let $\Lambda=\End_{\sO_{\cZ}}(p^*(\oplus_\alpha\Delta_\alpha))$ and $S_\beta=R\Hom_{\sO_{\cZ}}(p^*(\oplus_\alpha\Delta_\alpha),u_*\nabla_\beta)$, where $u:G/P\to Z$ is the zero section.
\begin{thm}[Theorem~\ref{thm: quiver_gen_rel}]\label{thm-intro-general}
Let $\Delta(G/P)=\{\Delta_\alpha\mid\alpha\in I\}$ be a full strongly exceptional collection consisting of equivariant sheaves over $G/P$ with the dual collection $\nabla (G/P)$. Assume the resolution $q:Z\to \Spec R$ is $G$-equivariant with $q^{-1}(0)=G/P$, and the only fixed closed point of $\Spec R$ is $0\in\Spec R$. Assume moreover that $p^*(\oplus_\alpha\Delta_\alpha)$ is a tilting bundle over $Z$, $\Lambda\cong\End_{\sO_{Z}}(p^*(\oplus_\alpha\Delta_\alpha))$, and $\End_{\sO_{Z}}(p^*(\oplus_\alpha\Delta_\alpha))$ is a non-commutative desingularization. Then,
\begin{enumerate}
  \item $S_\alpha$'s are equivariant simple objects in $\Lambda\hbox{-}mod$;
  \item a basis of the vector space $\Ext^1_\Lambda(S_\alpha,S_\beta)^*$ generates $\Lambda$ over $\oplus_\alpha k_\alpha$;
  \item with the generators for $\End_R(q_*p^*(\oplus_\alpha\Delta_\alpha))$ as above, $\Ext^2_{\sO_Z}(\nabla_i,\nabla_j)^*$ generates the relations.
\end{enumerate}
\end{thm}
In fact, the strongness assumption for $\Delta(G/P)$ is not essential. A more general statement can be find in Theorem~\ref{thm: quiver_gen_rel}.
\par
If $\Spec R$ is the closure of an orbit in a representation of $G$, and the commutative desingularization is  an equivariant vector bundle over $G/P$, the above theorem and the discussions proceeding it give a general approach to calculate a non-commutative desingularization of orbit closures. 

Exceptional collections over $G/P$ are difficult to find if the characteristic of the field is positive, and very few cases are known. However, in order to apply Theorem~\ref{thm-intro-general} to calculate the idempotents, generators, and relations of the non-commutative desingularization, we just need a set of objects satisfying much weaker condition than that of an exceptional collection. In the case $G/P$ is the Grassmannian, there is no known characteristic free exceptional collection. Nevertheless, in \cite{BLV3}, Buchweitz, Leuschke, and van den Bergh constructed a characteristic free tilting bundle over the Grassmannian.
Theorem~\ref{thm-intro-general}, along with Lemma~\ref{Lem: Exts}, can be applied to calculating idempotents, generators, and relations of the noncommutative desingularization obtained from the tilting bundle constructed in \cite{BLV3}, up to Morita equivalence.

In the present paper, we apply the general theory above to some special examples when $\Spec R$ is  the determinantal variety in the space of $n\times m$-matrices, the determinantal variety in the space of $n\times n$ symmetric matrices, or the pfaffian variety in  the space of $n\times n$ skew-symmetric matrices.  The ring $\Lambda$ is a non-commutative weak desingularization of $\Spec R$. 
We analyze the depth of $\Lambda$ as an $R$-module, and construct a non-commutative crepant resolution in some special cases.

More precisely, let $E$ be  a vector spaces over $k$ of dimension $n$, and let  $\Grass$ be the Grassmannian of $(n-r)$-planes in $E$. 
Let $0\to \cR\to E\times\Grass\to \cQ\to 0$ be the tautological
sequence over $\Grass$.
We consider the the total spaces $\cZ^s$ and $\cZ^a$  respectively \footnote{We use the upper script $s$ to remind us that we are in the symmetric matrix case. Similarly, later on we will use the upper script $a$ in the skew-symmetric matrix case.}  of the vector bundles $\Sym_2\cQ^*$ and $\wedge^2\cQ^*$ on $
\Grass$.
Let $H^s$  be the affine space $\Sym_2(E^*)$. Upon choosing a set of
basis for $E$, the
coordinate ring of $H^s$ can be identified with
$S^s=k[x_{ij}]_{i\leq j}$. 
Let $\Spec R^a$ be a  $GL_n$-orbit closure in $H^s$. 
Then $R^s$ is the quotient of $S^s$ by the ideal
generated by the $(r+1)\times (r+1)$ minors of the generic matrix $(x_{ij})$.
The total space $\cZ^s$ of  is a commutative desingularization of $\Spec R^s$ \cite[Chapter 6]{W03}.
Similarly, 
 let $\Spec R^a$ be a  $GL_n$-orbit closure in the space of $n\times n$ skew-symmetric matrices. Then the total space  $\cZ^a$ is a commutative desingularization of $\Spec R^a$ \cite[Chapter 6]{W03}.

Let $\TIL_0$ be the tilting bundle over
$\Grass_{n-r}(E)$ of Buchweitz, Leuschke, and van den Bergh  from 
\cite{BLV3}.
\begin{thm}[Theorem~\ref{sym tilt}, Proposition~\ref{lem:sym_max_ref}, Theorem~\ref{anti tilt}, Proposition~\ref{prop: MCM_skew}, and \S~\ref{subsec:main_example}]
Let $Z$ be either $\cZ^s$ or $\cZ^a$, and respectively let $R$ be either $R^s$ or $R^a$, so that $Z$ desingularizes $\Spec R$. 
\begin{enumerate}
  \item The bundle $p^*\TIL_0$ is a tilting bundle
over $Z$, i.e., $\Ext_{Z}^i(p'^*\TIL_0, p^*\TIL_0)=0$ for $i>0$ and $\Ext^\bullet(p^*\TIL_0, C)=0$ implies $C$ is an exact complex. In particular, $D^b(Coh(Z))\cong D^b(\End_{Z}(p^*\TIL_0)\hbox{-}mod)$.
  \item  The $R$ module $q_*p^*\TIL_0$ is maximal Cohen-Macaulay, and so is every direct summand of it.
  \item The ring $\Lambda=\End_{Z}(p^*\TIL_0$ 
  has finite global dimension, hence is a non-commutative weak desingularization of $\Spec R$. 
  \item The $R$-module $\End_{Z}(p^*\TIL_0)$ is torision free, and the map $\End_{Z}(p^*\TIL_0)\to \End_S(q_*p^*\TIL_K)$ is the embedding of $\End_{Z}(p'^*\TIL_K)$ into its reflexive envelope.
 Moreover,  $\End_{Z}(p^*\TIL_0)$ is not reflexive unless $R= R^s$ and $r=n-1$, in which case it is maximal Cohen-Macaulay over $R^s$ hence a non-commutative crepant resolution.
\end{enumerate}
\end{thm}

Nevertheless, in \S~\ref{sec:NCCR for r=1}, in the case when $R=R^s$ and $r=1$,  we modify the algebra $\Lambda=\End_{Z}(p^*\TIL_0$ to get a non-commutative crepant resolution of $\Spec R$.

Moreover, we use Theorem~\ref{thm-intro-general} to express the ring $\Lambda$ in terms of quiver with relations. 

Note that the vector space spanned by the set of arrows between any two vertices is naturally a representation of $G$. We find it convenient to introduce the language of equivariant quivers, as all the Ext's are naturally representations of $G$. For the precise definitions of equivariant quivers and their representations, see Section~\ref{sec: equi_quiver}. Very roughly, an \textit{equivariant quiver} is a triple $Q=(Q_0, Q_1,\alpha)$, where $(Q_0,Q_1)$ is a quiver and $\alpha$ is an assignment associating to each arrow $q\in Q_1$ a finite dimensional irreducible representation of $G$. For any equivariant quiver, there is an underlying usual quiver, upon choosing a basis for each representation associated to each arrow. The path algebra of an equivariant quiver is endowed with a natural rational $G$-action, so that we can consider the equivariant representations of it. While imposing relations on an equivariant quiver, we require the relations to be subrepresentations. In fact, if the equivariant quiver gives the endomorphism ring of a tilting object over $G/P$, the derived category of equivariant sheaves over $G/P$ is equivalent to the derived category of equivariant representations of the path algebra with relations of the equivariant quiver.

The drawback of our approach, compared to \cite{BLV}, is that in order to do explicit calculation we need to use the Borel-Weil-Bott Theorem, which dose not have a counterpart in positive characteristic. Nevertheless, up to Subsection~\ref{subsec: comb_Kap}, the results are characteristic free, unless otherwise specified. \textit{ We will write the field as $\mathbb{C}$ if its characteristic is zero, and $k$ if we do not assume anything on its characteristic.}

\par
Now we describe the equivariant quiver with relations for the non-commutative desingularization. We will use $C_{\alpha,\beta}^\gamma$ for the Littlewood-Richardson coefficient (see e.g., \cite{F97} for the definition). For a vector space $V$ and a non-negative integer $a$, we will denote $V^{\oplus a}$ simply by $aV$. For any two vertices $\alpha$ and $\beta$, the space of paths from $\beta$ to $\alpha$ will be denoted by $\Hom(\beta, \alpha)$.
\par

When $k=\CC$, let 
$\TIL_K=\oplus_{\alpha\in B_{r,n-r}}\SCH_\alpha\cQ^*$ be the Kapranov's
tilting bundle 
over $\Grass_{n-r}(E)$ by $p:Z\to\Grass$.
The algebras $\End_{Z}(p^*\TIL_0)$ and $\End_Z(p^*\TIL_K)$ are Morita equivalence. For simplicity, we consider $p^*\TIL_K$.

When $R=R^s$, the equivariant quiver with relations of the endomorphism ring $\End_S(q_*'p'^*\TIL_K)$ is given as follows (see Proposition~\ref{prop: sym_gen} and Proposition~\ref{prop: sym_rel}).
The set of vertices is indexed by $B_{r,n-r}$.
\begin{itemize}
  \item In the case $n-r=1$, the set of arrows from $\beta$ to $\alpha$ is given by $E$ if $C_{\beta,(1,0,\cdots,0)}^\alpha\neq 0$ or $C_{\alpha,(1,0,\cdots,0)}^\beta\neq 0$. No arrows otherwise. The relations are generated by the following representations in the space $\Hom(\beta, \alpha)$ $$(C^\alpha_{\beta,(1,1,0,\cdots,0)}\SCH_2E )\oplus(C^\beta_{\alpha,(1,1,0,\cdots,0)}\SCH_2E )\oplus(\delta^\beta_{\alpha}\wedge^2E ).$$
  \item In the case $n-r=2$, arrows from $\beta$ to $\alpha$ is given by $E$ if $C_{\beta,(1,0,\cdots,0)}^\alpha\neq 0$ and $\mathbb{C}$ if $C_{\alpha,(1,1,0,\cdots,0)}^\beta\neq 0$.  No arrows otherwise. The relations are generated by the following representations in the group $\Hom(\beta, \alpha)$ $$(C^{\beta^t}_{\alpha^t,(1,-1)}\wedge^2E )\oplus(C^{\beta^t}_{\alpha^t,(-1,-2)}E )\oplus(C^\beta_{\alpha,(1,1,0,\cdots,0)}\SCH_2E )\oplus( C^\beta_{\alpha,(2,0,\cdots,0)}\wedge^2E ).$$
  \item In the case $n-r\geq3$, The set of arrows has the same description as in the case $n-r=2$. The relations are generated by the following representations in the group $\Hom(\beta, \alpha)$ $$(C^{\beta^t}_{\alpha^t,(0,\cdots,0,-1,-1,-2)}\mathbb{C})\oplus(C^{\beta^t}_{\alpha^t,(1,0,\cdots,0,-1,-1)}E )\oplus(C^{\beta^t}_{\alpha^t,(2,0,\cdots,0)}\SCH_2E )\oplus(C^{\beta^t}_{\alpha^t,(1,1,0,\cdots,0)}\wedge^2E ).$$
\end{itemize}
\par

When $R=R^a$, the quiver with relations of the endomorphism ring is given as follows (see Proposition~\ref{prop: skew_gen} and Proposition~\ref{prop: skew_rel}).
 The set of vertices is indexed by $B_{r,n-r}$.
\begin{itemize}
  \item In the case $n-r=1$, arrows from $\beta$ to $\alpha$ is given by $E$ if $C_{\beta,(1,0,\cdots,0)}^\alpha\neq 0$  and $\mathbb{C}$ if $C_{\alpha,(1,1,0,\cdots,0)}^\beta\neq 0$.  No arrows otherwise. The relations are generated by the following representations in the group $\Hom(\beta, \alpha)$ $$(C_{\alpha^t,(1,0,\cdots,0)}^{\beta^t}\wedge^3E )\oplus(C_{\alpha^t,(2,0,\cdots,0)}^{\beta^t}\SCH_2E ).$$
  \item In the case $n-r=2$, there is one more arrow from $\beta$ to $\alpha$  given by $\mathbb{C}$ for $C^{\beta^t}_{\alpha^t,(1,-1,0,\dots,0)}$ besides the above ones.  No arrows otherwise. The relations are generated by the following representations in the group $\Hom(\beta, \alpha)$  $$(C_{\alpha^t,(0,\cdots,0,-1,-3)}^{\beta^t}\mathbb{C})\oplus(C_{\alpha^t,(1,0,\cdots,0,-2)}^{\beta^t}E )\oplus(C_{\alpha^t,(2,0,\cdots,0)}^{\beta^t}\SCH_2E )\oplus( C_{\alpha^t,(1,1,0,\cdots,0)}^{\beta^t}\wedge^2E ).$$
  \item In the case $n-r\geq3$, The set of arrows has the same description as in the case $n-r=1$.  The set of relations has the same description as in the case $n-r=2$.
\end{itemize}

\par
In the case of generic determinantal varieties, the quiver with relations is studied in \cite{BLV2}. We used the geometric technique in \cite{W03} to study the decomposition of the non-commutative desingularization as representations of the group in \S~\ref{subsec:min-present}.

In \S~\ref{sec: other_examples}, we apply the methods from the present paper to other examples of non-commutative desingularizations. 

\subsection*{Conventions} For an abelian category $\mathfrak{A}$, its derived category will be denoted by $D(\mathfrak{A})$, and its bounded derived category $D^b(\mathfrak{A})$. For a scheme  $X$, we denote the abelian category of
quasi-coherent sheaves over it by $Qcoh(X)$, and the abelian category of coherent sheaves $Coh(X)$. For short, $D^b(X)=D^b(Qcoh(X))$.
For a ring $A$, commutative or not, $A\hbox{-}Mod$ will be used to denote the abelian category of (left) $A$-modules, and $A\hbox{-}mod$ the abelian category of finitely generated (left) $A$-modules.
For a partition $\lambda$, we denote by $\SCH_\lambda$ the corresponding Schur functor. We will identify partitions and Young diagrams and following the conventions in \cite{F97}. For a vector bundle $\sF$ over $X$, we denote its dual bundle $\sHom_X(\sF,\sO)$ by $\sF^*$. Similarly, for a vector space $E$, its dual space is denoted by $E^*$. For any $R$-module $M$, when we talk about the algebra structure of $\End_R(M)$, we consider the opposite multiplication. Similar for $\sEnd_{\sO_X}(\sF)$ for any quasicoherent sheaf $\sF$ on $X$.

\subsection*{Acknowledgments} Many ideas in this paper are inspired by \cite{BLV}. We are grateful to Graham Leuschke for carefully reading an earlier version of this paper and providing a list of comments, and to Michel van den Bergh for pointing out an error in an earlier version. The second named author would like to thank Roman Bezrukavnikov, Zongzhu Lin, and Yi Zhu for helps with questions related to this paper. He is also grateful to Yaping Yang for innumerable times of fruitful discussions.

\section{Constructions of non-commutative desingularizations}
\label{sec: general_discussion}
In this section we recall some basic notions about noncommutative desingularization. Then, we introduce a construction that in many interesting situations provides noncommutative desingularizations. A criterion is given to test whether this construction gives a noncommutative desingularization or not.
\subsection{General notions related to tilting bundles}
For a scheme $X$, a vector bundle (or more generally a perfect
complex) $\TIL$ is called a \textit{tilting bundle} (resp. tilting complex) over $X$ if it satisfies
the following:
\begin{enumerate}
  \item $\TIL^{\perp}=0$ in $D^b(X)$, i.e., for any complex $M$, $\Hom_{D^b(X)}(\TIL,
M[i])=0$ for all $i$ implies $M=0$, where $[1]$ is the
shifting  functor;
  \item $\Ext^i(\TIL, \TIL)=0$ for $i>0$.
\end{enumerate}
\begin{theorem}[7.6 in \cite{HV}]
\label{finite dim}
For $X$ a projective scheme over a Noetherian affine scheme of finite type, and $\TIL\in D(Qcoh(X))$ a tilting object. We have the following
\begin{enumerate}
  \item $R\Hom_{\sO_X}(\TIL,-)$ induces an equivalence $$D(Qcoh(X))\cong D(\End_{\sO_X}(\TIL)\hbox{-}Mod).$$
  \item This equivalence restricts to an equivalence $$D^b(Coh(X))\cong D^b(\End_{\sO_X}(\TIL)\hbox{-}mod).$$
  \item If $X$ is smooth then $\End_{\sO_X}(\TIL)$ has finite global dimension.
\end{enumerate}
\end{theorem}
\par

\begin{definition}[\S~5 in \cite{BO}]\label{def: non-c desing}\label{def: nccr}
\begin{enumerate}
\item For an algebraic variety $X$, a \textit{non-commutative  weak desingularization}
of $X$ is a coherent sheaf of algebras $\sA$ on $X$, which is generically a matrix algebra, such that the abelian category of sheaves of right modules $\sA$ has
finite homological dimension.
\item We say $\sA$ is a  \textit{non-commutative   desingularization} if furthermore $\sA\cong\sEnd_{\sO_X}(\sF)$ for some reflexive cohernt sheaf $\sF$.
\item When $X$ is Cohen-Macaulay, we say $\sA$ is a  \textit{non-commutative   crepant resolution} (NCCR) if furthermore $\sA$ is maximal
Cohen-Macaulay (MCM).
\end{enumerate}
\end{definition}
\begin{remark}
\begin{enumerate}
\item In the original definition of \cite{BO}, the sheaf $\sF$ is only required to be torsion free. Also, in the original definition of \cite{VDB2}, an NCCR is considered only when $X$ is Gorenstein.
\item The analogue of $\pi^*$ is $-\otimes_{\sO_X}\sA$; the analogue of $\pi_*$ is restriction of scalars from $\sA$ to $\sO$; the analogue of $\pi^!$ is $R\Hom_{X}(\sA,-)$.
\item We adapt a notion of non-commutative  weak desingularization here, while the analogue of the key properties of a commutative resolution, birationallity (generically a matrix algebra), properness (coherent), and smoothness (finite global dimension) are still maintained.
\end{enumerate}
 
\end{remark}
\par
We following lemmas are well-known. We include the proofs for completeness.
\begin{lemma}
\label{orth} Assume $p:X\to Y$ is the projection of a vector bundle
over an algebraic scheme $Y$, $\mathscr{T}$ is a quasi-coherent sheaf over $Y$
such that $\mathscr{T}^\perp=0$ in $D^b(Y)$, then, $$(p^*\mathscr{T})^\perp=0$$ in
$D^b(X)$.
\end{lemma}
\begin{proof}
The projection $p$ is an affine morphism. Denote the corresponding
sheaf of algebras of the affine morphism $p$ by $\mathscr{A}$. The
push-forward functor $p_*$ induces an equivalence between $Qcoh(X)$
and the subcategory $\mathscr{A}\hbox{-}mod$ of $Qcoh(Y)$. In particular,
$p_*$ is exact.
\par
The sheaf $\mathscr{A}=\Sym(X^*)$ is locally free. Therefore, $p^*$
is also exact.
\par
We have a pair of adjoint exact functors $(p^*, p_*)$ between
$Qcoh(X)$ and $Qcoh(Y)$. By a standard result in homological
algebra  \cite[\uppercase\expandafter{\romannumeral3}.6]{GM},
they induce an adjoint pair between $D^b(X)$ and $D^b(Y)$, which we
still denote by $(p^*, p_*)$. They commute with the shifting functor $[1]$.
\par
For a complex $C$ over $X$, assume $\Hom_{D^b(X)}(p^*\mathscr{T},
C[i])=0$ for all $i$. Using the adjoint property, we get
$\Hom_{D^b(Y)}(\mathscr{T},p_*C[i])=\Hom_{D^b(X)}(p^*\mathscr{T}, C[i])=0$.
By assumption, $p_*C\cong0$ in $D^b(Y)$. This implies $p_*C\cong0$
in $D^b(\mathscr{A})$, hence, $C\cong0$ in $D^b(X)$.
\end{proof}
\par
\begin{lemma}
\label{ext=0} Assume $p:X\to Y$ is the projection of a vector bundle
over an algebraic scheme $Y$, $\mathscr{T}_1$ and $\mathscr{T}_2$
are two vector bundles over $Y$ such that $H^i(X,
p^*\sHom(\mathscr{T}_1, \mathscr{T}_2))=0$ for all $i>0$, then
$\Ext^i(p^*\mathscr{T}_1, p^*\mathscr{T}_2)=0$ for all $i>0$.
\end{lemma}
\begin{proof}
Denote the corresponding sheaf of algebras of the affine morphism
$p$ by $\mathscr{A}$. We have the local-global spectral sequence $E$ with
$$E^{ij}_2=H^i(X,\sExt^j_X(p^*\mathscr{T}_1, p^*\mathscr{T}_2))$$
which converges to $\Ext^{i+j}(p^*\mathscr{T}_1, p^*\mathscr{T}_2)$.
\par
Since $p^*\mathscr{T}_1$ and $p^*\mathscr{T}_2$ are both locally
free, $\sExt^j_X(p^*\mathscr{T}_1, p^*\mathscr{T}_2)=0$ for all
positive $j$. We have $$\Ext^{i}(p^*\mathscr{T}_1,
p^*\mathscr{T}_2)=H^i(X,\sHom_{\sO_X}(p^*\mathscr{T}_1,
p^*\mathscr{T}_2)).$$
\par
 Since $\mathscr{A}$ is locally free, we can identify
$$p^*\sHom_{\sO_Y}(\mathscr{T}_1, \mathscr{T}_2)\cong\sHom_{\sO_Y}(\mathscr{T}_1, \mathscr{T}_2\otimes_{\sO_Y}\sA)\cong\sHom_{\sO_X}(p^*\mathscr{T}_1, p^*\mathscr{T}_2).$$ Use the assumption that $H^i(X,
p^*\sHom_{\sO_Y}(\mathscr{T}_1, \mathscr{T}_2))=0$, then,
$\Ext^{i}(p^*\mathscr{T}_1, p^*\mathscr{T}_2)=0$ for $i>0$.
\end{proof}
\par
\subsection{Inverse image of tilting bundles and noncommutative desingularization}
\label{subsec: inverse image til}\label{subsec:general NCCR} Let $V$ be a smooth projective variety, and let $\mathbb A^N$ be an affine space.
Suppose we
have a tilting bundle $\TIL_{V}$ over 
$V$, and suppose the total space $Z$ of a vector subbundle of $\mathbb A^N\times V$ over
$V$ desingularizes an
affine subvariety $\Spec R\subseteq \mathbb A^N$, i.e., the
projection $\mathbb A^N\times V\to \mathbb A^N$ restricted to $Z$ is birational
onto $\Spec R$, so that the restriction is a desingularization. We seek the conditions for the inverse image of
$\TIL_{V}$ to be a tilting bundle over $Z$,  and to give a
non-commutative crepant resolution of $\Spec R$.
\begin{equation}\label{diag: general}
\xymatrixrowsep{13pt}\xymatrix{
Z\ar[dd]_{q'} \ar@/^/[drr]^{p'} \ar@{_{(}->}[dr]^j\\
 & V\times \mathbb A^N\ar[d]_q \ar[r]^p&  V \\
\Spec R\ar@{^{(}->}[r]^i & \mathbb A^N}
\end{equation}
\par

We get the following method to construct crepant noncommutative desingularizations.
\begin{prop}
\label{til-resol criterion}
Notation as in diagram~(\ref{diag: general}). Let  $\TIL$ be a tilting bundle over $V$.
\begin{enumerate}
  \item If $H^i(Z,p'^*\sEnd_{\sO_{V}}(\TIL))=0$ for all positive $i$, then $p'^*\TIL$ is a tilting bundle over $Z$. In particular, $\End_{\sO_Z}(p'^*\TIL)$ is a non-commutative weak desingularization of $\Spec R$.  
\item If moreover $\End_{\sO_Z}(p'^*\TIL)\cong \End_R(q'_*p'^*\TIL)$ and $q'_*p'^*\TIL$ is reflexive, then $\End_{\sO_Z}(p'^*\TIL)$ is a non-commutative desingularization.
\item If moreover, $\End_{\sO_Z}(p'^*\TIL)$ is maximal Cohen-Macaulay, then $\End_R(q'_*p'^*\TIL)$ gives a noncommutative crepant desingularization of  $\Spec R$.
\end{enumerate}
\end{prop}
\begin{proof}
Note that if $p'^*\TIL$ is a tilting bundle, then $\End_{\sO_Z}(p'^*\TIL)$ has finite global dimension by Theorem~\ref{finite dim}, since $Z$ is smooth. The first part of this proposition follows directly from Lemma~\ref{orth}, Lemma~\ref{ext=0},  and the definition of tilting bundles.
\par
The second and third part follows from the first part and the definitions.
\end{proof}

Now we discuss the conditions $\End_{\sO_Z}(p'^*\TIL)\cong \End_R(q'_*p'^*\TIL)$ and the Cohen-Macaulayness of $\sEnd_{\sO_Z}(p'^*\TIL))$ from Proposition~\ref{til-resol criterion}. We have the following criterion for the latter. Although it is known (see e.g., \cite{BLV2}), we include the proof for completeness.
\begin{lemma}\label{lem:mcm}
Assume $Z$ is the total space of the vector bundle on $V$, and $\End_{\sO_Z}(p'^*\TIL)\cong \End_R(q'_*p'^*\TIL)$ with $H^i(Z,\sEnd_{\sO_Z}(p'^*\TIL))=0$. Then, $\End_R(q'_*p'^*\TIL)$ is maximal Cohen-Macaulay iff $H^i(V,\TIL^*\otimes \TIL\otimes \Sym(Z^*)\otimes\omega_Z)=0$ for any $i>0$, where $\omega$ is the dualizing complex.
\end{lemma}
\begin{proof}
By definition, $\End_{\sO_Z}(p'^*\TIL)$ is maximal Cohen-Macaulay iff $\Ext_R^i(\End_{\sO_Z}(p'^*\TIL),\omega_R)=0$ for all $i>0$. As $q$ is a proper and $Rq_*\sEnd_{\sO_Z}(p'^*\TIL)\cong \End_{\sO_Z}(p'^*\TIL)$, using the Grothendieck duality for proper morphisms  \cite[1.2.22]{W03} $$\Ext_R^i(\End_{\sO_Z}(p'^*\TIL),\omega_R)
\cong\Ext^i_{\sO_Z}(\sEnd_{\sO_Z}(p'^*\TIL),\omega_Z).$$
We know $\Ext_R^i(\End_{\sO_Z}(p'^*\TIL),\omega_R)=0$ is equivalent to the following
\begin{eqnarray*}
\Ext_R^i(\End_{\sO_Z}(p'^*\TIL),\omega_R)
&\cong&\Ext^i_{\sO_Z}(\sEnd{\sO_Z}(p'^*\TIL),\omega_Z)\\
&\cong& H^i(Z,\sEnd_{\sO_Z}(p'^*\TIL)\otimes\omega_Z)\\
&\cong& H^i(V,\TIL^*\otimes \TIL\otimes \Sym(Z^*)\otimes\omega_Z)\\
&\cong& 0.
\end{eqnarray*}
\end{proof}

Now we discuss the natural map $\End_{\sO_Z}(\TIL)\to\End_R(\TIL(Z))$.
\begin{prop}\label{lem:mcm_nccr}
Let $X=\Spec(R)$ be an affine normal  scheme and $f:Z\to
X$ a crepant resolution with exceptional locus having codimension at
least 2 in $X$. Let $\TIL$ be a tilting bundle over $Z$ such that $\TIL(Z)$ is a reflexive $R$-module. Then $\End_{\sO_Z}(\TIL)\to\End_R(\TIL(Z))$ is an isomorphism iff
$\End_{\sO_Z}(\TIL)$ is reflexive. 
\end{prop}
\begin{proof}
First recall that (7.4.2 of \cite{Bou}) for a finitely generated module $M$ over a commutative noetherian domain, $M$ being reflexive is equivalent to being torsion free and $M=\bigcap M_p$ where the intersection is taken over all codimension 1 primes.

\par
There is a natural morphism of rings $\End_{\sO_Z}(\TIL)\to \End_R(\TIL(Z))$. The target is reflexive.  This morphism is an isomorphism outside the exceptional locus of $f$ which has codimension at least 2 in $X$. If the source is also reflexive, then we have $\End_{\sO_Z}(\TIL)=\bigcap\End_{\sO_Z}(\TIL)_p\cong \bigcap\End_R(\TIL(Z))_p=\End_R(\TIL(Z))$, where the intersection is taken over all codimension 1 primes.
\end{proof}

\begin{prop}\label{prop:ref}
If the exceptional locus of $q':Z\to \Spec R$ has codimension at least two in both $Z$ and $\Spec R$, and $R$ is an integral domain, then both $\TIL(Z)$ and $\End_{\sO_Z}(p'^*\TIL)$ are reflexive; in particular, $\End_{\sO_Z}(p'^*\TIL)$ is a non-commutative desingularization.
\end{prop}
\begin{proof}
Now we prove the third part. First note that by Lemma~4.2.1 in \cite{VDB1}, if the exceptional locus of $q'$ has  codimension at least 2 in both $Z$ and $\Spec R$, then $q_*'$ sends any reflexive sheaves to reflexive sheaves. Both $Z$ and $\Spec R$ are integral schemes, hence for any reflexive sheaf $N$ over one or the other, $\sHom(M,N)$ is also reflexive for any $M$. This is because $\sHom(M,N)$ is torsion-free whenever $N$ is, and both intersection and localization commute with $\sHom(M,-)$, given integrality of the scheme. Thus, we know $p'^*\TIL$ is reflexive, so are $\sEnd_{\sO_Z}(p'^*\TIL)$, $q_*'\sEnd_{\sO_Z}(p'^*\TIL)$, $q_*'p'^*\TIL$, and $\End_R(q'_*p'^*\TIL)$. 
\end{proof}

However, when the map $q':Z\to \Spec R$ has an exceptional divisor in $Z$, the depth of the module $\sEnd_{\sO_Z}(p'^*\TIL)$ is complicated. We analyze this case in \S~\ref{sec:depth}.

\subsection{A geometric technique}\label{subsec: geo tech}
Now we review a geometric technique used throughout this paper. The
reference for this subsection is \cite{W03}.
\par
For a projective variety $V$, $\mathbb{A}^{N}_k$ an affine space, the space $\mathbb A^N \times V$ can be
viewed as the total space of the trivial vector bundle $\mathcal{E}$ of
dimension $N$ over $V$. Let us consider the subvariety $Z$ in $\mathbb A^N
\times V$ which is the total space of a subbundle $\mathcal{S}$ in
$\mathcal{E}$. We denote by $q$ the projection $q : \mathbb A^N \times V \to
\mathbb A^N$ and by $q'$ the restriction of $q$ to $Z$. Let $X = q (Z)$. We
get the basic diagram
$$
\xymatrixrowsep{15pt}
\xymatrix{
Z\ar[d]_{q'}\ar@{^{(}->}[r] &  H\times V\ar[d]_q\\
X\ar@{^{(}->}[r] &  H.\\
}$$

\par
The projection from $\mathbb A^N \times V$ onto $V$ is denoted by $p$, and the
quotient bundle $\mathcal{E}/\mathcal{S}$ by $\mathcal{T}$ . Thus we
have the exact sequence of vector bundles on $V$,
$$0 \to \mathcal{S} \to \mathcal{E} \to \mathcal{T} \to 0.$$

The coordinate ring of $\mathbb A^N$ will be denoted by $A$. It is a
polynomial ring in $N$ variables over $k$. We will identify 
sheaves on $\mathbb A^N$ with $A$-modules. The direct image $p_*(\sO_Z)$ can
be identified with the sheaf of algebras $\Sym(\eta)$, where
$\eta=\mathcal{S}^*$. For a vector bundle $\mathcal{V}$ over $V$,
the $\sO_{\mathbb A^N\times V}$-module $\sO_Z\otimes p^*\mathcal{V}$ will be
denoted by $M(\mathcal{V})$.
\begin{theorem}[5.1.2 in \cite{W03}]
\label{5.1.2 basic} For a vector bundle $\mathcal{V}$ on $V$, we
define free graded $A$-modules
$$F(\mathcal{V})_{i} = \bigoplus_{j \geq 0} H^{j}(V, \bigwedge^{i+j} \xi \otimes \mathcal{V}) \otimes_{k} A
(-i-j)$$ where $\xi = \mathcal{T}^* $ and $(i)$ means shifting by
$i$.
\begin{enumerate}
  \item  There exist minimal differentials $$ d_{i} (\mathcal{V}):= F(\mathcal{V})_{i} \to F(\mathcal{V})_{i-1}$$
         of degree $0$ such that $F(\mathcal{V})_{\bullet}$ is a complex of free graded $A$-modules
         with $$H_{-i}(F(\mathcal{V})_{\bullet})=\mathcal{R}^{i} q_{*} M(\mathcal{V}).$$
         In particular, the complex $F(\mathcal{V})_{\bullet}$ is exact in positive degrees.
  \item  The sheaf \ $ \mathcal{R}^{i} q_{*} M(\mathcal{V})$ is equal to $H^{i}(Z, M(\mathcal{V}))$ and it can be also identified with the graded $A$-module $H^{i}(V, \Sym(\eta) \otimes \mathcal{V})$.
  \item  If $\phi: M(\mathcal{V}) \to M(\mathcal{V}^{'})(n)$ is a morphism of graded sheaves, then there exists a morphism of complexes
     $$f_{\bullet}(\phi): F(\mathcal{V})_{\bullet} \to F(\mathcal{V}^{'})_{\bullet}(n)$$
     Its induced map $H_{-i}(f_{\bullet}(\phi))$ can be identified with the induced map
     $$H^{i}(Z, M(\mathcal{V})) \to H^{i}(Z, M(\mathcal{V}^{'}))(n).$$
\end{enumerate}

\end{theorem}

This theorem will be mentioned as the basic theorem of geometric
method in this paper.
\par
Now we come to a criterion for maximal Cohen-Macaulayness in the context of geometric technique. The proof, which can be found in 5.1.5 of \cite{W03}, is based on the basic theorem of geometric technique and Lemma~\ref{lem:mcm}.
\begin{prop}\label{thm:dual_mcm}
Let $\mathcal{V}$ be a bundle over $V$, and let $\check{V}:=\omega_V\otimes\wedge^{top}\xi\otimes\mathcal{V}^*$. Assume $\dim Z=\dim X$ and $R^iq_*'(\sO_z\otimes p^*\mathcal{V})=0$ for all $i>0$. Then, $R^0q_*'(\sO_z\otimes p^*\mathcal{V})$ is a maximal Cohen-Macaulay module supported on $X$ iff $R^iq_*'(\sO_z\otimes p^*\check{\mathcal{V}})=0$ for all $i>0$.
\end{prop}
This proposition will be used in Section~\ref{sec: other_examples} to prove that some non-commutative desingularizations we study are crepant.

\section{Grassmannians}\label{sec: grass}
In this section we collect some preliminary results about coherent sheaves on the Grassmannians. It can also be viewed as a collection of examples to the notions of exceptional collection, tilting bundle, and quasi-hereditary structure we review in \S~\ref{sec: general_discussion} without providing examples.
\subsection{Representation theory of $GL_n$}
We recall here some classical results about representation theory of reductive group $G$ over $k$, with emphasis on the case when $G=GL_n$. For the proofs of the results, we refer to \cite{J}.

Let $G$ be a split reductive group and $T$ a split maximal torus whose weight lattice will be denoted by $\hat{T}$. We fix a Borel subgroup $B$ containing $T$. The dominant camber in $\hat{T}$ determined by $B$ will be denoted by $\hat{T}_{+}$. The category of finite dimensional representations of $G$ will be denoted by $\Rep(G)$.

For a subgroup $H$ of $G$ such that $G/H$ is a scheme, the category of $G$-equivariant coherent sheaves on $G/H$ will be denoted by $Coh_G(G/H)$. We can define a functor $\mathfrak{L}_{G/H}:\Rep(H)\to Coh_G(G/H)$, $V\mapsto G\times_HV$. It is an equivalence of categories, with quasi-inverse given by $\cE\mapsto \cE_{[H]}$, i.e., taking the fiber at the point $[H]\in G/H$. In particular, both functors are exact functors.

For a $G$-scheme $X$, i.e., a scheme $X$ with an algebraic $G$-action, we have an adjoint pair $H^0(X,-):Coh_G(X)\to \Rep(G)$ and $-\otimes\sO_X:\Rep(G)\to Coh_G(X)$. In general, $H^0(X,-)$ is only left exact and $-\otimes\sO_X$ is only right exact. In the case when $X=G/H$, the composition $V\mapsto (V\otimes\sO_X)_{[H]}$ is naturally equivalent to $\Res_H^G:\Rep(G)\to \Rep(H)$, therefore, its adjoint $\Ind_H^G:\Rep(H)\to \Rep(G)$ is the composition $H^0(G/H,-)\circ\mathfrak{L}_{G/H}$.
These functors are summarized in the following diagram.
$$
\xymatrixcolsep{100pt}
\xymatrix@!0{
\Rep(H)\ar@/^/@<1ex>[r]^{\mathfrak{L}_{G/H}}\ar@/^2pc/@<3ex>[rr]^{\Ind_H^G} & Coh_G(G/H)\ar@/^/@<1ex>[l]^{[H]}\ar@/^/@<1ex>[r]^{H^0(G/H,-)}& \Rep(G)\ar@/^/@<1ex>[l]^{-\otimes\sO_{G/H}}\ar@/^2pc/@<3ex>[ll]^{\Res_H^G}
}$$

As the functor $\mathfrak{L}_{G/H}$ is exact, we have the Borel-Weil Theorem.
\begin{theorem}
We have a natural isomorphism of functors $R^i\Ind^G_H\cong H^i(G/H, \mathfrak{L}_{G/H}-)$.
\end{theorem}

Now we take $H=B$. Note that $\Rep(B)\cong \Rep(T)$, i.e., irreducible representations of $B$ are indexed by weights $\alpha\in \hat{T}$.
The following vanishing theorem is originally due to Kempf.
\begin{theorem}\label{thm:Kempfvan}
If $\alpha\in \hat{T}_+$, then $$R^i\Ind_B^G\alpha=H^i(G/B, \mathfrak{L}_{G/H}\alpha)=0$$ for $i>0$.
\end{theorem}

Now consider an $n$-dimensional vector space $V$. Let $G=\GL(V^*)$, for $\lambda\in \hat{T}_{+}$, the induced representation $\Ind_B^G(\lambda)$ is called the Schur module corresponds to the dominant weight $\lambda$, denoted by $\SCH_\lambda V$. Let $w_0$ be the longest element in the Weyl group, the representation $\Ind_B^G(-w_0(\lambda))$ is called the Weyl module, denoted by $\WEY_\lambda V$.

If $k=\CC$, these two modules coincide, and are both simple representations. For general $k$, $\SCH_\lambda V$ has a simple socle which coincides with the simple top of $\WEY_\lambda V$. Taking $\WEY$ as the standard objects and $\SCH$ as the costandard objects gives $\Rep(G)$ a quasi-hereditary structure.

The following theorem tells us the tensor product of any two standard objects has a filtration by standard objects, and the multiplicity of each factor can be calculated.
\begin{theorem}[See \cite{W03}, 2.3.2, 2.3.4]\label{thm: Litt_Rich}
Let $V$ and $W$ be two vector spaces over $k$, and $\alpha$, $\beta$ be any two dominant integral weights.
\begin{enumerate}
  \item There is a natural filtration on $\Sym_t(V\otimes W)$ whose associated graded object is a direct
sum with summands tensor products $\SCH_\delta V\otimes\SCH_\alpha W$ of Schur functors.
  \item There is a natural filtration on $\SCH_\alpha V \otimes \SCH_\beta V$ whose associated graded object is a direct sum of Schur functors $\SCH_\delta V$. The multiplicities can be computed using the usual Littlewood-Richardson rule.
\end{enumerate}
\end{theorem}

We identify $GL_n$ with the space of all the invertible $n\times n$ matrices, and $T$ as the diagonal invertible matrices. Then $\hat{T}$ can be identified with the set of all $n$-tuples of integers $\mathbb{Z}^n$, and $\hat{T}_{+}$ consists of $n$-tuples of non-increasing integers. Furthermore, we  identify $n$-tuples of non-increasing non-negative integers with partitions of lengths no more than $n$, which in turn are identified with Young diagrams following the conventions in \cite{F97}.  For a Young diagram $\alpha$, we will also be considering its transpose $\alpha'$ or $\alpha^t$.

For a vector space $E$ of dimension $n$ and a partition $\lambda$ of length no more than $n$, upon choosing a basis for $E^*$, we can identity $GL(E^*)$ with $GL_n$ using this basis, then $\SCH_\lambda E$ and $\WEY_\lambda E$ can be define as above. This procedure behaves well with respect to change of basis, hence can be thought of as functors. For explicit descriptions of the Schur functor $\SCH_\lambda$ and Weyl functor $\WEY_\lambda$ associated to a partition $\lambda$, see e.g., \cite{F97} or \cite{W03}. In particular, they can be applied to vector bundles $h: \mathcal{E}\to X$.

Now we consider the case when $P$ is a maximal parabolic subgroup, i.e., when $G/P$ is  $\Grass=\Grass(n-r, n)$, the Grassmannian of $n-r$-planes in the $n$-dimensional vector space $E^*$ (or equivalently the Grassmannian of $r$-planes in $E$).
Let $0\to\cR\to \sO\otimes E^* \to\cQ\to 0$ be the tautological sequence on $\Grass$.
In this case, the Levi subgroup $L=\GL_{n-r}\times\GL_r$. For the defining representation $k^{n-r}$ of $\GL_{n-r}$, the sheaf $\mathfrak{L}_{\Grass}(k^{n-r})$ is $\cR^*$, and similarly for the defining representation $k^r$ of $\GL_r$, the sheaf $\mathfrak{L}_{\Grass}(k^r)$ is $\cQ^*$.

A weight $\alpha$ is called $(n-r)$-dominant if it is dominant as a weight of $GL_{n-r}\times GL_r$, i.e., $\alpha_1\geq\cdots\geq\alpha_{n-r}$ and
$\alpha_{n-r+1}\geq\cdots\geq\alpha_n$.
For an $(n-r)$-dominant $\alpha$ weight we can consider two weights $\beta =
(\alpha_1,\cdots,\alpha_{n-r})$ and $\gamma =
(\alpha_{n-r+1},\cdots,\alpha_n$). We define the vector bundle
$\mathcal{V}(\alpha) = \SCH_\beta\cR^{*}\otimes \SCH_\gamma\cQ^{*}$.

The following is an easy corollary of the Kempf Vanishing Theorem \ref{thm:Kempfvan}.
\begin{corollary}
Consider
the integral dominant weight $\alpha$ and the corresponding vector bundle $\mathcal{V}(\alpha)$ on $\Grass(n-r,E^*)$ defined as above. Then
$H^0(\Grass,\mathcal{V}(\alpha)) = \SCH_{\bar\alpha}E\otimes(\wedge^nE)^{\otimes\alpha_n}
$
and $H^i(\Grass,\mathcal{V}(\alpha))= 0$ for $i > 0$.
\end{corollary}

\par
The symmetric group $\Sigma_n$ acts on the set of weights. Let
$\alpha=(\alpha_1,\cdots,\alpha_n)$. The permutation $\sigma_i = (i,
i + 1)$ acts on the set of weights by: $\sigma_{i\cdot} \alpha =
(\alpha_1,\cdots, \alpha_{i-1}, \alpha_{i+1}-1, \alpha_i + 1,
\alpha_{i+2},\cdots, \alpha_n)$.
\par
Let $\alpha\in \Z^n$
be a dominant integral weight. We denote $\bar\alpha = (\alpha_1 -\alpha_n, \dots , \alpha_{n-1}-\alpha_n, 0)$.
By definition the weight $\bar\alpha$ is a partition.

If $k=\CC$, we know more about cohomologies of bundles corresponding to non-dominant weights. Note that part of it is the Borel-Weil Theorem which holds in arbitrary characteristic.
\begin{theorem}[Bott]
\label{Bott} Let $k=\CC$. We consider a weight $\alpha$ satisfying
$\alpha_{i}\geq \alpha_{i+1}$ for $i \neq n-r$ and the corresponding
vector bundle $\mathcal{V}(\alpha)$ over $\Grass$
defined above. Then one of the two mutually exclusive possibilities
occurs:
\begin{enumerate}
  \item There exists an element $\sigma \in \Sigma_{n}, \sigma \neq 1$, such that $\sigma (\alpha)= \alpha$. Then the higher
direct images $H^i(\Grass,\mathcal{V}(\alpha))$ are zero for
$i \geq 0$.
  \item There exists a unique element $\sigma \in \Sigma_{n}$ such that $\sigma (\alpha):= (\beta)$ is a partition (i.e.
is non-increasing). In this case all higher direct images
$H^i(\Grass,\mathcal{V}(\alpha))$ are zero for $i \neq
l(\sigma)$, and
$$H^{l(\sigma)}(\Grass,\mathcal{V}(\alpha))=\SCH_\beta E.$$
\end{enumerate}
\end{theorem}
\par

\subsection{Exceptional collections on Grassmannians}
The main reference for this section is \cite{BLV3}.

Let $\catA$ be a $k$-linear abelian category.
\begin{definition}
An ordered set of objects $\Delta=\{\Delta_\alpha,\alpha\in I\}$ in
$D^b(\mathfrak{A})$ is called \textit{exceptional} if we have
$\Ext^\bullet(\Delta_\alpha,\Delta_\beta ) = 0$ for $\alpha < \beta$ and $\End(\Delta_\alpha)
= k$. An exceptional set is called \textit{strongly exceptional} if in addition
$\Ext^n(\Delta_\alpha,\Delta_\beta) = 0$ for $n\neq 0$. It is said to be \textit{full} if it is spanning for
$D^b(\mathfrak{A})$.
\end{definition}

In particular, in the definition above, we take $\mathfrak{A}$ to be the
category of coherent sheaves over some scheme. It is not hard to see
that if we have a finite full strongly exceptional set $\Delta = \{\Delta_\alpha,
\alpha\in I\}$ consisting of vector bundles, then $T=\oplus_\alpha\Delta_\alpha$ is
 a tilting bundle.
\par
\begin{definition}\label{def: dual excep}
For an exceptional collection $\Delta$, let $\nabla = \{\nabla_\alpha , \alpha\in I \}$ be another subset of objects in
$D^b(\mathfrak{A})$, in bijection with $\Delta$. We say that $\nabla$ is the \textit{dual collection}
to $\Delta$ if $\Ext^\bullet (\nabla_\beta,\Delta_\alpha )= 0$ for $\beta > \alpha$,
and there exists an isomorphism $\nabla_\beta\cong \Delta_\beta$ mod
$D_{<\beta}$, where $D_{<\beta}$ is the full triangulated subcategory
generated by $\{\Delta_\alpha\mid \alpha < \beta\}$.
\end{definition}
\par
There are some well-known facts (see e.g., \cite{Bez06}):
\begin{enumerate}
  \item $\Ext^\bullet(\nabla_\alpha ,\Delta_\alpha) =
k$, where $k$ lies in homological degree 0;
  \item $\Ext^\bullet(\nabla_\alpha ,\Delta_\beta ) = 0$ for $\alpha \neq \beta$;
  \item the dual collection is unique if it exists at all.
\end{enumerate}

Note that even if $\Delta$ is a strongly exceptional collection, its dual collection $\nabla$ might not be strong.
\par


Even if an exceptional collection is not strongly exceptional, under milder assumptions, we still have a tilting object. For an example of the general discussion in this section, we refer to Section~\ref{sec: grass}.

\label{subsec: quasi-hered}

Let $\catA$ be an abelian artinian $k$-linear category with a fixed complete set of pairwise distinct simple objects $\{S_\lambda\mid\lambda\in I\}$. Assume $\Lambda$ is finite for simplicity. Let $P_\lambda\to S_\lambda$ be the projective cover and $S_\lambda\to Q_\lambda$ be the injective envelope for each $\lambda\in I$. Endow $I$ with a partial order. We define the standard objects $\Delta_\lambda$ to be the largest quotient of $P_\lambda$ whose simple factors $S_\mu$ have $\mu<\lambda$. The costandard objects $\nabla_\lambda$ are defined to be the largest submodule of $Q_\lambda$ with all simple factors $S_\mu$ having $\mu\leq\lambda$.

Note that when $P_\lambda$ coincide with $\Delta_\lambda$, the set $\Delta=\{\Delta_\lambda\mid\lambda\in I\}$ is a full strongly exceptional collection in $D^b(\catA)$. In this case, $S_\lambda$ coincide with $\nabla_\lambda$.

\begin{definition}
The category $\catA$ with $\{\nabla_\lambda\mid\lambda\in I\}$ and $\{\Delta_\lambda\mid\lambda\in I\}$ is said to be \textit{quasi-hereditary} if all the indecomposable projective objects $P_\lambda$ have filtrations with all the associated subquotients coincide with one of the standard objects.
\end{definition}

The following proposition is proved in \cite{DR92}. The version stated here is weaker  than the original one proved in \cite{DR92}, but enough for our purpose.
Let $\catA$ be a $k$-linear abelian category, and $\Delta=\{\Delta_\lambda\mid \lambda\in I\}$ a collection of objects, the additive full subcategory of  $\catA$ consisting of objects admitting filtration with subquotients coincide with $\Delta_\lambda$'s will be denoted by $\mathfrak{F}(\Delta)$.
\begin{prop}\label{prop: standardizable}
Let $\catA$ be a $k$-linear abelian category, and $\Delta=\{\Delta_\lambda\mid \lambda\in I\}$ is a finite exceptional collection in $\catA$ with the properties that $\dim_k\Hom(\Delta_\alpha,\Delta_\beta)<\infty$ and $\dim_k\Ext^1(\Delta_\alpha,\Delta_\beta)<\infty$ for any $\alpha$, $\beta\in I$.
Then
\begin{enumerate}
\item there is a collection of objects $\Phi=\{\Phi_\lambda\mid \lambda\in I\}$ in $ \mathfrak{F}(\Delta) $, in bijection with $\{\Delta_\lambda\}$ such that the object $\Phi:=\oplus_{\lambda\in I}\Phi_\lambda$ is a projective generator of $ \mathfrak{F}(\Delta) $, and $\End(\Phi)$-$mod$  has a quasi-hereditary structure with the standard objects given by $\Hom(\Phi,\Delta_\lambda)$.

\item If $\{\Delta_\lambda\}$ is a full exceptional collection, and $\Ext^i(\Phi,\Delta_\lambda)=0$ for all $i>0$ and $\lambda\in I$, then $\Phi$ is a tilting object in $D^b(\catA)$. In particular, we have an equivalence of triangulated categories $D^b(\catA)\cong D^b(\End(\Phi)\hbox{-}mod)$.

\item If $\{\nabla_\alpha\}$ is the dual collection, then the costandard objects in $\End(\Phi)$-$mod$ are given by $R\Hom(\Phi, \nabla_\lambda)$.
\end{enumerate}
\end{prop}

Under the assumptions of Proposition~\ref{prop: standardizable} (2) and (3), let $\Sigma_\lambda$ be the object in $\End(\Phi)$-$mod$ that corresponds to the simple top of $R\Hom(\Phi,\Delta_\lambda)$ under the equivalence $D^b(\catA)\cong D^b(\End(\Phi)\hbox{-}mod)$. Note that we automatically get natural maps $\Phi_\lambda\to \Delta_\lambda$, and $\Sigma_\lambda\to \nabla_\lambda$.
In particular,  $\Ext^\bullet(\Phi_\alpha ,\Sigma_\alpha) =
k$, where $k$ lies in homological degree 0.

Kapranov \cite{K} constructed an exceptional collection over
$\Grass=\Grass_{n-r}(E^*)$, the Grassmannian of $(n-r)$-planes in the vector space $E^*$ over $\mathbb{C}$.
\par
Let $$0\to \cR\to
E^*\times\Grass\to \cQ\to 0$$ be the tautological exact sequence over $\Grass$.
Recall that we will write $B_{u,v}$ to mean the set of partitions with no more than $v$ columns and no more than $u$ rows. For a partition $\lambda$, recall that  $\SCH_\lambda$ is the Schur functor corresponding to it. For a partition $\alpha$ we write $\alpha'$ for its transpose, and for any weight $\alpha$, we call $\sum_i\alpha_i$ its area which is denoted by $|\alpha|$.

\begin{theorem}[Kapranov, see also \cite{BLV3}]
For a suitable choice of ordering, $$\{\SCH_\alpha\cR^{* }\mid\alpha\in B_{n-r,r}\}$$ is
an exceptional collection in $D^b(Coh(\Grass))$, whose dual exceptional collection is given by
$$\{\SCH_{\alpha'}\cQ[|\alpha|]\mid\alpha\in B(n-r,r)\}.$$

If $k=\CC$, or $r=1$ or $n-1$, this exceptional collection is strong, hence $$\TIL_K=\oplus_{\alpha \in B_{n-r,r}} \SCH_\alpha \cR^{*}$$ is a tilting
bundle.
\end{theorem}
\par
Observe that $-\otimes\wedge^{top}\cQ$ and $-\otimes\wedge^{top}\cR$
define a $\Z^2$-action on the triangulated category
$D^b(Coh(\Grass_{n-r}(E^*)))$. This action sends one exceptional
collection to another, and preserves duality.
\par
Applying the $\mathbb{Z}^2$-action we can see that $\Delta(\Grass)=\{\SCH_\lambda\cQ^*\mid\lambda\in B_{r,n-r}\}$ over $\Grass$ is also a full exceptional collection. As can be checked by definition, the dual collection is given by $\nabla(\Grass)=\SCH_{(n-r)^r}\cQ^*\otimes \SCH_{(\alpha^c)'}\cR[(n-r)r-|\alpha|]$, where for a subpartition $\alpha$ of $r^{n-r}$ we write $\alpha^c$ for its complement in the rectangle $(r^{n-r})$.
\par
If $k$ has positive characteristic or $1<r<n-1$, Proposition~\ref{prop: standardizable}.(1) gives a collection $\Phi(\Grass)=\{ \Phi_\alpha\mid\alpha\in B_{r,n-r}\}$.  For $\alpha\in B_{r,n-r}$, denote the simple top of $\WEY_\alpha k^{n-r}$ in $\Rep(\GL_{n-r})$ by $L_\alpha$ and its projective cover by $M_\alpha$, similarly denote the simple top of $\WEY_{\lambda'}k^r$ by $L'_{\alpha'}$.
It is shown in \cite{BLV3} that $\Phi_\alpha$ given in Proposition~\ref{prop: standardizable}.(1) is equal to $\mathfrak{L}_{\Grass}(M_\alpha)$, and $\Sigma_\alpha=\mathfrak{L}_{\Grass}(L'_{\alpha'})[|\alpha|]$. Also the hypothesis in Proposition~\ref{prop: standardizable}.(2) is satisfied, hence, $\TIL=\oplus_{\alpha\in B_{r,n-r}}\Phi_\alpha$ is a tilting bundle.

Furthermore, there is a characteristic free collection of vector bundles on the Grassmannian, whose all direct summands are $\Phi_\alpha$.
For a partition $\alpha=(\alpha_1\cdots, \alpha_r)$ and any vector space $V$, $$\wedge^\alpha V=\wedge^{\alpha_1}\otimes\cdots\otimes\wedge^{\alpha_r}V.$$
\begin{theorem}[\cite{BLV3}]\label{thm-til-grass}
The vector bundle $\TIL_0=\oplus_{\alpha\in B_{r,n-r}}\wedge^{\alpha'}\cQ$ is a classical tilting bundle on $\Grass$.
In $\End_{\Grass}(\TIL_0)$-$mod$, the pair of collections $$\Delta_\alpha=R\Hom_{\Grass}(\TIL_0,\SCH_\alpha \cQ)$$ $$\nabla_\alpha=R\Hom_{\Grass}(\TIL_0,\SCH_{\alpha'}\cR^*[|\alpha|])$$ gives a quasi-hereditary structure. The simple objects are $R\Hom_{\Grass}(\TIL_0,\mathfrak{L}_{\Grass}(L'_{\alpha'})[|\alpha|])$ and their projective covers are given by $R\Hom_{\Grass}(\TIL_0,\mathfrak{L}_{\Grass}(M_{\alpha}))$.

In particular, $\End_{\Grass}(\TIL_0)$ is Morita equivalent to the basic algebra $\End_{\Grass}(\oplus_\lambda\mathfrak{L}_{\Grass}(M_\alpha))$.
\end{theorem}

To summarize, take $\Delta(\Grass)\{\SCH_\lambda\cQ^*\mid\lambda\in B_{r,n-r}\}$ in Proposition~\ref{prop: standardizable}.(1), then $\TIL=\oplus_{\alpha\in B_{r,n-r}}\Phi_\alpha$ is a multiplicity-free tilting bundle, which is the same as $\TIL_0$ up to multiplicity, and specializes to $\TIL_K$ when $k=\CC$.
The vector bundle $\TIL_0$ in Theorem~\ref{thm-til-grass} will be referred to as the BLV's tilting bundle, and $\TIL_K$ the Kapranov's tilting bundle.
We will also write $\TIL_0(r,n)$ and $\TIL_K(r,n)$ when the base  needs to be emphasized.

The following proposition proved in \cite{BLV3} will be used later.
\begin{prop}\label{prop: vanishing}
Let $\alpha\in B_{r,n-r}$ and let $\delta$ be any partition. Then for all $i >0$ one has
$$H^i(\Grass,(\wedge^{\alpha'}\cQ)^*\otimes_{\sO_{\Grass}}\SCH_{\delta}\cQ)= 0.$$
\end{prop}

\section{Determinantal varieties in spaces of matrices}\label{sec:sym}
In this section we study a non-commutative desingularization of determinantal varieties in
the space of $n\times m$-matrices, determinantal varieties in
the space of symmetric matrices, and pfaffian varieties.

Let $E$ be a vector space over $k$ of dimension $n$, and let $\Grass$ be the Grassmannian of $n-r$ planes in $E$ where $0<r<n$. Let $$0\to \cR\to E\times\Grass\to \cQ\to 0$$ be the tautological
sequence over $\Grass$.
Let $G$  be a vector space of dimension $m$ with $m\geq n$, and $F=E^*$.

\subsection{Determinantal varieties}
Upon choosing bases for $G$ and $F$, the affine space $\Hom_k(G,F)\cong\mathbb A^{mn}$  is also identified with the set
of $(m\times n)$-matrices $(x_{ij})$. The coordinate ring of $\mathbb A^{mn}$ can
be identified with $S^m=\mathbb{C}[x_{ij}]$.
\par
By base extension from $\Spec k$ to $\mathbb A^{mn}$, we get two vector bundles
$\cF$ and $\cG$ over $\mathbb A^{mn}$, and a universal morphism
$\varphi:\cG\to\cF$.  We desingularize the locus
$\Spec R^m$ where $\varphi$ has rank $\leq r$. In other words, $R^m$ is
the quotient of $S^m$ by the ideal generated by the $(r+1)\times
(r+1)$ minors of $(x_{ij})$.

Let $\cY^m$ be the product $\Grass\times H^m$, $p$ and $q$ be
the projection to $\Grass$ and $\mathbb A^{mn}$ respectively. Inside of $\cY^m$,
there is an incidence variety, denoted by $\cZ^m$, defined by
$$\cZ^m=\{(g,h)\in\Grass\times \mathbb A^{mn}:\im g\circ h=0\}.$$ 
In the setup of \S~\ref{subsec: inverse image til}, take $R=R^m$, $\mathbb A^{N}=\mathbb A^{mn}$, $V=\Grass$, and $Z=\cZ^m$. We follow the same notation for the maps among these spaces, which are summarized in Diagram~\eqref{diag: general}.
\par
\begin{prop}[\S~6.1.1 in \cite{W03}]
The variety $\cZ^m$ is a desingularization of $\Spec R^m$.
\end{prop}
The variety $\cZ^m$ can be described as the total space of a vector
bundle over $\Grass$.

We apply the functor $G^*\otimes-$ to the
dualized tautological sequence $0\to \cQ^*\to F\times\Grass\to
\cR^*\to 0$ to get
$$0\to \calS\to \cE\to \calT\to 0.$$
Following \cite{W03}, let $\eta$ be the sheaf of sections of
$\calS^*=G\otimes \cQ$. The desingularization $\cZ$ is the total space
of $\calS$. Equivalently, $p':\cZ^m\to \Grass$ is an affine morphism with
$p_*\sO_{\cZ^m}$ equal to the sheaf of algebra $\Sym(\eta)$.

The following theorem follows from the discussions in \S~\ref{sec: general_discussion} and \S~\ref{sec: grass}.  This has also been proved in \cite{BLV2}. 
\begin{theorem}
\label{det tilt} The BLV's tilting bundle over
$\Grass_{n-r}(F^*)$ is denoted by
$$\TIL_0=\oplus_{\alpha\in B_{r,n-r}}\wedge^{\alpha'}\cQ^*.$$ The rank $r$ determinantal
variety $\Spec R^m$ and its desingularization $\cZ^m$ are as above. 
\begin{enumerate}
\item The
bundle $p'^*\TIL_0$ is a tilting bundle over $\cZ^m$.
\item The modules $\TIL_0(\cZ^m)$ and $\End_{\cZ^m}(p'^*\TIL_0)$ are MCM.
\item The algebra $\End_{\cZ^m}(p'^*\TIL_0)$ is an NCCR of $\Spec R^m$.
\item The functor $R\Hom_{\sO_{\cZ^m}}(p'^*\TIL_0,-)$ induces an equivalence $D^b_G(Coh(\cZ^m))\cong D^b_G(\End_{\sO_{\cZ^m}}(p'^*\TIL_0)\hbox{-}mod).$
\end{enumerate}
\end{theorem}

\subsection{Minimal presentation for the non-commutative desingularization}\label{subsec:min-present}
We study the minimal presentation of the noncommutative desingularization
$\End_{S^m}(q_*'p'^*\TIL_K)=\End_{\cZ^m}(p'^*\TIL_K)$ as module over $S=\CC[x_{i,j}]$. \textit{In this subsection we assume $k=\CC$.}
\par
One can get the quiver with relations for the non-commutative desingularization. The quiver with relations in this case is described in \cite{BLV2}. We study the endomorphism ring using the basic theorem of geometric technique.
\par

In the basic theorem of geometric technique Theorem~\ref{5.1.2 basic}, we take
the vector bundle $\cV$ to be $\SCH_\alpha\cV=\SCH_\alpha\cQ^{*}$. Let $\eta=\cQ\otimes
G$. The modules $H^0(\Grass, \Sym(\eta)\otimes \SCH_\alpha\cQ^{*})$ will
be denoted by $M_\alpha$. By the same argument as in
Proposition~\ref{det tilt}, we get the vanishing of higher derived
images for $M(\SCH_\alpha\cV)$, i.e.,
$$R^iq_*M(\SCH_\alpha\cV)=R^iq_*(\sO_\cZ\otimes p^*\SCH_\alpha \cV)=H^i(\Grass,\Sym(\eta)\otimes \SCH_\alpha \cV)=0$$ through the expansion of $\Sym(\eta)\otimes \SCH_\alpha \cV$ using the Cauchy and the Littlewod-Richardson formulas, and Bott's theorem.
\par
According to Theorem~\ref{5.1.2 basic}, there is a presentation of
$M_\alpha$ given by $F_{\alpha,1}\to F_{\alpha,0}$ with
$F_{\alpha,i}$ defined to be
\begin{eqnarray*}
F_{\alpha,i}&=&\bigoplus_{j\geq0}H^j(\Grass,
\wedge^{i+j}(\xi)\otimes
\SCH_\alpha \cQ^*)\otimes S\\
&=&\bigoplus_{j\geq0}H^j(\Grass, \oplus_{|\mu|=i+j}\SCH_{\mu'}G\otimes
\SCH_\mu \cR\otimes \SCH_\alpha \cQ^*)\otimes
S\\
&=&\bigoplus_{j\geq0}\oplus_{|\mu|=i+j}H^j(\Grass, \SCH_\mu \cR\otimes
\SCH_\alpha \cQ^*)\otimes \SCH_{\mu'}G\otimes
S.\\
\end{eqnarray*}
\par
To compute the $F_{\alpha,i}$'s, all we need to do is to compute
$H^j(\Grass, \SCH_\mu \cR\otimes \SCH_\alpha \cQ^*)$.
\par

\begin{theorem}Notation as above,
\begin{enumerate}
  \item $F_{\alpha,1}=\SCH_{(\alpha,1^{r+1-t},0^{m-r-1})}G\otimes \SCH_{(1^{r+1-t})}F^*\otimes
  S$;
  \item $F_{\alpha,0}=\SCH_\alpha G \otimes S$.
\end{enumerate}
\end{theorem}
\begin{proof}
Assume $l(\alpha)=t$, i.e., $\alpha_t\neq 0$ and $\alpha_{t+1}=0$,
$|\mu|=j+1$. To show (1), it suffices to prove that
$H^j(\SCH_{\mu'}E\otimes \SCH_{\mu}\cR\otimes \SCH_{\alpha}\cQ^*)$ is isomorphic to
\begin{equation*}
 \left\{
\begin{aligned}
         \SCH_{(\alpha,1^{r+1-t},0^{m-r-1})}G\otimes \SCH_{(1^{r+1-t})}F^* & \hbox{, if }j=|\alpha|+r-t-1 \hbox{ and } \mu=(\alpha'+(r-t,0^{n-r-1})); \\
                  0&\hbox{, otherwise}.
                          \end{aligned} \right.
                          \end{equation*}
\par
We have, using the language of \S~\ref{subsec: geo tech},
\begin{eqnarray*} H^j(\SCH_{\mu'}G\otimes \SCH_{\mu}\cR\otimes
\SCH_{\alpha}\cQ^*)&=&H^j(\SCH_{\mu}\cR\otimes
\SCH_{\alpha}\cQ^*)\otimes \SCH_{\mu'}G\\
 &=&H^j(\cV(0/\mu,\alpha))\otimes \SCH_{\mu'}G.
\end{eqnarray*}
Since $\alpha_1\leq n-r$, one easily sees that $l(\mu)=\alpha_1$. In
fact, if $l(\mu)<\alpha_1$, we would get $(0,\cdots,0,1,\cdots)$
performing the symmetric group action described in \S~\ref{subsec: geo
tech}; and we would get negative entries followed by 0s performing
the symmetric group action if $l(\mu)>\alpha_1$.
\par
Note also that we must have $\mu\supset \alpha$. Otherwise we would
get 0's followed by positive entries performing the symmetric group
action.
\par
Performing adjacent transposition actions $\alpha_1$ times, we can delete
the first row of $\alpha$, delete the first column of $\mu$, and
move everything remains in $\mu$ down by one row. Inductively, after
performing  transposition actions $|\alpha|$ times, we get
$(0,\cdots,0,\alpha'/\mu',0^{r-t})$.
\par
To get non-trivial $H^{|\mu|-1}$, we perform
$|\mu|-|\alpha|-1$ transposition actions to $(\alpha'/\mu',0^{r-t})$
to make it dominant. The only possibility is
$\alpha'/\mu'=(0,\cdots,0,-r+t+1)$, with the corresponding
$\mu=(\alpha'+(r-t,0^{n-r-1}))$ and
$H^{|\mu|-1}=\SCH_{(\alpha,1^{r+1-t},0^{m-r-1})}E\otimes
\SCH_{(1^{r+1-t})}F^*$.
\par
By the same argument, one proves (2).
\end{proof}

\subsection{Determinantal varieties in the space of symmetric matrices}\label{subsec:sym}
We consider the affine space $\mathbb A^{{n+1}\choose{2}}$, which is the
subspace of $\Hom_k(E,E^*)$ consisting of symmetric morphisms. This
space can be identified with $Sym_2(E^*)$. Upon choosing a set of
basis for $E$, $\mathbb A^{{n+1}\choose{2}}$ can be identified with the set of
symmetric $(n\times n)$-matrices $(x_{ij})$ with $x_{ij}=x_{ji}$ the
coordinate ring of which can be identified with
$S^s=k[x_{ij}]_{i\leq j}$.
We have a universal
morphism $\varphi:\cE\to\cE^*$ over $\mathbb A^{{n+1}\choose{2}}$. For $r\leq n$, let $\Spec R^s\subseteq \Spec S^s$ be the locus where $\varphi$ has rank $\leq
r$. In other words, $R^s$ is the quotient of $S^s$ by the ideal
generated by the $(r+1)\times (r+1)$ minors of $(x_{ij})$.

Let $\cY^s$ be the product $\Grass\times \mathbb A^{{n+1}\choose{2}}$, $p$ and
$q$ be the projection to $\Grass$ and $\mathbb A^{{n+1}\choose{2}}$ respectively. On
$\cY^s$, there is an incidence variety, denoted by $\cZ^s$, defined
as $\cZ^s=\{(g,h)\in\Grass\times \mathbb A^{{n+1}\choose{2}}:\im h\in g\}$. 
In the setup of \S~\ref{subsec: inverse image til}, take $R=R^s$, $\mathbb A^{N}=\mathbb A^{{n+1}\choose{2}}$, $V=\Grass$, and $Z=\cZ^s$. We follow the same notation for the maps among these spaces, which are summarized in Diagram~\eqref{diag: general}.

\begin{prop}[6.3.2 in \cite{W03}]
The variety $\cZ^s$ is a desingularization of $\Spec R^s$.
\end{prop}
The variety $\cZ^s$ can be described as the total space of the
vector bundle $\Sym_2(\cQ^*)$ over $\Grass$. Equivalently,
$p':\cZ^s\to \Grass$ is an affine morphism with $p_*\sO_{\cZ^s}$
equal to the sheaf of algebra $\Sym(\Sym_2\cQ)$.
\par

We consider the inverse image of
$\TIL_0=\oplus_{\alpha\in B_{r,n-r}}\wedge^{\alpha'}\cQ$, the BLV's
tilting bundle over $\Grass_{n-r}(E)$, by $p':\cZ^s\to\Grass$.
We  show the following.
\par
\begin{lemma}
For all $i>0$, and any $\alpha$, $\beta\in B_{n-r,r}$, we have $$H^i(\cZ^s, p'^*\sHom_{\sO_{\Grass}}(\wedge^{\alpha}\cQ,
\wedge^{\beta}\cQ))=0.$$
\end{lemma}
\begin{proof}
We have $p'^*\sHom_{\sO_{\Grass}}(\wedge^{\alpha}\cQ,
\wedge^{\beta}\cQ)\cong\sHom_{\sO_{\Grass}}(\wedge^{\alpha}\cQ
,\wedge^{\beta}\cQ\otimes_{\sO_{\Grass}} \Sym (\Sym_2\cQ))$. Use Theorem~\ref{thm: Litt_Rich}, we
know that the sheaf $\wedge^{\beta}\cQ\otimes_{\sO_{\Grass}} \Sym (\Sym_2\cQ)$ has a filtration with subquotients
of the form $\SCH_{\gamma}\cQ$. By Proposition~\ref{prop: vanishing}, all the
higher cohomology of $$\sHom_{\sO_{\Grass}}(\SCH_{\alpha}\cQ^*,
\SCH_{\beta}\cQ^*\otimes_{\sO_{\Grass}}\Sym(\Sym_2\cQ))$$ vanishes.
\end{proof}
\par
From the lemma above, using Proposition~\ref{til-resol criterion}, we
get the following.
\begin{theorem}
\label{sym tilt} Let the BLV's tilting bundle over
$\Grass_{n-r}(E)$ be denoted by $\TIL_0$. Let the rank-$r$ determinantal
variety $\Spec R^s$ of symmetric matrices and its desingularization
$\cZ^s$ be as above.
The bundle $p'^*\TIL_0$ is a tilting bundle
over $\cZ^s$.
In particular, $\End_{\sO_{{\cZ^s}}}(p'^*\TIL_0)$ is a non-commutative weak desingularization of $\Spec R^s$;
the functor $R\Hom_{\sO_{\cZ^s}}(p'^*\TIL_0,-)$ induces an equivalence $D^b_G(Coh({\cZ^s}))\cong D^b_G(\End_{\sO_{{\cZ^s}}}(p'^*\TIL_0)\hbox{-}mod)$.
\end{theorem}

The calculation of the quiver and relations for $\End_{\sO_{{\cZ^s}}}(p'^*\TIL_0)$ is in \S~\ref{sec:calculations}. Now we study when $\End_{\sO_{{\cZ^s}}}(p'^*\TIL_0)$ is a non-commutative desingularization, and when it is an NCCR.

\begin{prop}\label{lem:sym_mcm}\label{lem:sym_max_ref}\label{prop: sym_mcm}
Notation as above, we have the following.
\begin{enumerate}
  \item If $r=n-1$, $\End_{\cZ^s}(p'^*\TIL_0)$ is maximal Cohen-Macaulay over $\Spec R^s$. In particular, it is an NCCR.
  \item If $r<n-1$, $\End_{\cZ^s}(p'^*\TIL_0)$ is never maximal Cohen-Macaulay over $\Spec R^s$.
\item The $R^s$-module $q_*'p'^*\TIL_0$ is maximal Cohen-Macaulay.
\end{enumerate}
\end{prop}
\begin{proof}
According to Lemma~\ref{lem:mcm}, it suffices to compute \[H^i(\Grass,\wedge^\alpha \cQ^*\otimes \wedge^\beta \cQ\otimes\omega_{\cZ^s}\otimes\Sym(\Sym_2\cQ))\] for $i>1$ and $\alpha$, $\beta\in B_{r,n-r}$. The sheaf $\omega_{\cZ^s}$ has been computed in  \cite[6.7]{W03}, which yields \[\omega_{\cZ^s}\cong (\wedge^nE^*)^{\otimes -n-1+r}\otimes (\wedge^rQ^*)^{\otimes n-r-1},\] hence, $$\wedge^\alpha \cQ^*\otimes \wedge^\beta \cQ\otimes\omega_{\cZ^s}\otimes\Sym(\Sym_2\cQ)\cong \wedge^\alpha \cQ^*\otimes \wedge^\beta \cQ\otimes (\wedge^rQ^*)^{\otimes n-r-1}\otimes\Sym(\Sym_2\cQ)(\wedge^nE^*)^{\otimes -n-1+r}.$$ Using Theorem~\ref{thm: Litt_Rich}, as well as Proposition~\ref{prop: vanishing}, we see that when $r=n-1$ we have $$H^i(\Grass,\wedge^\alpha \cQ^*\otimes \wedge^\beta \cQ\otimes\omega_{\cZ^s}\otimes\Sym(\Sym_2\cQ))=0$$ for all $i>0$.

Now let $r<n-1$. We need to show the non-vanishing of \[H^i(\Grass,\wedge^\alpha \cQ^*\otimes \wedge^\beta \cQ\otimes\omega_{\cZ^s}\otimes\Sym(\Sym_2\cQ))\] for some positive $i$. By upper-semi-continuity, it suffices to show this non-vanishing in characteristic zero case. For $k=\CC$, taking $\beta=0$ and $\alpha$ with $\alpha'_1=2$, the Bott Theorem tells us  \[\wedge^\alpha \cQ^*\otimes(\wedge^rQ^*)^{\otimes n-r-1}\otimes\Sym(\Sym_2\cQ)\] does have non-vanishing higher cohomology. Therefore, in this case $\End_{\cZ^s}(p'^*\TIL_0)$ is never maximal Cohen-Macaulay.

The third statement follows from the same calculation.
\end{proof}

When $r<n-1$, we prove in \S~\ref{sec:depth} that $\End_{\cZ^s}(p'^*\TIL_0)$ is not reflexive, therefore, the map $\End_{\cZ^s}(p'^*\TIL_0)\to \End_{R^s}(q'_*p'^*\TIL_0)$ is not an isomorphism. It is injective nevertheless.

\subsection{Pfaffian varieties in the space of skew-symmetric matrices}
\label{subsec:skew}
We consider the affine space $\mathbb A^{{n}\choose{2}}$, which is the
subspace of $\Hom_k(E^*,E)$ consisting of skew-symmetric morphisms.
This space can be identified with $\wedge^2(E)$. Upon choosing a set
of basis,  $\mathbb A^{{n}\choose{2}}$ can be identified with the set of
skew-symmetric $(n\times n)$-matrices $(x_{ij})$ with $x_{ij}=-x_{ji}$,
whose coordinate ring can be identified with
$S^a=k[x_{ij}]_{i<j}$.
\par
We get a universal morphism $\varphi:\cE^*\to\cE$ over $\mathbb A^{{n}\choose{2}}$. For even $r = 2u$ satisfying $0 \leq r \leq n$, we desingularize the locus $\Spec R^a$, where $\varphi$ has
rank $\leq r$. In other words, $R^a$ is the quotient of $S^a$ by the
ideal generated by the $(r+1)\times (r+1)$ minors of $(x_{ij})$.
\par
The incidence variety $\cZ^a$ inside $\cY^a:=\Grass\times \mathbb A^{{n}\choose{2}}$
desingularizing $\Spec R^a$, is defined by
$\cZ^a=\{(g,h)\in\Grass\times \mathbb A^{{n}\choose{2}}:\im h\in g\}$. 
In the setup of \S~\ref{subsec: inverse image til}, take $R=R^a$, $\mathbb A^{N}=\mathbb A^{{n}\choose{2}}$, $V=\Grass$, and $Z=\cZ^a$. We follow the same notations for the maps among these spaces, which are summarized in Diagram~\eqref{diag: general}. 
\par
\begin{prop}[6.4.2 in \cite{W03}]
The variety $\cZ^a$ is a desingularization of $\Spec R^a$.
\end{prop}
The variety $\cZ^a$ can be described as the total space of the
vector bundle  $\wedge^2(\cQ^*)$ over $\Grass$. Equivalently,
$p':\cZ^a\to \Grass$ is an affine morphism with $p_*\sO_{\cZ^a}$
equal to the sheaf of algebra $\Sym(\wedge^2\cQ)$.
\par
As before, we pull back the BLV's tilting bundle $\TIL_0=\oplus_{\lambda\in B_{r,n-r}}\wedge^{\lambda'} Q$ over
$\Grass_{n-r}(E)$ by the projection $p':\cZ^a\to\Grass$. Similar to the symmetric case, we can show that for all $i>0$, and all $\alpha$, $\beta\in B_{r,n-r}$, $$H^i(\cZ^a, p'^*\sHom_{\sO_{\Grass}}(\wedge^{\alpha'}\cQ,
\wedge^{\beta'}\cQ))=0.$$
\par
From the claim above, using Proposition~\ref{til-resol criterion}, we
get the following.
\begin{theorem}
\label{anti tilt} The BLV's tilting bundle over $\Grass_{n-r}(E)$
is denoted by $\TIL_0$. The rank $r$ determinantal variety $\Spec
R^a$ of anti-symmetric matrices and its desingularization $\cZ^a$ are as
above. The bundle $p'^*\TIL_0$ is a tilting bundle over $\cZ^a$. In particular, $\End_{\sO_{{\cZ^a}}}(p'^*\TIL_0)$ is a non-commutative weak desingularization of $\Spec R^s$;
the functor $R\Hom_{\sO_{\cZ^a}}(p'^*\TIL_0,-)$ induces an equivalence $D^b_G(Coh({\cZ^a}))\cong D^b_G(\End_{\sO_{{\cZ^a}}}(p'^*\TIL_0)\hbox{-}mod)$.
\end{theorem}

The following proposition is proved along the same line.
\begin{prop}\label{prop: MCM_skew}
Over $\Spec R^a$, the coherent sheaf $q'_*p'^*\TIL_0$ is maximal Cohen-Macaulay. But the coherent sheaf $\End_{\cZ^a}(p'^*\TIL_0)$ is never maximal
Cohen-Macaulay.
\end{prop}
\begin{proof}
According to Lemma~\ref{lem:mcm}, it suffices to compute \[H^i(\Grass,\wedge^\alpha \cQ^*\otimes \wedge^\beta \cQ\otimes\omega_{\cZ^a}\otimes\Sym(\wedge^2\cQ))\] for $i>0$ and $\alpha$, $\beta\in B_{n-r,r}$. The sheaf $\omega_{\cZ^a}$ has been computed in  \cite[6.7]{W03}, which says $\omega_{\cZ^a}\cong (\wedge^nE^*)^{\otimes -n+1+r}\otimes (\wedge^rQ^*)^{\otimes n-r+1}$. Hence, $$\wedge^\alpha \cQ^*\otimes \wedge^\beta \cQ\otimes\omega_{\cZ^a}\otimes\Sym(\wedge^2\cQ)\cong \wedge^\alpha \cQ^*\otimes \wedge^\beta \cQ\otimes (\wedge^nE^*)^{\otimes -n+1+r}\otimes (\wedge^rQ^*)^{\otimes n-r+1}\otimes\Sym(\wedge^2\cQ).$$
If $\alpha=0$, there is no higher cohomology, according to Proposition~\ref{prop: vanishing}. This shows $q'_*p'^*\TIL_0$ is maximal Cohen-Macaulay.

In order to show $\End_{\cZ^a}(p'^*\TIL_0)$ is never maximal, we only need to find out the non-vanishing of $H^i(\Grass,\wedge^\alpha \cQ^*\otimes\wedge^\beta \cQ\otimes (\wedge^rQ^*)^{\otimes n-r+1}\otimes\Sym(\wedge^2Q))$ for some positive $i$. By upper-semi-continuity, it suffices to show this non-vanishing in characteristic zero. For $k=\CC$, taking $\beta=0$ and $\alpha\neq 0$, the Bott Theorem tells us  \[\wedge^\alpha \cQ^* \otimes (\wedge^rQ^*)^{\otimes n-r+1}\otimes\Sym(\wedge^2Q))\] does have nonvanishing higher cohomology. We are done.
\end{proof}
We prove in \S~\ref{sec:depth} that $\End_{\cZ^a}(p'^*\TIL_0)$ is never reflexive, therefore the map $\End_{\cZ^a}(p'^*\TIL_0)\to \End_{R^a}(q'_*p'^*\TIL_0)$ is not an isomorphism. It is injective nevertheless. 

\section{Discrepancy of a non-commutative desingularization}
\label{sec:depth}
In this section, we prove a criterion for the direct-image of a tilting bundle to be reflexive (Proposition~\ref{prop:main}). Then we apply it to study the non-commutative weak desingularizations of determinantal varieties in the space of symmetric and skew-symmetric matrices constructed from \S~\ref{sec:sym}. 
\subsection{Reflexive and MCM modules}
This subsection is to recall some basic notions and first properties from commutative algebra, can be skipped when reading. 

\begin{definition}
Let $R$ be a Noetherian commutative ring, and $I$ be an ideal. Then for any finitely generated $R$-module $M$, $\depth_IM=\{n\mid \Ext^n_R(R/I,M)\neq0\}$. 

When $(R, m, k)$ is a local ring, $M$ a finite $R$-module. 
Define $\depth M=\depth_mM=\min\{n\mid \Ext_R^n(k,M)\neq0\}$.
\end{definition}

\begin{definition}
Let $R$ be a Noetherian commutative ring, a finite $R$-module
$M$ is said to satisfy Serre's $S_n$ if $\depth_{R_{\mathfrak{p}}} M_{\mathfrak{p}} \geq\min\{n,\dim R_{\mathfrak{p}}\}$ for every $\mathfrak{p}\in\Spec R$.
\end{definition}

The following fact can be found in \cite{Bou}.
\begin{prop}
Let $R$ be a Noetherian commutative ring, a finite $R$-module
$M$ is reflexive iff $\supp M=R$ and satisfies $S_2$.
\end{prop}

All the following first properties about dualizing complexes can be found in \cite{Stack}, where the convention has been adjusted from \cite{Stack} so that they are compatible with \cite{W03}. 

Let $R$ be a Noetherian ring. Let $\omega_R$ be a dualizing complex. Let
$m \subset R$ be a maximal ideal and set $k = R/m$. Then $RHom_R(k, \omega_R) \cong k[n]$ for some
$n \in \mathbb Z$.
We say that the dualizing
complex $\omega_R$ is normalized if $n=\dim R$.

\begin{lemma}[\cite{Stack}, Lemma 15.6.] \label{lem:15.6}
Let $A$ be a Noetherian ring, and  $B = S^{-1}A$ its localization with respect to some multiplicative set $S$. If $\omega_A$
is a dualizing complex, then $\omega_A\otimes_A B$ is a  dualizing complex for $B$.

If $\omega_A$ is normalized, then so is $\omega_A\otimes_A B$.
\end{lemma}

As a consequence, for any finite $R$-module $M$ and $\fp\in\Spec R$, we have  $\Ext^i_R(M,\omega_R)_\fp\cong \Ext^i_{R_\fp}(M_\fp,\omega_{R_\fp})\cong \Ext^i_{R_\fp}(M_\fp,\omega_{R_\fp})$.

\begin{lemma}[\cite{Stack}, Lemma 16.4.] \label{lem:16.4}
Let $(R, m, k)$ be a Noetherian local ring with normalized dualizing
complex $\omega_R$. Let $M$ be a finite $R$-module and let $d = \dim(\supp(M))$, $n=\dim R$. Then
\begin{enumerate}
\item if $Ext^i_R(M,\omega_R)$ is nonzero, then $i \in \{n-d, \dots, n\}$;
\item the dimension of the support of $Ext^{n-i}_R(M, \omega_R)$ is at most $-i$;
\item $\depth(M)$ is the smallest integer $\delta \geq 0$ such that $Ext^{n-\delta}_R (M, \omega_R) \neq 0$.
\end{enumerate}
\end{lemma}

\subsection{Depth of push-forward of vector bundles}
Now assume $f:\tilde{Z}\to Z=\Spec R$ is a resolution of singularities, such that the exceptional locus in $Z$ is irreducible of codimension at least 2, and the exceptional  locus in $\tilde{Z}$ is an irreducible divisor. 
\begin{lemma}\label{lem:codim_fiber}
In the set-up as above, let $Y\subseteq Z$ be a closed irreducible subvariety contained in the exceptional locus, such that generically on $Y$, the map $f$ has relative dimension $d$. Then $\codim Y\geq d-1$, with equality iff $Y$ is the exceptional locus itself.
\end{lemma}
\begin{proof}
Clearly  $f^{-1}(Y)$ is an proper subscheme in $\tilde{Z}$, hence $\dim \tilde{Z}-1\geq \dim f^{-1}(Y)=\dim Y+d$.
Plug-in the equality $\dim Y+\codim Y=\dim Z=\dim \tilde{Z}$, we get $\codim Y\geq d-1$. 

The equality holds if and only if $\dim \tilde{Z}-1= \dim f^{-1}(Y)$, in which case $f^{-1}(Y)$ is an exceptional divisor. By assumption on $f$, $f^{-1}(Y)$ is the exceptional locus in $\tilde{Z}$ and hence $Y$ is the exceptional locus in $Z$.
\end{proof}

Let $\calM$ be a vector bundle on $\tilde{Z}$ such that $Rf_*\calM\cong f_*\calM$, denoted by $M$.

\begin{prop}\label{prop:main}
In the setup above, let $Y_0$ be the exceptional locus in $Z$, and let $d_0+1$ be the  codimension of $Y_0$. Then,  $M$ is reflexive iff $H^{d_0}(Z,\sHom_{\tilde{Z}}(\calM,\omega_{\tilde{Z}}))$ is supported on a proper subscheme of $Y_0$.
\end{prop}
\begin{proof}
Let $\fp\in\Spec R$ be of codimension 1, whose strict transform in $\tilde{Z}$ is $\eta$, then clearly $M_\fp\cong \calM_\eta$ which is locally free.
Note that $\dim R_\fp=\codim \fp$. For $M$ to be reflexive, we only need to show $M_\fp$ is of depth at least 2 as a module over $R_\fp$, which in turn (by Lemma~\ref{lem:16.4}) amounts to showing $\Ext_{\tilde{Z}}^i(\calM,\omega_{\tilde{Z}})_{\fp}=0$ for $i=\dim R_\fp$ and $\dim R_\fp-1$. 

Let $\fp\in\Spec R$ be contained in the exceptional locus, then Grothendieck duality and Lemma~\ref{lem:15.6} yield that $\Ext_{R_\fp}^i(M_{\fp},\omega_{R_\fp})\cong\Ext^i_{R}(M,\omega_R)_\fp\cong\Ext_{\tilde{Z}}^i(\calM,\omega_{\tilde{Z}})_{\fp}\cong H^i(Z,\sHom_{\tilde{Z}}(\calM,\omega_{\tilde{Z}}))_\fp$.

As $\Spec R_\fp\inj Z$ is an open embedding, in particular a flat morphism, we can use cohomology and flat base change to  $\Spec R_\fp\inj Z$ to get that \[H^j(Z,\sHom_{\tilde{Z}}(\calM,\omega_{\tilde{Z}}))|_{R_\fp}\cong H^i(Z|_{R_\fp},\sHom_{\tilde{Z}}(\calM,\omega_{\tilde{Z}})|_{R_\fp} ).\]
In particular, if $j$ is larger than the dimension of any fibers  over $R_\fp$, then the above cohomology vanishes for dimension reasons. 
Hence, if $\fp$ is a proper subvariety of $Y_0$, then Lemma~\ref{lem:codim_fiber} implies that $\Ext_{\tilde{Z}}^i(\calM,\omega_{\tilde{Z}})_{\fp}=0$ for $i=\dim R_\fp$ and $\dim R_\fp-1$. 
In other words, $M_\fp$ has depth strictly more than 2 as module over $R_\fp$, when $\fp$ is a proper subvariety of $Y_0$. 

Now we consider the case when $\fp$ is the vanishing ideal of $Y_0$. By the same dimension argument as in the last paragraph,  $\Ext_{\tilde{Z}}^i(\calM,\omega_{\tilde{Z}})_{\fp}=0$ for $i=\dim R_\fp=d_0+1$ (the dimension of generic fiber over $Y_0$ is $d_0$).  Therefore, $M_\fp$ has depth 2 if and only if $H^{d_0}(Z,\sHom_{\tilde{Z}}(\calM,\omega_{\tilde{Z}}))_\fp=0$, which in turn amounts to requesting $H^{d_0}(Z,\sHom_{\tilde{Z}}(\calM,\omega_{\tilde{Z}}))$ to be supported on a proper subscheme of $Y_0$.
\end{proof}

\subsection{Applications into determinantal varieties}\label{subsec:main_example}
Let  $\Spec R^s$ be the space of $n\times n$-symmetric matrices of rank $\leq r$. We have the tautological sequence \[0 \to \cR \to E\otimes\sO_{\Grass(n-r, E)}\to   \cQ \to 0\]
with $\rank \cR = n-r$ and $ \rank \cQ = r$. Let $\cZ^s$ be the total space of $\Sym_2\cQ^*$. Let $\calM_\gamma$ be $p^*S_\gamma\cQ$ on $\cZ^s$ for any partition $\gamma=(\gamma_1,\ldots,\gamma_r)$ with $n-r\ge\gamma_1\ge\gamma_r\ge -n+r$.

The following is a direct corollary to Proposition~\ref{prop:main}.
\begin{corollary}
Let
$M_\gamma:=H^0(\Grass(r,n),S_\gamma Q\otimes \Sym(S_2Q))$, where $\gamma=(\gamma_1,\ldots,\gamma_r)$ with $n-r\ge\gamma_1\ge\gamma_r\ge -n+r$. Then, $M_\gamma$
is reflexive iff $\gamma_1\leq 1$, in which case it is MCM.
\end{corollary}
\begin{proof}
For simplicity, denote the subset of  $\Sym_2E$ consisting of matrices of rank $\leq i$ by $Y_i\subseteq \Sym_2E$. In particular, $\Spec R^s=Y_r$. 
We have $\dim Y_i=\frac{i(i+1)}{2}+i(n-i)=\frac{1}{2}(2ni-i^2+i)$. Hence, $\codim_{Y_i}Y_r=(n-r)(r-i)+\frac{(r-i)(r-i+1)}{2}$, and $\codim_{Y_{r-1}}Y_r=n-r-1$.
By Proposition~\ref{prop:main}, $M_\gamma$ is reflexive iff $H^{n-r}(\cZ^s,\sHom(\calM_\gamma,\omega_Z))$ is supported on $Y_{r-2}$ (i.e., vanishes when localized at $Y_{r-1}$).  However,  if $\gamma_1\geq2$, then \[H^{n-r}(\cZ^s,\sHom(\calM_\gamma,\omega_Z))=H^{n-r}(\Grass(n-r,n),S_\gamma \cQ^*\otimes (\wedge^r \cQ^*)^{n-r-1}\otimes \Sym(\Sym_2 \cQ))\]
has support exactly $Y_{r-1}$.

The fact that $M_\gamma$ is MCM when $\gamma_1\leq 1$ is proved using the same argument as Lemma~\ref{lem:sym_mcm}.
\end{proof}

\begin{corollary}
\begin{enumerate}
\item When $r<n-1$, the $R^s$-algebra $\End_{\cZ^s}(p'^*\TIL_0)$ is not reflexive; in particular, it is not the endomorphism algebra of any reflexive $R^s$-module. 
\item The $R^s$-module $\End_{\cZ^s}(p'^*\TIL_0)$ is torsion-free.
\item The map $\End_{\cZ^s}(p'^*\TIL_0)\to \End_R(q'_*p'^*\TIL_0)$ is the embedding of the $R^s$-module $\End_{\cZ^s}(p'^*\TIL_0)$ into its reflexive envelope.
\end{enumerate}
\end{corollary}
Similarly, 
let $\Spec R^a$ be the space of $n\times n$-skew-symmetric matrices of rank $\leq r$, and let $\cZ^a$ be the total space of $\wedge^2\cQ^*$.
Let $\calM_\gamma$ be $p^*S_\gamma\cQ$ on $\cZ^s$ for any partition $\gamma=(\gamma_1,\ldots,\gamma_r)$ with $n-r\ge\gamma_1\ge\gamma_r\ge -n+r$.
\begin{corollary}
\begin{enumerate}
\item For any $r$, the $R^a$-algebra $\End_{\cZ^a}(p'^*\TIL_0)$ is not reflexive; in particular, it is not the endomorphism algebra of any reflexive $R^a$-module. 
\item The $R^a$-module $\End_{\cZ^a}(p'^*\TIL_0)$ is torsion-free.
\item The map $\End_{\cZ^a}(p'^*\TIL_0)\to \End_R(q'_*p'^*\TIL_0)$ is the embedding of $\End_{\cZ^a}(p'^*\TIL_0)$ into its reflexive envelope.
\end{enumerate}
\end{corollary}
\begin{proof}
Denote the subset of  $\wedge^2E$ consisting of matrices of rank $\leq i$ by $Y_i\subseteq \wedge^2E$. In particular,  $Y_r=\Spec R^a$.  We have $\dim Y_i=\frac{i(2n-i-1)}{2}$, and hence $\codim_{Y_{r-2}}Y_r=2(n-r)+1$. 

Let $\gamma$ be a partition so that $\gamma_1,\gamma_2\geq 1$. Then, \[H^{2(n-r)}(\cZ^a,\sHom(\calM_\gamma,\omega_Z))=H^{2(n-r)}(\Grass(n-r,n),S_\gamma \cQ^*\otimes (\wedge^r \cQ^*)^{n-r+1}\otimes \Sym(\wedge^2 \cQ))\]
has support exactly $Y_{r-2}$.
\end{proof}
\section{(Non-)singularity of an endomorphism algebra}\label{sec:NCCR for r=1}

\subsection{Reflexive equivalences}
In this section, by modules we mean finite generated modules.
Let $R$ is a commutative ring. For any $R$-algebra $\Lambda$ and a $\Lambda$-module $M$, we say $M$ is a reflexive $\Lambda$-module if it is reflexive as an $R$-module. 

The additive category of reflexive $\Lambda$-modules is denoted by $\Ref \Lambda$; the additive category of projective $\Lambda$-modules is denoted by $\proj \Lambda$; the additive category generated by an $R$-module $M$ is denoted by $\add M$.

\begin{lemma}\label{lem:ref_eq}
Suppose $R$ is a commutative ring and $M\in  R\hbox{-}mod$.

\begin{enumerate}
\item The functor $\Hom_R(M,-):R\hbox{-}mod \to   \End_R(M)\hbox{-}mod$ restricts to an equivalence
$\Hom_R(M,-):\add M \cong  \proj \End_R(M)$.
\item If $R$ is a normal domain and $M\in  \Ref R$ is non-zero, then we have an equivalence
$\Hom_R(M,-):\Ref R \cong \Ref \End_R(M)$.
\end{enumerate}
\end{lemma}
Both are well-known  (see e.g. \cite[II.2.1]{ARS} for (1), and  \cite[1.2]{RV} for (2)).

\begin{lemma}\label{lem:codim1_Iso}
Let $R$ be a commutative normal domain,  $M\in  \Ref R$, and $\Lambda_0=\End_R(M)$.
Let $N$ and $T$ be two reflexive $\Lambda_0$-modules and $f:N\to T$ an $R$-module morphism. Then $f$ is  an isomorphism of $\Lambda_0$-modules iff $f$ is an isomorphism  of $ \Lambda_{0 \mathfrak{p}}$-modules, when $f$ is localized at $\mathfrak{p}$,  for every codimension-1 prime $\mathfrak{p}\in \Spec R$.
\end{lemma}
\begin{proof}
Clearly $f$ is an isomorphism of $R$-modules, since $N=\cap_{\mathfrak{p}}N_{\mathfrak{p}}\to \cap_{\mathfrak{p}}T_{\mathfrak{p}}=T$, where the intersection is over all the codimension-1 primes $\mathfrak{p}\in \Spec R$. Also, since $M\in  \Ref R$, we have $\Lambda_0\subseteq \Lambda_{0\mathfrak{p}}$, which in turn implies that $f$ is a morphism of $\Lambda_0$-modules. Therefore, it is an isomorphism of $\Lambda_0$-modules.
\end{proof}

\subsection{Smoothness of $\End_R(M)$}
Now let $M=\oplus_{i\in I}M_i$ be a reflexive $R$-module, and let $\Lambda_0\cong \End_R(M)$.
The indecomposable projectives over $\Lambda_0$ are $\widetilde{P_i}=\Hom_R(M,M_i)$. By Lemma~\ref{lem:ref_eq}, $\Hom_{\Lambda_0}(\widetilde{P_i},\widetilde{P_j})\cong \Hom_R(M_i,M_j)$ and $\End_{\Lambda_0}(\oplus_{i\in I}\widetilde{P_i})\cong\Lambda_0$.

Let $\Lambda$ be an $R$-algebra, finitely generated as an $R$-module. Let $\{P_i\}_{i\in I}$ be a complete set of pairwise distinct indecomposable projective $\Lambda$-modules. Let $J\subseteq I$ be the  subset such that   $j\in J$ iff $P_j$ is reflexive as $R$-module, i.e., $P_j^{**}\cong P_j$. Here $^*$ is $\Hom_R(-,R)$.

\begin{assumption}\label{Assum:pair}
Assume that for any $i\in I$, there is a $j_i\in J$ such that $P_i^{**}\cong P_{j_i}$. 
Assume there is an algebra isomorphism $\Lambda_0\cong \Lambda^{**}$, in particular, $\Hom_R(M,M_i)\cong P_i^{**}$.
\end{assumption}
Under this assumption, let $\phi: \Lambda\to \Lambda_0$ be the composition $\Lambda\to\Lambda^{**}$ and the isomorphism $\Lambda_0\cong \Lambda^{**}$ from Assumption~\ref{Assum:pair}. For any $\Lambda_0$-module $M$, the $\Lambda$-module obtained via restriction of scalars through $\phi$ is denoted by $\Res M$. 

\begin{example}\label{exam:sym_rank1}
Let $\Spec R$ be the space of $n\times n$-symmetric matrices of rank $\leq r=1$. We have the tautological sequence \[0 \to \calR \to E\otimes\calO_{\Grass(n-1, E)}\to   \calQ \to 0\]
with $\rank \calR = n - 1$ and $ \rank \calQ = 1$. Let $Z$ be the total space of $\Sym_2\calQ^*
\cong \calO(-2)$. Let $\calM_i$ be $p^*\calQ^{i}$ on $Z$ for any integer $i$. The tilting bundle is $\TIL=\oplus_{i=0,\dots,n-1}\calM_{-i}$ (notice of the negative sign, in order for the global section to be MCM). 

Let $\Lambda=\End_Z(\TIL)$. The indecomposable projective $\Lambda$ are $\{P_{-i}:=\Hom_Z(\TIL,\calM_{-i})\}_{i=0,\dots,n-1}$. 
Let $M_i=H^0(Z,\calM_i)$ for any integer $i$. 
Then $M_i$ is MCM iff $i\leq 1$, in which case $M_i\cong M_j$ iff $i=j\mod 2$ (up-to a degree shifting).
We have \[P_{-i}\cong \oplus_{j=0,\dots,n-1}H^0(Z,\calM_{j-i})\cong \oplus_{j=-i,-i+1,\dots,n-1-i}M_{j}.\]
In particular, $P_{-i}$ is reflexive iff $i\geq n-2$, i.e., $i=n-2$ or $i=n-1$.

Let $M:=H^0(Z,\TIL)=\oplus_{i=0,\dots,n-1}M_{-i}$. The projective modules over $\Lambda_0=\End_R(M)$ are 
\[\widetilde{P_{-i}}:=\Hom_R(M,M_{-i})\] for $i=0,\dots,n-1$.
As $R$-modules, $\widetilde{P_{-i}}\cong \widetilde{P_{-j}}$ iff $i=j\mod 2$. 

Also, by the same argument as in Proposition~\ref{prop:ref}, $P_{-i}$ is reflexive iff $P_{-i}\cong\Hom_R(M,M_{-i})$, which in turn is equivalent to $i=n-2$ or $i=n-1$. (If $P_{-i}$ is reflexive, then the argument in {\it loc. cit.} tells us that $P_{-i}\cong\Hom_R(M,M_{-i})$; other-wise, since $\Hom_R(M,M_{-i})$ is reflexive, it could not be isomorphic to $P_{-i}$ which is not reflexive. In fact, for $i<i-2$, $P_{-i}$ contains $M_j$ as a direct summand of $R$-module for $j>1$, while every direct summand of $\Hom_R(M,M_{-i})$ is either isomorphic to $M_0$ or $M_1$.)

Therefore, putting all the above together,  Assumption~\ref{Assum:pair} is satisfied in this example.
\end{example}

\begin{example}
Take an even more special example, i.e., $n=3$ in Example~\ref{exam:sym_rank1}. Then $\End_Z(\TIL)$ has matrix presentation
\[\left(\begin{matrix}M_0&M_1[1]&M_0[2]\\M_1&M_0&M_1[1]\\M_2&M_1&M_0.\end{matrix}\right),\] while $\End_R(M)$ has matrix presentation 
\[\left(\begin{matrix}M_0&M_1[1]&M_0[2]\\M_1&M_0&M_1[1]\\M_0[-1]&M_1&M_0.\end{matrix}\right).\]
\end{example}

\begin{lemma}\label{lem:ref_proj}
Assume $\Lambda$ contains $R$.
Let $\Lambda$ be an $R$-algebra and $M$ an $R$-module, satisfying Assumption~\ref{Assum:pair}.
Let $N$ be a reflexive $\Lambda_0$-module, such that $\Res N$ is projective as $\Lambda$-module. Then, $N$ is projective as $\Lambda_0$-module.
\end{lemma}
\begin{proof}
Assume there is an isomorphism of $\Lambda$-modules $f:N\to T$, where $T$ is projective. Note that $\Lambda$ contains $R$ and hence $f$ is an isomorphism as $R$-modules. In particular, $T$ is reflexive, and hence a direct sum of $P_j$'s for $j\in J$. As $P_j$'s for $j\in J$ are naturally projective $\Lambda_0$ modules,  $T$ also has a natural $\Lambda_0$-modules structure, which is projective.

For any codimension-1 prime $\mathfrak{p}\in\Spec R$, the natural map $\phi: \Lambda\to \Lambda_0$ induces an isomorphism of the localization $\Lambda_{\mathfrak{p}}\to \Lambda_{0\mathfrak{p}}$. Therefore, $N_{\mathfrak{p}}$ and $T_{\mathfrak{p}}$ are isomorphic  as $\Lambda_{0\mathfrak{p}}$-modules. By Lemma~\ref{lem:codim1_Iso}, we have that $N$ is isomorphic to $T$ as $\Lambda_0$-modules. In particular, $N$ is projective as a $\Lambda_0$-module.
\end{proof}

The following is an analogue of \cite[Theorem~5.4]{IWmm}.

\begin{prop}\label{prop:fgldim}
Let $R$ be a CM commutative normal domain. 
Let $\Lambda$ be an $R$-algebra and $M$ an $R$-module, satisfying Assumption~\ref{Assum:pair}, and $\Lambda$ has finite global dimension.
If $R\in \add M$, and $\Lambda_0$ is a MCM as an $R$-module, then $\Lambda_0$ has finite global dimension. In particular, $\Lambda_0$ is an NCCR of $R$.
\end{prop}
\begin{proof}
Let $d=\max\{\dim R,\gl\dim \Lambda\}$. Let $Y$ be any $\Lambda_0$-module, and let \[P^{d-1}\to P^{d-2}\to \cdots\to P^0\to Y\] be a $\Lambda_0$-projective resolution of $Y$. (Here we do not ask $P^{d-1}\to P^{d-2}$ to be injective.)

Lemma~\ref{lem:ref_eq} shows that there is an exact sequence \[N^{d-1}\to N^{d-2}\to \cdots\to N^0\] in $\add M\subseteq R\hbox{-}mod$, applying $\Hom_R(M,-)$ to which gives the exact sequence 
\[P^{d-1}\to P^{d-2}\to \cdots\to P^0.\] Let $K$ be the kernel of the map $N^{d-1}\to N^{d-2}$.
Then \[0\to \Hom_R(M,K)\to P^{d-1}\to P^{d-2}\to \cdots\to P^0\to Y\] is exact. 
Calculation of depth ($\Ext^i_R(\Hom_R(M,K),\omega_R)\cong\Ext^{i+d}_R(Y,\omega_R)=0$ for $i>0$) shows that $\Hom_R(M,K)$ is MCM, in particular, reflexive. 
As restricting to $\Lambda$-mod is an exact functor, $\Hom_R(M,K)$ is also the kernel of $P^{d-1}\to P^{d-2}$ as $\Lambda$-modules. 
The choice that $d\geq \gl\dim \Lambda$ makes sure that $\Hom_R(M,K)$ is projective as $\Lambda$-module. Therefore, Lemma~\ref{lem:ref_proj} yields that $\Hom_R(M,K)$ is also projective as $\Lambda_0$-module. Hence, $\Lambda_0$ has global dimension no more than $d$.
\end{proof}

\begin{example}
Let $\Lambda$ and $M$ be as in Example~\ref{exam:sym_rank1}. We know  $\Lambda$ has finite global dimension but not MCM (not even reflexive), and $\Lambda_0$ is MCM. By Proposition~\ref{prop:fgldim}, we also know that $\Lambda_0$ has finite global dimension, and hence, $\Lambda_0$ is a NCCR of $\Spec R$.
\end{example}

\section{Category of modules over a noncommutative desingularization}
\label{sec:Cat}
In this section we discuss various properties of the category of modules over $\End_Z(\TIL)$, for tilting bundles over spaces $Z$ that are commutative desingularizations of orbit closures for actions of the reductive groups.

\subsection{Irreducible objects}
Let $G$ be a split reductive group, $P<G$ a parabolic subgroup. 
Suppose
$p':Z\to G/P$ is a vector bundle. Let $u:G/P\to Z$
be the zero section. Let $\Delta(G/P)=\{\Delta_\alpha|\alpha\in I\}$ be a full exceptional collection over $G/P$ and $\nabla(G/P)=\{\nabla_\alpha\mid\alpha\in I\}$ be the dual collection. Let $\Phi(G/P)$ and $\Sigma(G/P)$ be as in Proposition~\ref{prop: standardizable}. We assume $\Ext^i(\Phi,\Delta_\lambda)=0$ for all $i>0$ and $\lambda\in I$.
\par
We will be particularly interested in the case when $G/P=\Grass$, and $\Delta(G/P)$ is $\Delta(\Grass)=\{\SCH_\lambda\cQ^*\mid\lambda\in B_{r,n-r}\}$ and $\nabla(G/P)$ is $\nabla(\Grass)=\SCH_{(n-r)^r}\cQ^*\otimes \SCH_{(\alpha^c)'}\cR[(n-r)r-|\alpha|]$. 
In particular,  $\TIL=\oplus_{\alpha\in B_{r,n-r}}\Phi_\alpha$ is a multiplicity-free tilting bundle, which is the same as $\TIL_0$ up to multiplicity, and specializes to $\TIL_K$ when $k=\CC$.

\begin{lemma}\label{lem: dual_coll}
Let $\TIL=\oplus_\alpha\Phi_\alpha$ over $G/P$. Suppose $p'^*\TIL$ is a tilting bundle over $Z$. Then, $\Ext^t_{\sO_Z}(p'^*\Phi_\alpha,u_*\Sigma_\beta)$ vanishes unless
$t=0$ and $\alpha=\beta$ in which case it is 1-dimensional.
\par
In particular, $R\Hom_{\sO_Z}(p'^*\TIL,u_*\Sigma_\beta)=\Hom_{\sO_Z}(p'^*\Phi_\beta,u_*\Sigma_\beta)=k$.
\end{lemma}
\begin{proof}
Note that $\Ext^t_{\sO_Z}(p'^*\Phi_\alpha,u_*\Sigma_\beta)=\Ext_{\sO_{G/P}}^t(\Phi_\alpha,\Sigma_\beta)$. The conclusion then follows from the facts (1) and (2) of Definition \ref{def: dual excep} about dual exceptional collections
listed above.
\end{proof}

\begin{definition}
A pair $(\Phi(G/P),\Sigma(G/P))$ of collections of objects in $D^b(Coh(G/P))$ is called \textit{an equivariant dual pair} if $\Phi(G/P)=\{\Phi_\alpha\mid\alpha\in I\}$ is a  collection  of equivariant sheaves over $G/P$, and $\Sigma(G/P)$  is another collection in $D^b(G/P)$, also equivariant, such that $\Ext^*(\Phi_\alpha,\Sigma_\beta)=\delta_\alpha^\beta k$ where $k$ lies in degree zero.
\end{definition}

For an exceptional collection $\Delta(G/P)$ consisting of equivariant sheaves and its dual collection $\nabla(G/P)$ which is also equivariant, the pairs $(\Delta(G/P),\nabla(G/P))$ and $(\Phi(G/P),\Sigma(G/P))$ are both equivariant dual pairs.

\begin{prop}\label{prop: equi-simple}
Let $\Phi(G/P)$ and $\Sigma(G/P)$ be an equivariant dual pair such that $\TIL=\oplus_\alpha\Phi_\alpha$ is a tilting bundle. Assume the resolution $q:Z\to \Spec R$ is $G$-equivariant with $q^{-1}(0)=G/P$, and the only $G$-fixed closed point of $\Spec R$ is $0\in\Spec R$.  Assume moreover that $p'^*\TIL$ is a tilting bundle over $Z$ such that $\End_{\sO_Z}(p'^*\TIL)\cong\End_R(q_*p'^*\TIL)$.
\par
Then, the only $G$-equivariant simple $\End_{\sO_Z}(p'^*(\oplus_\alpha\Phi_\alpha))$ modules are given by $$S_\beta:=R\Hom_{\sO_Z}(p'^*(\oplus_\alpha\Phi_\alpha),u_*\Sigma_\beta)\cong\Hom_{\sO_Z}(p'^*\Phi_\beta,u_*\Sigma_\beta)$$ and they are pairwise distinct as $\End_{\sO_Z}(p'^*(\oplus_\alpha\Phi_\alpha))$-modules.
\end{prop}
\begin{proof}
First we show that the simple equivariant $\End_{\sO_Z}(p'^*(\oplus_\alpha\Phi_\alpha))$ modules have to be scheme theoretically supported on $i:\{0\}\inj \Spec R$. Let $M$ be a simple equivariant $\End_{\sO_Z}(p'^*(\oplus_\alpha\Phi_\alpha))$ module whose support contains $\{0\}$. Then $M$ is an $R$-module. The surjective morphism $M\surj i_*i^*M$ has non-trivial target, because $\{0\}$ is a subset of the support of $M$. Note that his map is an $G$-equivariant morphism of $\End_{\sO_Z}(p'^*(\oplus_\alpha\Phi_\alpha))$-modules. It has trivial kernel since $M$ is simple. Thus, $M\cong i_*i^*M$, i.e., $M$ has (scheme theoretical) support on $\{0\}$.
\par
Now we show that $M$, with support on $\{0\}$, is a simple module over $$i^*\End_{\sO_Z}(p'^*(\oplus_\alpha\Phi_\alpha))\cong i^*q_*\sEnd_{\sO_Z}(p'^*(\oplus_\alpha\Phi_\alpha)).$$ Using the cartesian diagram
$$
\xymatrixrowsep{15pt}
\xymatrix{
Z\ar[d]_q&G/P\ar@{_{(}->}[l]_u\ar[d]^{q'}\\
\Spec R & \{0\}\ar@{_{(}->}[l]^i
}$$
and Remark \uppercase\expandafter{\romannumeral3}.9.3.1 in \cite{Ha77}, we obtain a map $$i^*q_*\sEnd_{\sO_Z}(p'^*(\oplus_\alpha\Phi_\alpha))\to q_*'u^*\sEnd_{\sO_Z}(p'^*(\oplus_\alpha\Phi_\alpha)).$$ Note that $i^*q_*\sEnd_{\sO_Z}(p'^*(\oplus_\alpha\Phi_\alpha))$, as the inverse image through a proper morphism, is a finite dimensional algebra. There is a natural grading on it giving by the weights of $\mathbb{G}_m$-action, where $\mathbb{G}_m$ acts on $Z$ by scaling. The 0-th degree piece of $i^*q_*\sEnd_{\sO_Z}(p'^*(\oplus_\alpha\Phi_\alpha))$ is $\End_{G/P}(\oplus_\alpha\Phi_\alpha)$. Every element in the first degree piece can be easily checked to be in the Jacobson radical of $i^*q_*\sEnd_{\sO_Z}(p'^*(\oplus_\alpha\Phi_\alpha))$ using finite dimensionality. Therefore, we get an isomorphism $$ i^*\End_{\sO_Z}(p'^*(\oplus_\alpha\Phi_\alpha))/\rad\cong \End_{G/P}(\oplus_\alpha\Phi_\alpha)/\rad.$$

\par
So far we know $M$ is a simple module over $\End_{G/P}(\oplus_\alpha\Phi_\alpha)$. As $\Phi$ is a full exceptional collection over $G/P$, its endomorphism ring is a finite dimensional basic algebra, with all the simple modules of the form $\Hom_{G/P}(\oplus_\alpha\Phi_\alpha,\Sigma_\beta)\cong \Hom_{Z}(p'^*(\oplus_\alpha\Phi_\alpha),u_*\Sigma_\beta)$. Thus, $M$ has to be isomorphic to one of them. (Note that the Jacobson radical acts trivially on the simple objects. Thus, two simple modules are isomorphic over $\End_{G/P}(\oplus_\alpha\Phi_\alpha)$ iff they are isomorphic over $\End_{\sO_Z}(p'^*(\oplus_\alpha\Phi_\alpha))$.)
\par
For the claim that $S_\alpha\neq S_\beta$ unless $\alpha=\beta$, note that $R\Hom_{\sO_Z}(p'^*(\oplus_\alpha\Phi_\alpha),-)$ induces an equivalence of derived categories. To show $S_\alpha\neq S_\beta$, it suffices to show $u_*\Sigma_\alpha\neq u_*\Sigma_\beta$, and this is clear by Lemma~\ref{lem: dual_coll}.
\end{proof}
\begin{remark}\label{rmk:equ_simple}
As can be seen from the proof, the finite dimensional algebra $\End_{\sO_{G/P}}(\oplus_\alpha\Phi_\alpha)$ is a subring of $\End_{\sO_Z}(p'^*(\oplus_\alpha\Phi_\alpha))$. The modules $S_\alpha$ are also simple modules over $\End_{\sO_{G/P}}(\oplus_\alpha\Phi_\alpha)$. Actually, it is easier to see that $S_\alpha$'s are distinct as modules over $\End_{\sO_{G/P}}(\oplus_\alpha\Phi_\alpha)$.
\end{remark}
\begin{remark}
In Proposition~\ref{prop: equi-simple} we only characterized all the simple $\End_{\sO_Z}(p'^*\TIL)$-modules which happen to admit equivariance structure. There could be more simple objects in the abelian category of $G$-equivariant $\End_{\sO_Z}(p'^*\TIL)$-modules.
\end{remark}

As $R\Hom(p'^*\TIL,p'^*\Phi_\alpha)$ could be different than the projective covers of $\SCH_\alpha$, from now on we make the convention that by $P_\alpha$ we mean $R\Hom(p'^*\TIL,p'^*\Phi_\alpha)$.

\begin{definition}
For an algebra with a rational $G$-action, and any $G$-equivariant module $M$, we define $\rad_GM$ to be the intersection of all the $G$-equivariant maximal submodules of $M$.
\end{definition}
\textbf{Caution:} Note that $\rad_GM$ is not the intersection of all the maximal subobjects of $M$ in the category of $G$-equivariant modules.

\begin{lemma}\label{lem: radG}
Assume $\Lambda$ is a $k$-algebra with a rational $G$-action such that $\Lambda/\rad_G\Lambda$ is semi-simple. For any equivariant module $M$ over $\Lambda=\End_{\sO_Z}(p'^*\TIL)$, we have $\rad_GM= \rad_G\Lambda\cdot M$.
\end{lemma}
\begin{proof}
To show the inclusion $\rad_GM\supseteq \rad_G\Lambda\cdot M$, it suffices to show that $\rad_G\Lambda$ is in the kernel of the map $m\cdot:\Lambda\to M/N$ (multiplication by $m$) for any maximal submodule $N$ and any element $m\in M$. This is true since $M/N$ is simple.
\par
Because of the semi-simplicity of $\Lambda/\rad_G\Lambda$, the $\Lambda/\rad_G\Lambda$-module $M/\rad_GM$ decomposes into direct sum of equivariant simple objects. And it is clear that $\rad_GM$ is the maximal submodule with this property. Hence we get the reverse inclusion.
\end{proof}

\subsection{Quiver and relations for the non-commutative desingularization}

\begin{theorem}\label{thm: quiver_gen_rel}
Let $\Phi(G/P)$ and $\Sigma(G/P)$ be an equivariant dual pair such that $\TIL=\oplus_\alpha\Phi_\alpha$ is a tilting bundle. Assume the resolution $Z\to \Spec R$ is $G$-equivariant with $q^{-1}(0)=G/P$, and the only fixed closed point of $\Spec R$ is $\{0\}\subset\Spec R$. Let $\Lambda:=\End_{\sO_Z}(p'^*\TIL)$. Assume moreover that $p'^*(\oplus_\alpha\Phi_\alpha)$ is a tilting bundle over $Z$ such that $\Lambda\cong\End_R(q_*p'^*\TIL)$. Then,
\begin{enumerate}
  \item we have $e_\alpha(\rad_G \Lambda/\rad_G^2\Lambda)e_\beta\cong\Ext^1_\Lambda(S_\alpha,S_\beta)^*$. In particular, a lifting of a basis of this vector space to $\rad_G \Lambda$ generates $\Lambda$ over $\oplus_\alpha k_\alpha$.
  \item With generators of $\End_{\sO_Z}(p'^*\TIL)$ chosen as above, $\Ext^2_{\sO_Z}(\Sigma_i,\Sigma_j)^*$ generates the relations.
\end{enumerate}
\end{theorem}

\begin{proof}[Proof of (1)]
We know that $S_\alpha$'s are all the simple $\Lambda$-modules which are $G$-equivariant. As $\Hom_\Lambda(P_\beta,S_\alpha)=0$ unless $\alpha=\beta$, in which case $\Hom_\Lambda(P_\beta,S_\alpha)=k$ and $P_\alpha/\rad_GP_\alpha\cong S_\alpha$. Note that this also implies $\Lambda/\rad_G\Lambda\cong\oplus k_\alpha$ which is semi-simple.
\par
We take the short exact sequence $$0\to \rad_GP_\alpha\to P_\alpha\to S_\alpha\to 0,$$ with $P_\alpha=R\Hom(p'^*\TIL,p'^*\Phi_\alpha)$. Hence we get an exact sequence$$\xymatrix@C=12pt{
0\ar[r]&\Hom(S_\alpha,S_\beta)\ar[r]&\Hom(P_\alpha,S_\beta)\ar[r]^-0&\Hom(\rad_GP_\alpha,S_\beta)\ar[r]&\Ext^1(S_\alpha,S_\beta)\ar[r]&0.
}$$
Hence
\begin{eqnarray*}
\Ext^1_\Lambda(S_\alpha,S_\beta) &\cong& \Hom_\Lambda(\rad_GP_\alpha,S_\beta)\\
&\cong&\Hom_\Lambda(\rad_G P_\alpha/\rad_G^2P_\alpha,S_\beta)\\
&\cong&\Hom_\Lambda( \Lambda e_\beta,\rad_G P_\alpha/\rad_G^2P_\alpha)^*\\
&\cong&\Hom_\Lambda( \Lambda e_\beta,(\rad_G \Lambda/\rad_G^2\Lambda)e_\alpha)^*\\
&\cong&e_\alpha(\rad_G \Lambda/\rad_G^2\Lambda)e_\beta^*.\\
\end{eqnarray*}
The claim that $\rad_G \Lambda$ generates $\Lambda$ over $\oplus_\alpha k_\alpha$ is clear from Lemma~\ref{lem: radG}.
\end{proof}

The second part of this theorem follows directly from the next Lemma, which gives an equivariant projective resolution of $S_\alpha$ and is interesting in its own right.
\begin{lemma}\label{lem:proj_res}
There is a projective resolution of $S_\alpha$ of the form
$$\xymatrix@C=12pt{
0&S_\alpha\ar[l]&P_\alpha\ar[l]&\oplus_\beta \Ext^1(S_\alpha, S_\beta)^*\otimes P_\beta\ar[l]&\oplus_\beta \Ext^2(S_\alpha, S_\beta)^*\otimes P_\beta\ar[l]&\cdots\ar[l]
}.$$
\end{lemma}
\begin{proof}
We start with a surjective map $S_\alpha\twoheadleftarrow P_\alpha$ whose kernel is $\rad_GP_\alpha$ which has a rational $G$-action, and therefore each vector lies in some finite dimensional subrepresentation. This implies the existence of a collection of finite dimensional representations $\{V^1_\gamma\}_\gamma$ of $G$ equipped with an equivariant surjective map $\rad_GP_\alpha\twoheadleftarrow\oplus_\gamma V^1_\gamma\otimes P_\gamma$ which fits into an exact sequence $$0\leftarrow S_\alpha\leftarrow P_\alpha\leftarrow \oplus_\gamma V^1_\gamma\otimes P_\gamma.$$ We take the kernel and proceed to get an equivariant projective resolution of $S_\alpha$ with $i$-th term of the form $\oplus_\gamma V^i_\gamma\otimes P_\gamma$.
\par
We apply $\Hom_\Lambda(-,S_\beta)$ to get $$0\rightarrow \delta_\alpha^\beta k\rightarrow \delta_\alpha^\beta k\rightarrow V^{1*}_\beta\rightarrow V^{2*}_\beta\rightarrow\cdots,$$ which is a chain complex of $G$-representations. Hence, its $i$-th homology can be identified with a subquotient of the $i$-th term. In particular, (replacing $V^1_\beta$ with $(\ker d_2)^*$ if necessary,) we obtain
$$
\xymatrixrowsep{15pt}
\xymatrix{
&&0&\coker\ar[l]\\
0&S_\alpha\ar[l]&P_\alpha\ar[l]\ar[u]&\oplus_\gamma V^1_\gamma\otimes P_\gamma\ar[l]\ar@{->>}[u]^{\pi_1}&\oplus_\gamma V^2_\gamma\otimes P_\gamma\ar[l]^{d_2}.
}$$
\par
We claim that the composition $d_2\circ\pi_1:\oplus_\gamma V^2_\gamma\otimes P_\gamma\to\coker$ is surjective. If so, we can replace $V^1_\gamma\otimes P_\gamma$ by $\Ext^1(S_\alpha, S_\gamma)\otimes P_\gamma$, and $\oplus_\gamma V^2_\gamma\otimes P_\gamma$ by $d_2^{-1}(\Ext^1(S_\alpha, S_\gamma)\otimes P_\gamma)$, and proceed iteratively to get the desired resolution.
\par
To prove the surjectivity of $d_2\circ\pi_1$, we can assume $\coker=V^{1'}_\gamma\otimes P_\gamma\neq 0$, hence $V^{1'}_\gamma\neq 0$ for some $\gamma$. We show that $V^{1'}_\gamma\otimes P_\gamma$ is in the image of $d_2\circ\pi_1$. Let $W_\gamma:=V^{1'}_\gamma\otimes P_\gamma/V^{1'}_\gamma\otimes P_\gamma\cap im(d_2\circ\pi_1)$. If this is not zero, $M/\rad_GM$ would contribute to $\Ext^1(S_\alpha,S_\beta)$ which makes it larger than it actually is. Hence, we are done.
\end{proof}
\begin{remark}
Note that the complexes $u_*\Sigma_\alpha$'s are not in the heart of the usual $t$-structure in $D^b(Coh(Z))$, but $R\Hom(p'^*\TIL,u_*\Sigma_\alpha)$'s are for the usual $t$-structure of $D^b(\Lambda\hbox{-}mod)$. This means the functor $R\Hom(p'^*\TIL,-)$ does not restrict to an equivalence $Coh(Z)\to \Lambda\hbox{-}mod$.
\end{remark}

\begin{remark}
This derived equivalence gives the triangulated category $D^b(\Lambda\hbox{-}mod)$ a $t$-structure, by lifting the tautological $t$-structure of $D^b(\Lambda\hbox{-}mod)$ (see e.g. \uppercase\expandafter{\romannumeral4}.4 in \cite{GM} for the definitions of $t$-structures of a triangulated categories and the tautological $t$-structures of derived categories). 
\end{remark}
Note that twisting the tilting bundle with any line bundle will give the same endomorphism ring. Consequently, any two different exceptional collections in the same $H^1(G/P, \sO^*)$-orbit in the set of exceptional collections give the same non-commutative desingularization. But the induced $t$-structures on $D^b(Coh(Z))$ are different.

\subsection{Grassmannians}
Now, let us work in the set-up as in diagram~\ref{diag: general} by taking $G/P$ to be the Grassmannian. Take the equivariant dual pair $\Delta(G/P)=\{\Delta_\alpha|\alpha\in I\}$ and $\nabla(\Grass)=\SCH_{(n-r)^r}\cQ^*\otimes \SCH_{(\alpha^c)'}\cR[(n-r)r-|\alpha|]$. We describes the Ext's between the equivariant simples $S_\beta=R\Hom_{\sO_Z}(p^*\TIL,u_*\Sigma_\beta)$ as in Proposition~\ref{prop: equi-simple}.  The proof of the following lemma is along the same lines as the corresponding one in \cite{BLV}.
\par

\begin{lemma}\label{Lem: Exts}
Let $Z$ be the total space of the vector bundle $\SCH_\delta \cQ^*$ for some partition $\delta$ such that the conditions in Proposition~\ref{prop: equi-simple} are satisfied for the tilting bundle $\TIL_0$. Let $S_\alpha$'s be the simples as in Proposition~\ref{prop: equi-simple}. Then, the Ext's among them are given by
$$\Ext^t(S_\alpha,S_\beta)\cong\oplus_sH^{t-s-|\beta|+|\alpha|}(\Grass,(\wedge^s \SCH_\lambda\cQ)^*\otimes \mathfrak{L}_{\Grass}(L'_{\beta'})\otimes \mathfrak{L}_{\Grass}(L'_{\alpha'})^*).$$
\end{lemma}
\begin{proof}
We have
\begin{eqnarray*}
\Ext^t(S_\alpha,S_\beta)&=& \Ext^t_{\sO_\cZ}(u_*\Sigma_{\alpha},u_*\Sigma_{\beta}) \\
&=& \Ext^{t-|\beta|+|\alpha|}_{\sO_\cZ}(u_*\mathfrak{L}_{\Grass}(L'_{\beta'}),u_*\mathfrak{L}_{\Grass}(L'_{\alpha'})) \\
&=&\oplus_sH^{t-s-|\beta|+|\alpha|}(\Grass,\wedge^s\SCH_\delta\cQ^*\otimes \SCH_{(\beta^c)'}\cR\otimes \SCH_{(\alpha^c)'}\cR^*).
\end{eqnarray*}
\par
The only thing need to explain in the above is the third equality. To get that we take the Koszul resolution of $u_*V:=u_*\SCH_{(n-r)^r}\cQ^*\otimes \SCH_{(\alpha^c)'}\cR$ as follows
$$\cdots\to\wedge^2\SCH_\delta\cQ\otimes V\otimes\sO_\cZ\to \SCH_\delta\cQ\otimes V\otimes\sO_\cZ\to V\otimes\sO_\cZ,$$
apply $\sHom_{\sO_\cZ}(-,u_*U)$, and use adjunction formula to get
$$0\to\sHom_{\Grass}(V,U)\to\sHom_{\Grass}(\SCH_\delta\cQ\otimes V,U)\to\sHom_{\Grass}(\wedge^2\SCH_\delta\cQ\otimes V,U)\to\cdots.$$ Note that the $t$-th hypercohomology of this complex is exactly $\Ext^t(S_\alpha,S_\beta)$. The third equality above is equivalent to the degeneracy of the hypercohomology spectral sequence  $$E_1^{r,s}=H^r(\sHom_{\Grass}(\wedge^s\SCH_\delta\cQ\otimes V,U))\Rightarrow \Ext^{r+s}(S_\alpha,S_\beta)$$ at $E_1$-page.
We claim that the hypercohomology spectral sequence does degenerate at $E_1$-page. In fact, after plugging in $V=\mathfrak{L}_{\Grass}(L'_{\alpha'})$ and $U=\mathfrak{L}_{\Grass}(L'_{\beta'})$, all the differentials are equivariant under $GL_n$ and hence equivariant under the $\Gm$-action. Here the $\Gm\subset\GL_n$ consists of the scalar matrices. Note that the weights of this $\Gm$-action in different columns in the spectral sequence are different. Therefore, there is no non-trivial differentials other than the vertical ones, all of which are in $E_0$.
\end{proof}

\begin{corollary}Assume $k=\mathbb{C}$. Let $Z$ be the total space of the vector bundle $\SCH_\delta \cQ^*$ for some partition $\delta$, and $\nabla(\Grass)=\{\SCH_\lambda\cQ^*\mid\lambda\in B_{r,n-r}\}$ satisfies the conditions in Proposition~\ref{prop: equi-simple}. Let $S_\alpha$'s be the simples as in Proposition~\ref{prop: equi-simple}. Then, the Ext's among them are given by
$$\Ext^t(S_\alpha,S_\beta)\cong\oplus_s\oplus_{\lambda\in \wedge^s\SCH_\delta}H^{t-s-|\beta|+|\alpha|}(\Grass,\SCH_\lambda \cQ^*\otimes \SCH_{\beta'}\cR^*\otimes \SCH_{\alpha'}\cR),$$ where $\wedge^s\SCH_\delta$ stands for the decomposition of $\wedge^s\SCH_\delta\mathbb{C}^{n-r}$ into irreducible representations counting multiplicity.
\end{corollary}
\begin{proof}This is because if $k=\CC$, we have
\begin{eqnarray*}
&&\oplus_sH^{t-s-|\beta|+|\alpha|}(\Grass,\wedge^s\SCH_\delta\cQ^*\otimes \SCH_{(\beta^c)'}\cR\otimes \SCH_{(\alpha^c)'}\cR^*)\\
&=&\oplus_s\oplus_{\lambda\in \wedge^s\SCH_\delta}H^{t-s-|\beta|+|\alpha|}(\Grass,\SCH_\lambda \cQ^*\otimes \SCH_{(\beta^c)'}\cR\otimes \SCH_{(\alpha^c)'}\cR^*)\\
&=&\oplus_s\oplus_{\lambda\in \wedge^s\SCH_\delta)}H^{t-s-|\beta|+|\alpha|}(\Grass,\SCH_\lambda \cQ^*\otimes \SCH_{\beta'}\cR^*\otimes \SCH_{\alpha'}\cR).
\end{eqnarray*}
\end{proof}

\section{Combinatorics from the Kapranov's exceptional collection}\label{subsec: comb_Kap}
In the rest of this section, we do some calculations related to the Kapranov's exceptional collection, which makes it easier to calculate the Ext's. \textit{From now on till the end this section, we assume $k=\mathbb{C}$.} We consider the exceptional collection $\nabla(\Grass)=\{\SCH_\lambda\cQ^*\mid\lambda\in B_{r,n-r}\}$ on $\Grass(r,n)$. 

Let $\alpha,\beta\in  B_{r,n-r}$ be two partitions, and $s,t$ be two integers so that $0\leq t-s-|\beta|+|\alpha|\leq r(n-1)$.
\begin{corollary}\label{cor: ext_delta}
Let $\gamma$ be a weight so that $\SCH_\gamma\cR^*\subset\SCH_{\beta'}\cR^*\otimes\SCH_{\alpha'}\cR$ and $k=t-s-|\gamma|$.
\begin{enumerate}
  \item For a fixed $t$, $H^{k}(\Grass,\SCH_\lambda \cQ^*\otimes \SCH_{\gamma}\cR^*)=0$ for any $s>t$.
  \item For a fixed $t$, $H^{k}(\Grass,\SCH_\lambda \cQ^*\otimes \SCH_{\gamma}\cR^*)=0$ for any $|\delta|s\geq k$.
  \item For a fixed $s$ and $t$, $H^{k}(\Grass,\SCH_\lambda \cQ^*\otimes \SCH_{\gamma}\cR^*)=0$ for any $\gamma=(\gamma_1,\cdots,\gamma_{n-r})$ with the positive area of $\gamma$ greater than $t-s$.
  \item For any $\gamma=(\gamma_1,\cdots,\gamma_{n-r})$, $H^{k}(\Grass,\SCH_\lambda \cQ^*\otimes \SCH_{\gamma}\cR^*)=0$ unless the negative part of $\gamma'$ is contained in $\lambda$.
\end{enumerate}
\end{corollary}
\begin{proof}
According to Bott's Theorem, there can be no more than one $k$ such that $H^k(\Grass,\SCH_\lambda \cQ^*\otimes \SCH_{\gamma}\cR^*)\neq0$, and this $k$ is the number of adjacent transpositions for $(\gamma_1,\cdots,\gamma_{n-r}, \lambda_1,\cdots,\lambda_r)$ to make it dominant. Assume $H^k\neq0$, we know $\gamma_{n-r}\geq -r$, and therefore the total negative area in $(\gamma_1,\cdots,\gamma_{n-r})$ is no larger than $k$. So, $-s+t=|\gamma|+k\geq0$ which proves (1).
\par
Again because $\gamma_{n-r}\geq -r$, the total area of $(\gamma_1,\cdots,\gamma_{n-r}, \lambda_1,\cdots,\lambda_r)$ is positive. Therefore, $|\delta|s\geq k$.
\par
We know that the total negative area in $(\gamma_1,\cdots,\gamma_{n-r})$ is no larger than $k$. Let the positive area of $\gamma$ be $l$. Then we have $-(|\gamma|-l)\leq k$. Hence $l\leq t-s$.
\par
The last part is clear.
\end{proof}
\par
The following examples, which are direct consequences of Corollary~\ref{cor: ext_delta}, will be used in Section~\ref{sec:sym} and Section~\ref{sec: skew}. To state them, we introduce the following notation. For any Young diagram $\alpha$, by $\tau_l\alpha$ we mean $\alpha$ delating the first $l$ columns, and by $\tau^l$ we mean $\alpha$ deleting everything after the first $l$ columns. For any partition $\alpha=(\alpha_1,\cdots,\alpha_l)$, by $(-\alpha)$ we mean a partition with $(-\alpha)_i=-\alpha_{l+1-i}$.
\begin{example}\label{cor: ext_1}
Notation as above. Let $t=1$, then, the only non-zero $H^{k}(\Grass,\SCH_\lambda \cQ^*\otimes \SCH_{\gamma}\cR^*)$ occurs only in the following cases:
\begin{itemize}
  \item $s=1$: $-\gamma'=\lambda$ in which case $H^k(\Grass,\SCH_\lambda \cQ^*\otimes \SCH_{\gamma}\cR^*)=\mathbb{C}$. Or $length(\gamma)=n-r$, and $(\tau^{n-r}\lambda)=(-\gamma)'$. In this case, the corresponding $H^k(\Grass,\SCH_\lambda \cQ^*\otimes \SCH_{\gamma}\cR^*)=\SCH_{\tau_{n-r}\lambda}E$.
  \item $s=0$: $\gamma=(1,0,\cdots,0)$. In this case, the corresponding $H^k(\Grass,\SCH_\lambda \cQ^*\otimes \SCH_{\gamma}\cR^*)=E$.
\end{itemize}
\end{example}
\begin{proof}
\par
According to Corollary~\ref{cor: ext_delta}, we only need to consider the $s=1$ and $s=0$ case. If $s=1$, then part 3 and 4 Corollary~\ref{cor: ext_delta} yield that $-\gamma'\subset\lambda$. We also have $k=-|\gamma|$. Note that $H^k(\Grass,\SCH_\lambda \cQ^*\otimes \SCH_{\gamma}\cR^*)\neq0$ implies that the weight becomes dominant after $k$ exchanges. Therefore, $(\lambda')_i=(-\gamma)_i$ for all $i\leq length(\gamma)$, and the corresponding $H^k(\Grass,\SCH_\lambda \cQ^*\otimes \SCH_{\gamma}\cR^*)=\SCH_{(\lambda_1-(\gamma')_{n-r},\lambda_2-(\gamma')_{n-r-1},\cdots, \lambda_{n-r}-(\gamma')_1)}E$. The statement for $s=0$ is clear.
\end{proof}
Similarly, one has the following example.
\begin{example}\label{cor: ext_2}
Notation as above. Let $t=2$, then, the only non-zero $H^{k}(\Grass,\SCH_\lambda \cQ^*\otimes \SCH_{\gamma}\cR^*)$ occurs only in the following cases.
\begin{itemize}
  \item $s=2$: $-\gamma'=\lambda$ in which case $H^k(\Grass,\SCH_\lambda \cQ^*\otimes \SCH_{\gamma}\cR^*)=\mathbb{C}$. Or $length(\gamma)=n-r$, and $(\tau^{n-r}\lambda)=(-\gamma)'$. In this case, the corresponding $H^k(\Grass,\SCH_\lambda \cQ^*\otimes \SCH_{\gamma}\cR^*)=\SCH_{\tau_{n-r}\lambda}E$.
  \item $s=1$, the positive part of $\gamma$ is $(1,0,\cdots,0)$: The negative part of $\gamma$ is $(-\lambda)$ in which case $H^k(\Grass,\SCH_\lambda \cQ^*\otimes \SCH_{\gamma}\cR^*)=\mathbb{C}$. Or the length of the negative part of $(\gamma)$ is $n-r-1$, $(\tau^{n-r-1}\lambda)=((-\gamma')_1,\cdots,(-\gamma')_{n-r-1})$, and $\tau_{n-r-1}\lambda_1\leq1$, in which the corresponding $H^k(\Grass,\SCH_\lambda \cQ^*\otimes \SCH_{\gamma}\cR^*)=\SCH_{(1,\tau_{n-r-1}\lambda)}E$.
  \item $s=1$, $\gamma$ has no positive part: $length(\gamma)=n-r$, and $((\tau^{n-r}\lambda)')_i=(-\gamma)_i$ for all $i\leq n-r-1$, $((\tau^{n-r}\lambda)')_{n-r}=(-\gamma)_{n-r}+1$, and $\lambda_{-\gamma_1+1}-(-\gamma')_{-\gamma_1+1}\geq2$, in which case the corresponding $H^k(\Grass,\SCH_\lambda \cQ^*\otimes \SCH_{\gamma}\cR^*)=\SCH_{(\tau_{n-r}\lambda_1,\cdots,\tau_{n-r}\lambda_{-\gamma_1},\tau_{n-r}\lambda_{-\gamma_1},1,0,\cdots,0)}E$. Or $\gamma=(0,-1,\cdots,-1)$, $\lambda_2=0$ and $\lambda_1\geq n-r+1$, in which case the corresponding $H^k(\Grass,\SCH_\lambda \cQ^*\otimes \SCH_{\gamma}\cR^*)=\SCH_{(\lambda_1-n+r,1,0,\cdots,0)}E$.
  \item $s=0$: $\gamma=(1,1,0,\cdots,0)$ in which case, the corresponding $H^k(\Grass,\SCH_\lambda \cQ^*\otimes \SCH_{\gamma}\cR^*)=\wedge^2E$. Or $\gamma=(2,0,\cdots,0)$ in which case, the corresponding $H^k(\Grass,\SCH_\lambda \cQ^*\otimes \SCH_{\gamma}\cR^*)=\SCH_2E$.
\end{itemize}
\end{example}
\par

The following lemma and remark give an algorithm to calculate the higher Ext's. We will fix $\lambda$ with $|\lambda|=s|\delta|$, and let $t=k-s-|\gamma|$.
\begin{lemma}\label{lem: ext_algo}
\begin{enumerate}
  \item For a fixed partition $\lambda$ with $l(\lambda)\leq r$, there is a unique dominant $GL_{n-r}$-weight $\gamma$ with $\gamma_{n-r}\geq -r$ such that $|\gamma|$ is minimal and $H^k(\Grass,\SCH_\lambda \cQ^*\otimes \SCH_{\gamma}\cR^*)\neq0$ for some $k$.
  \item Let the minimal $\gamma$ in(1) be $\gamma_{min\lambda}$ and the corresponding $t$ be $t_{min\lambda}$. Every other $\gamma$ with $H^k(\Grass,\SCH_\lambda \cQ^*\otimes \SCH_{\gamma}\cR^*)\neq0$ for some $k$ has $\gamma_k\geq(\gamma_{min\lambda})_k$ and the corresponding $t$ strictly greater then $t_{min\lambda}$.
\end{enumerate}
\end{lemma}
The proof of this lemma shows how to find this minimal $\gamma$, the corresponding $k$ and $t$.
\begin{proof}
We look at $\lambda+(r,r-1,\cdots,1)$. Let $i_0=r$, for $j>0$ let $i_j=\lambda_r+r-j-\lambda_{r-j}$, i.e., $\lambda_{r-j}+j+1=\lambda_r+1+(r-i_j)$. Suppose the largest $j$ with positive $i_j$ is $p$ with the corresponding $i_p=q$. We construct the minimal $\gamma$ as follows.
\par
Start with $k=r-1$ and $l=n-r$. If $k=i_j$ for some $j=1,\cdots,p$, we do nothing for $l$ and decrease $k$ by 1. Otherwise we have $k\neq i_j$ for any $j=1,\cdots,p$, in which case we set $\gamma_l=\lambda_r+(1+r-k)-(n-l+1)$ and then decrease both $k$ and $l$ by 1.  Repeat this process. Stop if we reach $l=0$ or $k=0$.
\par
If the above process stopped with $k=0$ and $l\neq0$, we reset $k=p-1$ and maintain the same $l$. If $\lambda_k+(1+r-k)=\gamma_{l-1}+(n-l+1)+1$, we keep the same $l$ and decrease $k$ by 1. Otherwise we have $\lambda_k+(1+r-k)>\gamma_{l+1}+(n-l+1)+1$, in which case we set $\gamma_l=\lambda_k+(1+r-k)-(n-l+1)+1$ and then decrease both $k$ and $l$ by 1. Repeat this process. Stop if we reach $l=0$ or $k=0$.
\par
Again if the above process stopped with $k=0$ and $l\neq0$, we set $\gamma_l=\gamma_{l+1}$ and decrease $l$ by 1. Repeat this process. Stop if we reach $l=0$.
\par
The number $k$ is the length of the permutation making $(\gamma,\lambda)$ dominant.
\par
The second part is clear from the construction.
\end{proof}
\par
The following remark gives a recursive method to find out all the other $\gamma$'s with some non-vanishing cohomology.
\begin{remark}\label{rmk: ext_algo}
Let the minimal $\gamma$ in (1) be $\gamma_{min\lambda}$ and the corresponding $t$ be $t_{min\lambda}$. Every other $\gamma$ with $H^k(\Grass,\SCH_\lambda \cQ^*\otimes \SCH_{\gamma}\cR^*)\neq0$ for some $k$ can be obtained by a sequence of the following operations $\gamma\mapsto \widehat{\gamma}$.
\par
The operation depends on the parameter $s=1,\cdots,n-r$. Set $\widehat{\gamma}_t=(\gamma_{min\lambda})_t$ for all $t>s$.

Start with $l=s$. We find the largest $j$ with $\max\{\widehat{\gamma}_{l+1}+(n-l)+1, (\gamma_{min\lambda})_l+(n-l+1)+1\}<\lambda_j+(r-j+1)$ and set $\widehat{\gamma}_l=\max\{\lambda_{j+1}+(r-j)-(n-l+1)+1,\widehat{\gamma}_{l+1}+(n-l)+1, (\gamma_{min\lambda})_l+(n-l+1)+1\}$. Then we decrease $l$ by 1, and repeat this process. Stop if we reach $l=0$.
\end{remark}
Note that in the calculation of Ext's we only care about those $\gamma$'s with $\gamma_1\leq r$. We always get all the possible $\gamma$'s with $\gamma_1\leq r$ after finitely many operations above.

\section{Equivariant quivers}\label{sec: equi_quiver}
Here we introduce the basic notions on equivariant quivers and their representations. They provide convenient language for the description of non-commutative desingularizations, especially when considering the equivariant derived categories. Also, as we will see in Subsection~\ref{subsec: BeiKapQuivers}, the derived category of coherent sheaves over homogeneous spaces are easier to describe in this way.
\par
\subsection{The notion of equivariant quivers}

Let $G$ be a reductive group, which in later sections we take to be $GL_n$.
\begin{definition}
An \textit{equivariant quiver} is a triple $Q=(Q_0, Q_1,\Eq)$ where $(Q_0,Q_1)$ is a quiver and $\Eq$ is an assignment associating each arrow $q\in Q_1$ a finite dimensional irreducible representation of $G$.
\end{definition}
When representing an equivariant quiver with a diagram, for $q\in Q_1$ we will simply label the arrow by $q(V)$ in lieu of the equation $\Eq(q)=V$.
\begin{definition}
Let $Q$ be an equivariant quiver. The \textit{path algebra} $kQ$ of $Q$ is the $k$-algebra whose underlying vector space is $$\bigoplus_{(q_1,\dots,q_l)\text{ is a path}} \Eq((q_1,\dots,q_l)),$$
where $\alpha((q_1,\dots,q_l)):=\bigotimes_{i=1}^l \alpha(q_i)$.
We define the multiplication $kQ\otimes kQ\to kQ$
$$\Eq((q_1,\dots,q_l))\otimes\Eq((p_1,\dots,p_h))\to \Eq((r_1\dots,r_s))$$
to be
\begin{equation*}
\left\{
\begin{aligned}
        \hbox{ id } & \hbox{, if the path }(r_1\dots,r_s)=(q_1,\dots,q_l,p_1,\dots,p_h) \\
                  0 & \hbox{, otherwise}
                          \end{aligned} \right.
\end{equation*}
\end{definition}
Obviously, $kQ$ has an action by $G$, and the multiplication is $G$-equivariant.
\par
We will say $Q$ is finite if both $Q_0$ and $Q_1$ are finite sets. We will concentrate on connected finite quivers.
\par
Let $Q$ be a finite quiver, and $I$ be a two sided ideal of $kQ$. We say $(Q,I)$ is an \textit{equivariant quiver with relations} if $I$ is generated by sub-representations of $\Eq(q_1,\dots,q_r)$ for paths $(q_1,\dots,q_r)$. We usually specify a set of such sub-representations as generators of the ideal and call this set relations. We say $I$ is \textit{admissible} if it is generated by sub-representations of $\Eq(q_1,\dots,q_r)$ for paths of length 2 or longer and $I$ contains some power of the arrow ideal. In this case, the pair $(Q,I)$ will be called a bound equivariant quiver, and $kQ/I$ the bound path algebra.
\par
We would like to define two notions of representations of a bound  equivariant quiver, depending on whether or not we will acknowledge $G$-action.
\par
The first notion is a representation which is simply a representation of the quiver obtained by replacing each arrow $q\in Q_1$ by as many arrows as the dimension of $\Eq(q)$.
\begin{definition}
A \textit{representation} of $(Q_0,Q_1,\Eq)$ is an assignment associating to each vertex $a\in Q_0$ a vector space $V_a$, and to each arrow $q:a\to b$ a $k$-linear morphism $V_a\otimes \Eq(q)\to V_b$ (or equivalently $\dim(\Eq(q))$ linear maps $V_a\to V_b$ according to a fixed basis of $\Eq(q)$).
\end{definition}
\par
We define the notion of equivariant representations of an equivariant quiver.
\begin{definition}
Let $Q=(Q_0, Q_1)$ be an equivariant quiver.
An \textit{equivariant representation} of $Q$ is an assignment associating to each vertex $a\in Q_0$ a representation $V_a$ of $G$, and to each arrow $q:a\to b$ a $G$-morphism $V_a\otimes \Eq(q)\to V_b$.
\end{definition}
Let $(Q,I)$ be an equivariant quiver with relations. For a bound representation (resp. bound equivariant representation) we require that the generators of $I$, which we can chose to be sub-representations by definition,  act trivially, i.e., all the morphisms above involving subspaces of $I$ are trivial maps.
\par
If $A$ is a commutative $k$-algebra with a rational $G$-action, we can also talk about bound representations (resp. bound equivariant representation) of $(Q,I)$ over $A$. By this we mean associating to each vertex a projective $A$-module and all the maps have to be $A$-linear. For equivariant representations we require that the projective modules have a rational $G$-action compatible with the $A$-module structure and that all maps corresponding to arrows are $G$-equivariant.
\par
\begin{prop}
The category of bound representations (resp. bound equivariant representations) of $(Q,I)$ (over $k$) is equivalent to the category of modules (resp. $G$-equivariant modules) over the ring $kQ/I$.
\end{prop}
By an equivariant module over $kQ/I$, we mean an $G$-equivariant $kQ/I$ action $kQ/I\otimes M\to M$.

\subsection{Beilinson and Kapranov quivers}\label{subsec: BeiKapQuivers}
According to a result of Beilinson, we have a full exceptional collection
$$\nabla(\mathbb{P}^{n-1}):=\{\Omega^{i-1}(i)\mid i\in[0,n-1] \}$$
 in the derived category of quasi-coherent sheaves over $\mathbb{P}^{n-1}$. Thus, the endomorphism ring of $\oplus_{1=0}^{n-1}\Omega^{i-1}(i)$ is derived equivalent to $\mathbb{P}^{n-1}$, where $\Omega^k$ is the $k$-th exterior power of the sheaf of K\"{a}hler differentials.
\par
Let $E$ be a vector space of dimension $n$ and we take $G=GL_n$ acting naturally on $E$. The Beilinson equivariant quiver,  denoted by $QB(n)$, is defined as follows:

$$\xymatrix{
\bullet_0\ar@/^/@<1ex>[r]^{\alpha_0(E)} &
\bullet_1\ar@/^/@<1ex>[r]^{\alpha_1(E)}&
\cdots\ar@/^/@<1ex>[r]^{\alpha_{n-1}(E)}
 &
\bullet_{n-1} }$$
with relations:

$$\alpha_i\alpha_{i+1}(\wedge^2E).$$
\begin{remark}
Let us pick up a basis for $E$, say, $e_1,\cdots,e_n$. Then, the above quiver, with equivariant structure forgotten, has $n$ arrows going from the $i$-th vertex to the $i+1$-th, denoted by $\alpha_i^1,\cdots,\alpha_i^n$, corresponding to the basis elements of $E$. The relations $\alpha_i\alpha_{i+1}(\wedge^2E)$ can be written as $\alpha_{i+1}^j\alpha_{i}^k-\alpha_{i+1}^k\alpha_i^j$ for all $j$, $k$.
\end{remark}
The bound path algebra of this quiver is isomorphic to $\End_{\sO_{\mathbb{P}^{n-1}}}(\oplus_{1=0}^{n-1}\Omega^{i-1}(i))$, as can be checked by $\Hom_{\sO_{\mathbb{P}(E)}}(\Omega^{a-1}(a),\Omega^{b-1}(b))\cong\wedge^{a-b}(E^*)$ if $a\leq b$ and 0 otherwise.
\par

\par
\begin{remark}
It is not hard to see   that the  category of representations of the Beilinson quiver is isomorphic to the category of graded modules over the exterior algebra \cite{BLV}. Also the same is true for equivariant representations.
\end{remark}

\par
Now let $k=\CC$. We would like to do the same thing for the Grassmannian $\Grass_r(n)$ with $n-r>1$ and call the corresponding equivariant quiver the Kapranov quiver $QK(r,n)$. The exceptional collection we take is
$$\nabla(\Grass_r(n)):=\{\SCH_\alpha\cR^{* },\alpha\in B_{r,n-r}\}.$$
 Let $G=GL_n(\mathbb{C})$ and $E=\mathbb{C}^n$ with the natural $G$ action.
\par
The Kapranov quiver has the set of vertices corresponding to the set of subpartitions of $((n-r)^r)$, and two vertices $\lambda_1$ and $\lambda_2$ are linked by an arrow $\lambda_1\to\lambda_2$ iff $\lambda_2/\lambda_1$ is a single box, and in this case this arrow is associated to the $G$-representation $E$.
\par
For any three partitions $\alpha$, $\beta$, $\mu$, their Littlewood-Richardson coefficient will be denoted by $C^\mu_{\alpha, \beta}$, i.e., the number of Littlewood-Richardson tableaux of shape $\mu / \alpha$ and of weight $\beta$ is $C^\mu_{\alpha, \beta}$. With this notation, the relations in Kapranov quiver are generated by the sub-representations $$C_{\alpha^t,(2,0,\cdots,0)}^{\beta^t}\SCH_2E^*\oplus C_{\alpha^t,(1,1,0,\cdots,0)}^{\beta^t}\wedge^2E$$ of the arrows in $\Hom(\beta,\alpha)$.
\par
The bound path algebra of $QK(r,n)$ is isomorphic to $\End_{\Grass}(\oplus_{\alpha\in B_{r,n-r}}\SCH_\alpha\cR^{* })$.
\par
\begin{example}
We compute the Kapranov equivariant quiver of the Grassmannian $\Grass_2(4)$.
$$
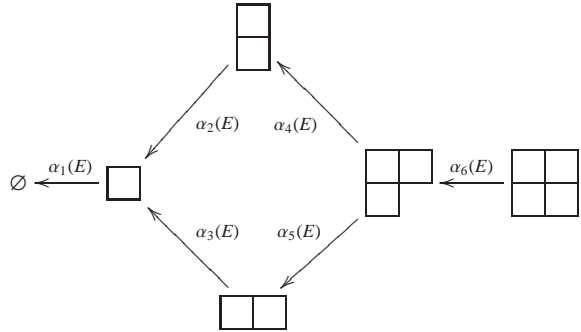
\xymatrix{
 & & \text{$\Young(|)$}\ar[dl]^{\alpha_2(E)} & & \\
 \varnothing & \text{$\Young()$}\ar[l]_{\alpha_1(E)}  & & \text{$\Young(,|)$}\ar[ul]^{\alpha_4(E)}\ar[dl]_{\alpha_5(E)} & \text{$\Young(,|,)$}\ar[l]_{\alpha_6(E)}\\
& & \text{$\Young(,)$}\ar[ul]_{\alpha_3(E)}
}$$
with relations:
\begin{itemize}
  \item $\Hom(\Young(,),\varnothing)$: $\alpha_1\alpha_3(\wedge^2E)$;
  \item $\Hom(\Young(|),\varnothing)$: $\alpha_1\alpha_2(\SCH_2E)$;
  \item $\Hom(\Young(,|,),\Young(,))$: $\alpha_5\alpha_6(\wedge^2E)$;
  \item $\Hom(\Young(,|,),\Young(|))$: $\alpha_4\alpha_6(\SCH_2E)$;
  \item $\Hom(\Young(,|),\Young())$: $\alpha_2\alpha_4-\alpha_3\alpha_5(\SCH_2E\oplus\wedge^2E)$.
\end{itemize}
\end{example}
\begin{prop}\label{prop: kap_equi_der_equ}
The functor $\Phi:=R\Hom_{\sO_{\Grass_r(n)}}(\oplus_{\alpha\in B_{r,n-r}}\SCH_\alpha\cR^{* },-)$ induces an equivalence of triangulated categories \[D^b_G(Coh(X))\cong D^b_G(\mathbb{C}QK(r,n)/I\hbox{-}mod),\] with quasi-inverse given by $\Psi:=-\otimes^L(\oplus_{\alpha\in B_{r,n-r}}\SCH_\alpha\cR^*)$.
\end{prop}
\begin{proof}
First note that both functors are well-defined on this level. The Hom space between any two equivariant sheaves is naturally a representation of $G$ and the $G$-action is compatible with multiplication by elements in $\End_{\Grass}(\oplus_{\alpha\in B_{r,n-r}}\SCH_\alpha\cR^{* })\cong\mathbb{C}QK(r,n)/I$, and similar for $\Psi$. Also both functors commute with the forgetful functors forgetting the $G$-action. Without the $G$-equivariance, this result has been discussed. The compatibility of $G$-actions can be checked directly.

\end{proof}

\section{Calculation of quiver with relations for noncommutative desingularizations}
\label{sec:calculations}
Now, we use Sections~\ref{sec:Cat}, \ref{subsec: comb_Kap}, and \ref{sec: equi_quiver} to study the quiver with relations for the non-commutative desingularization of determinantal varieties in the spaces of symmetric and skew-symmetric matrices obtained in Section~\ref{sec:sym}.
\subsection{Symmetric matrices}
We follow the notions in Section~\ref{subsec:sym}, i.e., 
taking $S^s=k[x_{ij}]_{i\leq j}$, and taking $\Spec R^s\subseteq \Spec S^s$ to be the locus where $\varphi$ has rank $\leq
r$. A resolutions of $\Spec R^s$, denoted by $\cZ^s$, is given by a total space of a vector bundle on $\Grass_{n-r}(E^*)$. On $\Grass$ we consider the equivariant pair  $\Delta(\Grass)=\{\SCH_\lambda\cQ^*\mid\lambda\in B_{r,n-r}\}$ and $\nabla(\Grass)=\SCH_{(n-r)^r}\cQ^*\otimes \SCH_{(\alpha^c)'}\cR[(n-r)r-|\alpha|]$.  Proposition~\ref{prop: standardizable} gives collections $\{\Phi_\alpha\mid \alpha\in B_{r,n-r}\}$ and $\{\Sigma_\alpha\mid \alpha\in B_{r,n-r}\}$.
In particular,  $\TIL=\oplus_{\alpha\in B_{r,n-r}}\Phi_\alpha$ is a multiplicity-free tilting bundle.
In this section we are mostly  interested in the case when $k=\CC$ so that $\Phi_\alpha=\Delta_\alpha$ for all $\alpha\in B_{r,n-r}$, hence $\TIL$ becomes $\TIL_K(r,n)$, the Kapranov's tilting bundle. 

For general $k$,  let $\TIL_0(r,n)$ be the BLV-tilting bundle. Then $\TIL_0$ and $\TIL$ are the same up to multiplicity. In particular, $\End_{\cZ^s}(p'^*\TIL_0)$ and $\End_{\cZ^s}(p'^*\TIL)$ are Morita equivalent.

Before we discuss about the properties and the quiver with relations for the non-commutative desingularization $\End_{\cZ^s}(p'^*\TIL)$, let us remark about the equivariant projectives and simples over this non-commutative desingularization.

\begin{remark}
When $k=\CC$,  $p'^*(\SCH_\alpha\cQ^*)$ can have the same global sections for different $\alpha$'s, although they differ as graded objects. However, for $\alpha\neq\beta$, the objects $R\Hom_{{\cZ^s}}(p'^*\TIL_K, p'^*(\SCH_\alpha\cQ^*))$ and $ R\Hom_{{\cZ^s}}(p'^*\TIL_K, p'^*(\SCH_\beta\cQ^*))$ are different objects in $D^b(\End_{\cZ^s}(p'^*\TIL_K))$, as we have seen in Lemma~\ref{lem: dual_coll}, there is an object with all Ext's to $p'^*(\SCH_\alpha\cQ^*)$ vanishes but having non-trivial Ext to $p'^*(\SCH_\beta\cQ^*)$. 
\end{remark}
\par
More explicitly, we describe here the equivariant simple modules, considered as modules over $\End_{\sO_Z^s}(p^*(\oplus_{\alpha\in B_{r,n-r}}\nabla_\alpha))$ and as representations.
\begin{remark}
As can be seen, in the set-up of this section, the conditions in Proposition~\ref{prop: equi-simple} and Remark~\ref{rmk:equ_simple} are satisfied. Thus, $S_\beta=R\Hom_{\sO_Z^s}(p^*(\oplus_{\alpha\in B_{r,n-r}}\Phi_\alpha), u_*\Sigma_\beta)$ are all the equivariant simple objects over $\End_{\sO_Z^s}(p^*(\oplus_{\alpha\in B_{r,n-r}}\Phi_\alpha))$. As has been remarked in \ref{rmk:equ_simple}, the finite dimensional algebra $\End_{\sO_{\Grass}}(\oplus_{\alpha\in B_{r,n-r}}\Phi_\alpha)$ is a subring of $\End_{\sO_Z^s}(p^*(\oplus_{\alpha\in B_{r,n-r}}\Phi_\alpha))$. The modules $S_\alpha$ are also simple modules over $\mathbb{C}QK(n,n-r)/I$.

In fact, when $k=\CC$, as $\End_{\sO_{\Grass}}(\oplus_{\alpha\in B_{r,n-r}}\Phi_\alpha)\cong\mathbb{C}QK(n,n-r)/I$ is a finite dimensional basic algebra, $S_\alpha$ is the simple representation of the underlying quiver $QK(n,n-r)$ corresponding to the vertex $\alpha$. In this case, it is evident that $S_\alpha$'s are distinct as modules over $\End_{\sO_Z^s}(p^*(\oplus_\alpha\nabla_\alpha))$.
\par
Let us look at the structure of $S_\beta$'s as $GL_n$-representations. By definition,
\begin{eqnarray*}
S_\beta&=&R\Hom_{\sO_Z^s}(p^*(\oplus_\alpha\nabla_\alpha), u_*\Delta_\beta)\\
&\cong&\Hom_{\sO_Z^s}(p^*\nabla_\beta, u_*\Delta_\beta)\\
&\cong&\Hom_{\Grass}(\nabla_\beta, \Delta_\beta)\\
&\cong& H^{(n-r)r-|\alpha|}(\Grass, \SCH_\beta\cQ\otimes \SCH_{(n-r)^r}\cQ^*\otimes \SCH_{\alpha^{ct}}\cR)\\
&\cong&\mathbb{C}.
 \end{eqnarray*}
 All of them are 1-dimensional trivial representations.
\end{remark}

\subsubsection{Calculations}
The following Proposition is a direct consequence of Lemma~\ref{Lem: Exts}.
\par
\begin{prop}\label{prop: Exts Sym}
Let $k=\CC$ so that  $\Phi(\Grass)=\Delta(\Grass)=\{\SCH_\lambda\cQ^*\mid\lambda\in B_{r,n-r}\}$ and $\Sigma(\Grass)=\nabla(\Grass)=\SCH_{(n-r)^r}\cQ^*\otimes \SCH_{(\alpha^c)'}\cR[(n-r)r-|\alpha|]$. Let $S_\alpha$ be as in Proposition~\ref{prop: equi-simple}. Then, the Ext's among them are given by
$$\Ext^t(S_\alpha,S_\beta)\cong\oplus_s\oplus_{\lambda\in Q_{-1}(2s)}H^{t-s-|\beta|+|\alpha|}(\Grass,\SCH_\lambda \cQ^*\otimes \SCH_{\beta^t}\cR^*\otimes \SCH_{\alpha^t}\cR),$$ where $Q_{-1}(2s)=\{\lambda\vdash2s|\lambda=(a_1,\cdots,a_r|a_1-1,\cdots,a_r-1)\}$ in the hook notation.
\end{prop}

More explicitly, we have the following formula.
\begin{prop}\label{prop: ext12_sym}
Under the assumptions of Proposition~\ref{prop: Exts Sym}. We have
\[
\Ext^1(S_\alpha,S_\beta)\cong\left\{
\begin{aligned}
        &(C_{\beta,(1,0,\cdots,0)}^\alpha E^*)\oplus(C_{\alpha,(1,0,\cdots,0)}^\beta E^*), &\hbox{ if } n-r=1 \\
        &(E^*)\oplus(C_{\beta,(1,0,\cdots,0)}^\alpha\mathbb{C}), &\hbox{ if } n-r\geq2.
                          \end{aligned} \right.
\]
$\Ext^2(S_\alpha,S_\beta)\cong$
\[
\left\{
\begin{aligned}
        &(C^\alpha_{\beta,(1,1,0,\cdots,0)}\Sym_2E^* )\oplus(C^\beta_{\alpha,(1,1,0,\cdots,0)}\Sym_2E^* )\oplus(\delta^\beta_{\alpha}\wedge^2E^* ),\\ &\hbox{ if }n-r=1; \\
        &(C^{\beta^t}_{\alpha^t,(1,-1)}\wedge^2E^* )\oplus(C^{\beta^t}_{\alpha^t,(-1,-2)}E^* )\oplus(C^\beta_{\alpha,(1,1,0,\cdots,0)}\Sym_2E^* )\oplus( C^\beta_{\alpha,(2,0,\cdots,0)}\wedge^2E^* ), \\ &\hbox{ if }n-r=2; \\
        &(C^{\beta^t}_{\alpha^t,(0,\cdots,0,-1,-1,-2)}\mathbb{C})\oplus(C^{\beta^t}_{\alpha^t,(1,0,\cdots,0,-1,-1)}E^* )\oplus(C^{\beta^t}_{\alpha^t,(2,0,\cdots,0)}\Sym_2E^* )\oplus(C^{\beta^t}_{\alpha^t,(1,1,0,\cdots,0)}\wedge^2E^* ), \\ &\hbox{ if }n-r\geq2.
                          \end{aligned} \right.
\]
\end{prop}
\begin{proof}
Note that the only element in $Q_{-1}(2)$ is $(2,0,\cdots,0)$ and the only element in $Q_{-1}(4)$ is $(3,1,0,\cdots,0)$. One can easily calculate $\Ext^1$ and $\Ext^2$ with the aid of Lemma~\ref{lem: ext_algo}.
\par
For $\lambda=(3,1,0,\cdots,0)$, the  $\gamma_{min(3,1,0,\cdots,0)}$ is given by Lemma~\ref{lem: ext_algo} is $(-2)$ if $n-r=1$ with the corresponding $t_{min(3,1,0,\cdots,0)}=2$; $\gamma_{min(3,1,0,\cdots,0)}=(-1,-2)$ if $n-r=2$ with the corresponding $t_{min(3,1,0,\cdots,0)}=2$; $\gamma_{min(3,1,0,\cdots,0)}=(-1,-1,-2)$ if $n-r\geq3$ with the corresponding $t_{min(3,1,0,\cdots,0)}=2$. Note that in any case above there is no operation described in Remark~\ref{rmk: ext_algo} satisfying the constrains given by Corollary~\ref{cor: ext_delta}.
\par
For $\lambda=(2,0,\cdots,0)$, the $\gamma_{min(2,0,\cdots,0)}$ given by Lemma~\ref{lem: ext_algo} is $(-1)$ if $n-r=1$, with the corresponding $t_{min(2,0,\cdots,0)}=1$; $\gamma_{min(2,0,\cdots,0)}=(0,\cdots,0,-1,-1)$ if $n-r\geq2$ with the corresponding $t_{min(2,0,\cdots,0)}=1$.
\par
Note that in the case $n-r=1$ there is one operation described in Remark~\ref{rmk: ext_algo} satisfying the constrains given by Corollary~\ref{cor: ext_delta}, which gives $\gamma=(0)$ with $t=2$. In the case $n-r=2$ there is one operation described in Remark~\ref{rmk: ext_algo} satisfying the constrains given by Corollary~\ref{cor: ext_delta} but can be applied successively,  which gives $\gamma=(1,-1)$ and $\gamma=(2,-1)$ with the corresponding $t=2$ and 3 respectively. In the case $n-r>2$ there are a lot of operations. But if we only care about those with corresponding $t=2$ and satisfying the constrains given by Corollary~\ref{cor: ext_delta},  there is only one which gives $\gamma=(1,0,\cdots,0,-1,-1)$.
\par
Note that only $Q_{-1}(2s)$ with $s=0,1,2$ contributes to $\Ext^1$ and $\Ext^2$. This finishes the proof.
\end{proof}

\subsection{Maximal minors in symmetric matrices}
In the rest of this section, we illustrate Proposition~\ref{prop: Exts Sym} and Proposition~\ref{prop: ext12_sym} with explicit examples and numerical consequences.

We assume $r=n-1$, $S^s=k[x_{i,j}]_{1\leq i\leq j\leq n}$
and $R^s=S^s/(\det(x_{ij}))$. In this case, $R^s$ is always Gorenstein, since
it is a hypersurface.
\par
The inverse image of the tilting bundle from the Grassmannian is still a
tilting bundle by the same argument as before. More explicitly, it
is $$\bigoplus_{i=0}^{n-1}\wedge^iQ^*\otimes\Sym(\Sym_2Q).$$

\par

\begin{prop}
The endomorphism ring $\End_S(q_*'p'^*\TIL_K)$ is isomorphic to the path algebra of the
quiver:
$$\xymatrix{
\bullet_0\ar@/^/@<4ex>[r]|-{\alpha_1}
\ar@/^/@<3ex>[r]|-{\alpha_2}_{\cdots} \ar@/^/@<1ex>[r]|-{\alpha_n} &
\bullet_1\ar@/^/@<4ex>[r]|-{\alpha_1}
\ar@/^/@<3ex>[r]|-{\alpha_2}_{\cdots} \ar@/^/@<1ex>[r]|-{\alpha_n}
\ar@/^/@<4ex>[l]|-{\beta_1} \ar@/^/@<3ex>[l]|-{\beta_2}_{\cdots}
\ar@/^/@<1ex>[l]|-{\beta_n} &  \cdots\ar@/^/@<4ex>[r]|-{\alpha_1}
\ar@/^/@<3ex>[r]|-{\alpha_2}_{\cdots} \ar@/^/@<1ex>[r]|-{\alpha_n}
\ar@/^/@<4ex>[l]|-{\beta_1} \ar@/^/@<3ex>[l]|-{\beta_2}_{\cdots}
\ar@/^/@<1ex>[l]|-{\beta_n} &
\bullet_{n-1}\ar@/^/@<4ex>[l]|-{\beta_1}
\ar@/^/@<3ex>[l]|-{\beta_2}_{\cdots} \ar@/^/@<1ex>[l]|-{\beta_n} }$$
with relations:

$$\alpha_i\alpha_j+\alpha_j\alpha_i,$$
$$\beta_i\beta_j+\beta_j\beta_i,$$
$$(\alpha_i\beta_j+\beta_j\alpha_i)-(\alpha_j\beta_i+\beta_i\alpha_j),$$
where for any term not making sense at some vertex, it is to be
understood as dropped.
\end{prop}
\par
In the language of equivariant quivers, the above quiver can be written as
$$\xymatrix{
\bullet_0\ar@/^/@<1ex>[r]^{\alpha_0(E)} &
\bullet_1\ar@/^/@<1ex>[r]^{\alpha_1(E)}\ar@/^/@<1ex>[l]^{\beta_0(E)}&
\cdots\ar@/^/@<1ex>[r]^{\alpha_{n-2}(E)}\ar@/^/@<1ex>[l]^{\beta_1(E)}
 &
\bullet_{n-1}\ar@/^/@<1ex>[l]^{\beta_{n-2}(E)} }$$
and the relations:

$$\alpha_i\alpha_{i+1}(\wedge^2E);$$
$$\beta_{i+1}\beta_{i}(\wedge^2E);$$
$$(\beta_i\alpha_{i}+\alpha_{i+1}\beta_{i+1})(\Sym_2E).$$
\begin{proof}
See the proof of Proposition~\ref{prop: sym_gen} and the proof of Proposition~\ref{prop: sym_rel}.
\end{proof}

\par
A description of the Ext's between the simples in the
module category over the path algebra of quiver with relations is given by Proposition~\ref{prop: Exts Sym}.
\par

\begin{example}
Let $k=\CC$. For any integer $a\in[0,n-1]$, the minimal projective resolution of the simple
object $S_a$ is:
$$\xymatrix{
P_a & *\txt{$E \otimes P_{a-1}$\\ $\oplus$\\ $E \otimes P_{a+1}$}
\ar[l] &  *\txt{$\SCH_2E \otimes P_{a-2}$\\ $\oplus$\\
$\wedge^2E \otimes P_{a}$\\$\oplus$\\$\SCH_2E \otimes P_{a+2}$}
\ar[l] &
*\txt{$\SCH_3E \otimes P_{a-3}$\\ $\oplus$\\ $\SCH_{22}E \otimes
P_{a-2}$\\$\oplus$\\$\SCH_{211}E \otimes
P_{a}$\\$\oplus$\\$\SCH_{22}E \otimes
P_{a+2}$\\$\oplus$\\$\SCH_{3}E \otimes
P_{a+3}$} \ar[l] & *\txt{$\SCH_4E \otimes P_{a-4}$\\ $\oplus$\\
$\SCH_{32}E \otimes P_{a-3}$\\$\oplus$\\$\SCH_{221}E \otimes
P_{a-1}$\\$\oplus$\\$\SCH_{221}E \otimes
P_{a+1}$\\$\oplus$\\$\SCH_{32}E \otimes
P_{a+3}$\\$\oplus$\\$\SCH_{4}E \otimes P_{a+4}$} \ar[l] }.$$
If $a$ is close to the boundary (i.e., 0 and n-1), some terms in the above resolution does not exist. In those cases, the terms in question should be understood as dropped.
\end{example}

In this case, the commutative desingularization and the noncommutative one are further related in the sense that $\cZ^s$ is the fine moduli space of certain representations of the noncommutative one.

\begin{example}
Take $n=2$ and $r=1$.
Then $R^s$ is the nil-cone of $\mathfrak{sl}_2$, and the commutative resolution $\cZ^s$ is the Springer resolution $\hbox{T}^*\mathbb{P}^1$.
The quiver with relations:
$$
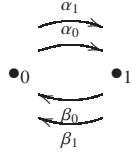
\xymatrix{
\bullet_0\ar@/^/@<1ex>[r]^{\alpha_0}\ar@/^/@<3ex>[r]^{\alpha_1}&
\bullet_1\ar@/^/@<1ex>[l]^{\beta_0}\ar@/^/@<3ex>[l]^{\beta_1} }$$
relations: $\alpha_0\beta_1=\alpha_1\beta_0;$
$\beta_0\alpha_1=\beta_1\alpha_0.$

Representations $W$ of dimension $(1,1)$ generated by $W_1$ are parameterized by $\hbox{T}^*\mathbb{P}^1$.
The irreducible such representations form the regular locus in $\cZ^s$.
These simple modules (there are infinitely many of them) do not admit $GL_2$-equivariant structures.
\end{example}

\subsection{Higher codimension subvarieties in symmetric matrices}
Now we start using Proposition~\ref{prop: ext12_sym} to describe the quiver with relations for this non-commutative desingularization.
\begin{prop}\label{prop: sym_gen}
Let $k=\CC$. Assume $n-r\geq 2$. In the quiver with relations for the algebra $\Lambda=\End_{\sO_{\cZ^s}}(p'^*\TIL_K)$ as in Proposition~\ref{prop: equi-simple}, the vertex set is indexed by $B_{r,n-r}$, and the arrow from $\beta$ to $\alpha$ is given by $E$ if $C_{\beta,(1,0,\cdots,0)}^\alpha\neq 0$  and given by $\mathbb{C}$ if $C_{\alpha,(1,1,0,\cdots,0)}^\beta\neq 0$. No arrows otherwise.
\end{prop}

\begin{prop}\label{prop: sym_rel}
Let $k=\CC$. In the quiver with relations for the algebra $\Lambda=\End_{\sO_{\cZ^s}}(p'^*\TIL_K)$ as in Proposition~\ref{prop: equi-simple}, the relations are generated by the following sub-representations of the arrows in $\Hom(\beta,\alpha)$:
\begin{itemize}
  \item $$(C^{\beta^t}_{\alpha^t,(1,-1)}\wedge^2E )\oplus(C^{\beta^t}_{\alpha^t,(-1,-2)}E )\oplus(C^\beta_{\alpha,(1,1,0,\cdots,0)}\Sym_2E )\oplus( C^\beta_{\alpha,(2,0,\cdots,0)}\wedge^2E )$$  in the case $n-r=2$;
  \item $$(C^{\beta^t}_{\alpha^t,(0,\cdots,0,-1,-1,-2)}\mathbb{C})\oplus(C^{\beta^t}_{\alpha^t,(1,0,\cdots,0,-1,-1)}E )\oplus(C^{\beta^t}_{\alpha^t,(2,0,\cdots,0)}\Sym_2E )\oplus(C^{\beta^t}_{\alpha^t,(1,1,0,\cdots,0)}\wedge^2E )$$ in the case $n-r\geq3$.
\end{itemize}

\end{prop}
\par

\subsubsection{Examples}
We study some combinatorial properties of the noncommutative desingularization. \textit{In this subsection we assume $k=\CC$.}

We look at two examples with $r=n-1$, $S^s=\mathbb{C}[x_{i,j}]_{1\leq i\leq j\leq n}$
and $R^s=S^s/(det)$.
\par
The inverse image of the tilting bundle is $$\bigoplus_{i=0}^{n-1}\wedge^iQ^*\otimes\Sym(\Sym_2Q).$$
The presentation of $H^0({\cZ^s},\wedge^iQ^*\otimes\Sym(\Sym_2Q))$ as a
$S$-module is given by
$$0\leftarrow H^0({\cZ^s},\wedge^iQ^*\otimes\Sym(\Sym_2Q))\leftarrow
\SCH_{(1^i,0^{n-i})}E^*\otimes S\leftarrow \SCH_{(2^i,1^{n-i})}E^*\otimes
S.$$

\begin{example}
Now we take $n=3$ and $r=2$ for a concrete example.
Here $S^s=\mathbb{C}[x_{i,j}]_{1\leq i\leq j\leq 3}$ and
$R^s=S^s/(det)$.
\par
In this example, $R$ is a normal Gorenstein domain since it is a hypersurface. Consequently, if  $\End_{\cZ^s}(p'^*\TIL_K)$ is maximal Cohen-Macaulay and $q_*'p'^*\TIL_K$ is reflexive, then $\End_{\cZ^s}(p'^*\TIL_K)$ is a non-commutative crepant desingularization due to Proposition~\ref{prop: sym_mcm}.

\par
The inverse image of the tilting bundle consists of three direct
summands: $\Sym(\Sym_2Q)$, $Q^*\otimes\Sym(\Sym_2Q)$,
and $\wedge^2Q^*\otimes\Sym(\Sym_2Q)$. Let us denote their global sections
as $R$-modules, (and consequently as $S^s$-modules), by
$$M_i:=H^0(\cZ^s,\wedge^iQ^*\otimes\Sym(\Sym_2Q))$$ with $i=0,1,2$. The presentations of
their global sections as $S^s$-modules are given by
$$0\leftarrow H^0(\cZ^s,\wedge^iQ^*\otimes\Sym(\Sym_2Q))\leftarrow
\SCH_{(1^i,0^{3-i})}E^*\otimes S^s\leftarrow
\SCH_{(2^i,1^{3-i})}E^*\otimes S^s.$$
\par
The $\Hom$'s between them are
$$\Hom_R(M_i,M_j)=H^0(\cZ^s,\wedge^jQ^*\otimes\wedge^iQ\otimes\Sym(\Sym_2Q)).$$
More explicitly, \begin{itemize}
     \item $\Hom_R(M_0,M_i)=M_i$,
     \item $\Hom_R(M_i,M_0)=M_i$,
     \item $\Hom_R(M_2,M_i)=M_{2-i}$,
     \item $\Hom_R(M_i,M_2)=M_{2-i}$.
     \end{itemize}
Through some computations, we get $$\Hom_R(M_1,M_1)=M_0\oplus N$$ and
the presentation of $N$ is $$0\leftarrow N\leftarrow
\SCH_{(200)}E^*\oplus \SCH_{(110)}E^*\otimes S^s\leftarrow
\SCH_{(220)}E^*\oplus \SCH_{(211)}E^*\otimes S^s.$$
\par
The endomorphism ring is isomorphic to the path algebra of the
following quiver
$$\xymatrix{
\bullet_0\ar@/^/@<4ex>[r]|-{\alpha_1} \ar@/^/@<3ex>[r]|-{\alpha_2}
\ar@/^/@<2ex>[r]|-{\alpha_3} & \bullet_1\ar@/^/@<4ex>[r]|-{\gamma_1}
\ar@/^/@<3ex>[r]|-{\gamma_2} \ar@/^/@<2ex>[r]|-{\gamma_3}
\ar@/^/@<4ex>[l]|-{\beta_1} \ar@/^/@<3ex>[l]|-{\beta_2}
\ar@/^/@<2ex>[l]|-{\beta_3} &
\bullet_{2}\ar@/^/@<4ex>[l]|-{\delta_1} \ar@/^/@<3ex>[l]|-{\delta_2}
\ar@/^/@<2ex>[l]|-{\delta_3} }$$
with relations
$$\gamma_i\alpha_j+\gamma_j\alpha_i,$$
$$\beta_j\delta_i+\beta_i\delta_j,$$
$$\beta_j\alpha_i-\beta_i\alpha_j,$$
$$\gamma_j\delta_i-\gamma_i\delta_j,$$
$$(\alpha_i\beta_j+\delta_j\gamma_i)-(\alpha_j\beta_i+\delta_i\gamma_j).$$

The Hom between any two direct summands is graded, with grading given by the weight of $\mathbb{G}_m$-action on $E$. The grading defined this way is different but (strictly) finer than the one used in the proof of Proposition~\ref{prop: equi-simple}.
\par
The Hilbert series of those modules can be computed from the
presentations:
\begin{itemize}
  \item $H_{M_0}=\frac{1-t^6}{(1-t^2)^6}$;
  \item $H_{M_1}=\frac{3-3t^4}{(1-t^2)^6}$;
  \item $H_{M_2}=\frac{3-3t^2}{(1-t^2)^6}$;
  \item $H_N=\frac{9-9t^2}{(1-t^2)^2}$.
\end{itemize}
\par
Putting the Hilbert series of the Hom's into a matrix, we get

$$\frac{1}{(1-t^2)^5}\times\left(
\begin{array}{ccc}
1+t^2+t^4 & 3t+3t^3 & 3t^2 \\
3t+3t^3 & 1+10t^2+t^4 & 3t+3t^3 \\
3t^2 & 3t+3t^3 & 1+t^2+t^4 \\
\end{array}
\right)$$
\par
The coefficients in front of each monomials in the entries of the
inverse matrix gives the multiplicity of the projectives in the
resolution of the simples. The matrix above has an inverse with
polynomial entries, which reflects the fact that the derived
category over the endomorphism ring has finite global dimension.
\par
The inverse matrix is
$$\left(
    \begin{array}{ccc}
      -t^6-3t^4+3t^2+1 & 3t^5-3t & -6t^4+6t^2 \\
      3t^5-3t & -t^6-3t^4+3t^2+1 & 3t^5-3t \\
      -6t^4+6t^2 & 3t^5-3t & -t^6-3t^4+3t^2+1 \\
    \end{array}
  \right)
$$
\par
It is easy to guess the resolution of the simples from this matrix.
$$\xymatrix@C=12pt{
S_0: & P_0 & *\txt{$E \otimes P_1$} \ar[l] & *\txt{$\wedge^2E \otimes P_0$ \\ $\oplus$\\
$\SCH_2E \otimes P_2$} \ar[l] &
*\txt{$\SCH_{211}E \otimes P_0$\\$\oplus$\\$\SCH_{22}E \otimes P_2$} \ar[l]
& *\txt{$\SCH_{221}E \otimes P_1$} \ar[l] & \SCH_{222}E \otimes P_0
\ar[l] & 0 \ar[l]};$$
$$\xymatrix@C=12pt{
S_1: & P_1 & *\txt{$E \otimes P_0$\\$\oplus$\\$E \otimes P_2$}
\ar[l] & *\txt{$\wedge^2E \otimes P_1$ } \ar[l] &
*\txt{$\SCH_{211}E \otimes P_1$} \ar[l]
& *\txt{$\SCH_{221}E \otimes P_0$\\$\oplus$\\$\SCH_{221}E \otimes P_2$}
\ar[l] & \SCH_{222}E \otimes P_1 \ar[l] & 0 \ar[l]}.$$
One can easily verify that this resolution coincide with the one  given by Lemma~\ref{lem:proj_res} and Proposition~\ref{prop: Exts Sym}.
\end{example}

\begin{example}
Let us look at one more example, with $n=4$ and $r=3$.
\par
As before, the direct summands of the tilting bundle over the
desingularization are $\wedge^iQ^*\otimes\Sym_2(\SCH_2Q)$ where
$i=0,\dots,3$.
\par
The presentations of their global sections as $S^s$-modules are
given by $$0\leftarrow
M_i:=H^0(\cZ,\wedge^iQ^*\otimes\Sym(\SCH_2Q))\leftarrow
\SCH_{(1^i,0^{n-i})}E^*\otimes S^s\leftarrow
\SCH_{(2^i,1^{n-i})}E^*\otimes S^s.$$
\par

\par
The endomorphism ring is isomorphic to the path algebra of the
quiver:
$$
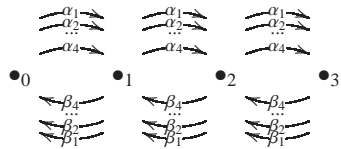
\xymatrix{
\bullet_0\ar@/^/@<4ex>[r]|-{\alpha_1}
\ar@/^/@<3ex>[r]|-{\alpha_2}_{\cdots} \ar@/^/@<1ex>[r]|-{\alpha_4} &
\bullet_1\ar@/^/@<4ex>[r]|-{\alpha_1}
\ar@/^/@<3ex>[r]|-{\alpha_2}_{\cdots} \ar@/^/@<1ex>[r]|-{\alpha_4}
\ar@/^/@<4ex>[l]|-{\beta_1} \ar@/^/@<3ex>[l]|-{\beta_2}_{\cdots}
\ar@/^/@<1ex>[l]|-{\beta_4} &
\bullet_{2}\ar@/^/@<4ex>[r]|-{\alpha_1}
\ar@/^/@<3ex>[r]|-{\alpha_2}_{\cdots} \ar@/^/@<1ex>[r]|-{\alpha_4}
\ar@/^/@<4ex>[l]|-{\beta_1} \ar@/^/@<3ex>[l]|-{\beta_2}_{\cdots}
\ar@/^/@<1ex>[l]|-{\beta_4} & \bullet_{3}\ar@/^/@<4ex>[l]|-{\beta_1}
\ar@/^/@<3ex>[l]|-{\beta_2}_{\cdots} \ar@/^/@<1ex>[l]|-{\beta_4} }$$
with relations: $$\alpha_i\alpha_j+\alpha_j\alpha_i,$$
$$\beta_i\beta_j+\beta_j\beta_i,$$
$$(\alpha_i\beta_j+\beta_j\alpha_i)-(\alpha_j\beta_i+\beta_i\alpha_j).$$
\par

To get the Hilbert polynomials of the Hom's among them, we only
need to compute the presentations of two modules, i.e.,
$\Hom(M_1,M_1)$ and $\Hom(M_1,M_2)$. They are given as follows.
\begin{itemize}
  \item $\Hom(M_1,M_1)=M_0\oplus C_1$, where $$0\leftarrow C_1\leftarrow \SCH_{11}E \otimes S^s\oplus \SCH_2E \otimes S^s\leftarrow \SCH_{2211}E \otimes S^s\oplus \SCH_{222}E \otimes S^s$$ is exact;
  \item $\Hom(M_1,M_2)=M_1\oplus C_2$, where $$0\leftarrow C_2\leftarrow \SCH_{111}E \otimes S^s\oplus \SCH_{21}E \otimes S^s\leftarrow \SCH_{221}E \otimes S^s\oplus \SCH_{2111}E \otimes
  S^s$$ is exact.
\end{itemize}
\par
The matrix of Hilbert polynomials between their Hom's is
$$\frac{1}{(1-t^2)^{10}}\left(
    \begin{array}{cccc}
      1-t^8 & 4-4t^6 & 6-6t^4 & 4-4t^2 \\
      4-4t^6 & 17-16t^4-t^8 & 28-24t^2-4t^6 & 6-6t^4 \\
      6-6t^4 & 28-24t^2-4t^6 & 17-16t^4-t^8 & 4-4t^6 \\
      4-4t^2 & 6-6t^4 & 4-4t^6 & 1-t^8 \\
    \end{array}
  \right)
.$$ Its inverse matrix is \tiny{$$\left(
    \begin{array}{cccc}
      15t^8-6t^{10}-t^{12}-15t^4+6t^2+1 & 20t^5-20t^7+4t^{11}-4t & -20t^4+20t^8-10t^{10}+10t^2 & -20t^3+60t^5-60t^7+20t^9 \\
      20t^5-20t^7+4t^{11}-4t & 15t^8-6t^{10}-t^{12}-15t^4+6t^2+1 & 20t^5-20t^7+4t^{11}-4t & -20t^4+20t^8-10t^{10}+10t^2 \\
      -20t^4+20t^8-10t^{10}+10t^2 & 20t^5-20t^7+4t^{11}-4t & 15t^8-6t^{10}-t^{12}-15t^4+6t^2+1 & 20t^5-20t^7+4t^{11}-4t \\
      -20t^3+60t^5-60t^7+20t^9 & -20t^4+20t^8-10t^{10}+10t^2 & 20t^5-20t^7+4t^{11}-4t &  15t^8-6t^{10}-t^{12}-15t^4+6t^2+1 \\
    \end{array}
  \right)
.$$}
\par
\normalsize{One can guess the resolution of the simples from this matrix.}
$$\xymatrix@C=9pt{
S_0: & P_0 & *\txt{$E \otimes P_1$} \ar[l] & *\txt{$\wedge^2E \otimes P_0$ \\ $\oplus$\\
$\SCH_2E \otimes P_2$} \ar[l] &
*\txt{$\SCH_3E \otimes P_3$\\$\oplus$\\$\SCH_{211}E \otimes P_0$\\$\oplus$\\$\SCH_{22}E \otimes P_2$}
\ar[l] & *\txt{$\SCH_{221}E \otimes P_1$\\$\oplus$\\$\SCH_{32}E \otimes
P_3$} \ar[l] \\
& *\txt{$\SCH_{3211}E \otimes P_1$\\$\oplus$\\$\SCH_{331}E \otimes P_3$}
\ar[l] &
*\txt{$\SCH_{333}E \otimes P_3$\\$\oplus$\\$\SCH_{3221}E \otimes
P_0$\\$\oplus$\\$\SCH_{3311}E \otimes
P_2$} \ar[l] & *\txt{$\SCH_{3322}E \otimes P_0$ \\ $\oplus$\\
$\SCH_{3331}E \otimes P_2$} \ar[l] & *\txt{$\SCH_{3332}E \otimes P_1$}
\ar[l] & \SCH_{3333}E \otimes P_0 \ar[l] & 0; \ar[l]}$$
$$\xymatrix@C=9pt{
S_1: & P_1 & *\txt{$E \otimes P_0$\\$\oplus$\\$E \otimes P_2$}
\ar[l] & *\txt{$\wedge^2E \otimes P_1$\\$\oplus$\\$\SCH_2E \otimes
P_3$ } \ar[l] &
*\txt{$\SCH_{211}E \otimes P_1$\\$\oplus$\\$\SCH_{22}E \otimes P_3$} \ar[l]
& *\txt{$\SCH_{221}E \otimes P_0$\\$\oplus$\\$\SCH_{221}E \otimes P_2$}
\ar[l]\\
& *\txt{$\SCH_{3211}E \otimes P_0$\\$\oplus$\\$\SCH_{3211}E \otimes
P_2$} \ar[l] & \txt{$\SCH_{3221}E \otimes
P_1$\\$\oplus$\\$\SCH_{3311}E \otimes P_3$} \ar[l] &
*\txt{$\SCH_{3322}E \otimes P_1$\\$\oplus$\\$\SCH_{3331}E \otimes P_3$ }
\ar[l] &
*\txt{$\SCH_{3332}E \otimes P_0$\\$\oplus$\\$\SCH_{3332}E \otimes P_2$} \ar[l]
 & \SCH_{3333}E \otimes P_1 \ar[l] & 0. \ar[l]}$$

Again, it is easy to verify that this resolution coincide with the one  given by Lemma~\ref{lem:proj_res} and Proposition~\ref{prop: Exts Sym}.
\end{example}
\par
\begin{example}
Now we look at an example of higher codimension symmetric minors case. We take $n=\dim E=4$ and $r=2$.
\par

The first two steps are relatively easy. By direct computation, we get the following quiver (recall that vertices are indexed by Young diagrams):
$$
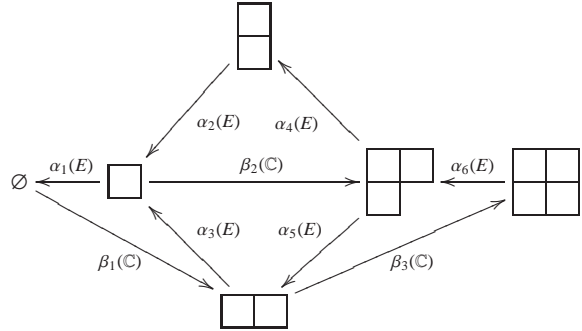
\xymatrix{
 & & \text{$\Young(|)$}\ar[dl]^{\alpha_2(E)} & & \\
 \varnothing\ar[drr]_{\beta_1(\mathbb{C})} & \text{$\Young()$}\ar[l]_{\alpha_1(E)} \ar[rr]^{\beta_2(\mathbb{C})} & & \text{$\Young(,|)$}\ar[ul]^{\alpha_4(E)}\ar[dl]_{\alpha_5(E)} & \text{$\Young(,|,)$}\ar[l]_{\alpha_6(E)}\\
& & \text{$\Young(,)$}\ar[ul]_{\alpha_3(E)}\ar[urr]_{\beta_3(\mathbb{C})}
}$$
with relations:
\begin{itemize}
  \item $\Hom(\Young(,),\varnothing)$: $\alpha_1\alpha_3(\wedge^2E)$;
  \item $\Hom(\Young(|),\varnothing)$: $\alpha_1\alpha_2(\Sym_2E)$;
  \item $\Hom(\Young(,|,),\Young(,))$: $\alpha_5\alpha_6(\wedge^2E)$;
  \item $\Hom(\Young(,|,),\Young(|))$: $\alpha_4\alpha_6(\Sym_2E)$.
  \item $\Hom(\Young(,|),\Young())$: $\alpha_2\alpha_4-\alpha_3\alpha_5(\Sym_2E\oplus\wedge^2E)$;
  \item $\Hom(\Young(),\Young())$: $\alpha_3\beta_1\alpha_1-\alpha_3\alpha_5\beta_2(\wedge^2E)$;
  \item $\Hom(\Young(,|),\Young(,|))$: $\alpha_6\beta_3\alpha_5-\beta_2\alpha_3\alpha_5(\wedge^2E)$;
  \item $\Hom(\varnothing,\Young(,|))$: $\beta_2\alpha_3\beta_1-\alpha_6\beta_3\beta_1(E)$;
  \item $\Hom(\Young(),\Young(,|,))$: $\beta_3\beta_1\alpha_1-\beta_3\alpha_5\beta_2(E)$.

\end{itemize}
\par
\par
Using Olver's description of Pieri inclusions (see, e.g., 1.2 of \cite{SW}), it is easy to check the above listed subrepresentations acts trivially on $p'^*\oplus_{\alpha\in B_{r,n-r}}\SCH_\alpha\cQ^*$. Then Proposition~\ref{prop: sym_rel} implies that these are all the relations.
\end{example}

\subsection{Pfaffian varieties of anti-symmetric matrices}\label{sec: skew}
We follow the notions in Section~\ref{subsec:skew}, i.e., 
taking $S^a=k[x_{ij}]_{i< j}$, and taking $\Spec R^a\subseteq \Spec S^a$ to be the locus where $\varphi$ has rank $\leq
r$. A resolutions of $\Spec R^a$, denoted by $\cZ^a$, is given by a total space of a vector bundle on $\Grass_{n-r}(E^*)$. 
We focus on the case when $k=\CC$, hence   $\Phi(\Grass)=\Delta(\Grass)=\{\SCH_\lambda\cQ^*\mid\lambda\in B_{r,n-r}\}$ and $\Sigma(\Grass)=\nabla(\Grass)=\SCH_{(n-r)^r}\cQ^*\otimes \SCH_{(\alpha^c)'}\cR[(n-r)r-|\alpha|]$. 
In particular,  $\TIL=\oplus_{\alpha\in B_{r,n-r}}\Phi_\alpha$ is $\TIL_K$, the Kapranov's tilting bundle.

Now, we describe the Ext's between the equivariant simples. Then, we  use it to get the quiver with relations for the non-commutative desingularization.
\par
\begin{prop}\label{prop: Exts skew}
Let $k=\CC$. Let $S_\alpha$ be as in Proposition~\ref{prop: equi-simple}. Then, the Ext's among them are given by
$$\Ext^t(S_\alpha,S_\beta)\cong\oplus_s\oplus_{\lambda\in Q_{1}(2s)}H^{t-s-|\beta|+|\alpha|}(\Grass,\SCH_\lambda \cQ^*\otimes \SCH_{\beta^t}\cR^*\otimes \SCH_{\alpha^t}\cR),$$ where $Q_{1}(2s)=\{\lambda\vdash2s\mid\lambda=(a_1,\cdots,a_r|a_1+1,\cdots,a_r+1)\}$ in the hook notation.
\end{prop}
More explicitly, we have the following proposition.
\begin{prop}\label{prop: ext12_skew}
Under the assumptions of Proposition~\ref{prop: Exts skew}. We have
\[
\Ext^1(S_\alpha,S_\beta)\cong(C_{\beta,(1,0,\cdots,0)}^\alpha E^*)\oplus(C_{\alpha,(1,1,0,\cdots,0)}^\beta\mathbb{C});
\]
$\Ext^2(S_\alpha,S_\beta)\cong$
\[
\left\{
\begin{aligned}
        &(C_{\alpha^t,(1,0,\cdots,0)}^{\beta^t}\wedge^3E^* )\oplus(C_{\alpha^t,(2,0,\cdots,0)}^{\beta^t}\SCH_2E^* ),\\ &\hbox{ if }n-r=1; \\
        &(C_{\alpha^t,(0,\cdots,0,-1,-3)}^{\beta^t}\mathbb{C})\oplus(C_{\alpha^t,(1,0,\cdots,0,-2)}^{\beta^t}E^* )\oplus(C_{\alpha^t,(2,0,\cdots,0)}^{\beta^t}\SCH_2E^* )\oplus( C_{\alpha^t,(1,1,0,\cdots,0)}^{\beta^t}\wedge^2E^* ), \\ &\hbox{ if }n-r\geq2.
                          \end{aligned} \right.
\]
\end{prop}
\begin{proof}
Note that the only element in $Q_{1}(2)$ is $(1,1,0,\cdots,0)$ and the only element in $Q_{1}(4)$ is $(2,1,1,0,\cdots,0)$. One can easily calculate $\Ext^1$ and $\Ext^2$ with the aid of Lemma~\ref{lem: ext_algo}.
\par
For $\lambda=(1,1,0,\cdots,0)$, the  $\gamma_{min(1,1,0,\cdots,0)}$ is given by Lemma~\ref{lem: ext_algo} is $(0,\cdots,0,-2)$ with the corresponding $t_{min(1,1,0,\cdots,0)}=1$. There are a lot of operations described in Remark~\ref{rmk: ext_algo}. Nevertheless, among those there is only one  with corresponding $t=2$ and satisfying the constrains given by Corollary~\ref{cor: ext_delta}. This operation gives $\gamma=(1,0,\cdots,0,-2)$.

\par
For $\lambda=(2,1,1,0,\cdots,0)$, the $\gamma_{min(2,1,1,0,\cdots,0)}$ is given by Lemma~\ref{lem: ext_algo} is $(-3)$ if $n-r=1$ with the corresponding $t_{min(2,1,1,0,\cdots,0)}=2$; $\gamma_{min(2,1,1,0,\cdots,0)}=(0,\cdots,0,-1,-3)$ if $n-r\geq2$ with the corresponding $t_{min(2,1,1,0,\cdots,0)}=2$.
\par
Note that only $Q_{1}(2s)$ with $s=0,1,2$ contributes to $\Ext^1$ and $\Ext^2$. This finishes the proof.
\end{proof}
\par
Now we start using Proposition~\ref{prop: ext12_skew} to describe the quiver with relations for this non-commutative desingularization.
\begin{prop}\label{prop: skew_gen}
Let $k=\CC$. In the quiver with relations for the algebra $\Lambda=\End_{\sO_{\cZ^a}}(p'^*\TIL_K)$ as in Proposition~\ref{prop: equi-simple}, the set of vertices correspond to $B_{r,n-r}$, and the arrow from $\beta$ to $\alpha$ is given by $E$ if $C_{\beta,(1,0,\cdots,0)}^\alpha\neq 0$  and given by $\mathbb{C}$ if $C_{\alpha,(1,1,0,\cdots,0)}^\beta\neq 0$. No arrows otherwise.
\end{prop}

\begin{prop}\label{prop: skew_rel}
Let $k=\CC$. In the quiver with relations for the algebra $\Lambda=\End_{\sO_{\cZ^a}}(p'^*\TIL_K)$ as in Proposition~\ref{prop: equi-simple}, the relations are generated by the following sub-representations of the arrows in $\Hom(\beta,\alpha)$:
\begin{itemize}
  \item $$(C_{\alpha^t,(1,0,\cdots,0)}^{\beta^t}\wedge^3E )\oplus(C_{\alpha^t,(2,0,\cdots,0)}^{\beta^t}\SCH_2E )$$  in the case $n-r=1$;
  \item $$(C_{\alpha^t,(0,\cdots,0,-1,-3)}^{\beta^t}\mathbb{C})\oplus(C_{\alpha^t,(1,0,\cdots,0,-2)}^{\beta^t}E )\oplus(C_{\alpha^t,(2,0,\cdots,0)}^{\beta^t}\SCH_2E )\oplus( C_{\alpha^t,(1,1,0,\cdots,0)}^{\beta^t}\wedge^2E )$$  in the case $n-r\geq2$.
\end{itemize}
\end{prop}
%

\section{More examples of non-commutative desingularizations}
\label{sec: other_examples}
Here we look at some non-commutative desingularizations beyond representations with finitely many orbits. The technique we developed in Section~\ref{sec: general_discussion} works in these examples. As one have seen in Proposition~\ref{prop: Exts Sym}, computing quiver with relations involves the problem of inner plethysm and cannot be done in general only be done in special cases.
For simplicity,\textit{ we assume $k=\CC$ in this section}, although some of the results are characteristic free.

\subsection{Rank varieties of anti-symmetric tensors}
Let $E$ be a vector space over $k$ of dimension $n$ and $\mathbb A^{{n}\choose{d}}=\wedge^dE^*$ be the
affine space consisting of anti-symmetric tensors of power $d$. Upon choosing a 
basis for $E$, its coordinate ring is identified with $A^{(d)}_a=\Sym(\wedge^dE)$. Let $X^{(d)}_a\subset \mathbb A^{{n}\choose{d}}$ be the rank variety consisting of tensors of rank $\leq n-1$. We will find a non-commutative desingularization for $X^{(d)}_a$. \par
Let us review a commutative desingularization. Let $$0\to \cR\to E\times\Grass\to \cQ\to 0$$ be the tautological
sequence over $\Grass$ where $\Grass=\Grass(1,E)$ is the Grassmannian of
lines in $E$. Let $\eta^{(d)}_a$ be the vector bundle $\wedge^d\cQ$. As can be found in   \cite[Section 7.3]{W03}, the total space $\cZ^{(d)}_a$ of the vector bundle $\eta^{(d)*}_a$ is a commutative desingularization, i.e., $\cZ^{(d)}_a=\underline{\Spec}_{\Grass}(\Sym\wedge^d\cQ)$. Alternatively, it can also be defined
as the incidence variety $\cZ^{(d)}_a=\{(S,\phi)\in\Grass\times \mathbb A^{{n}\choose{d}}\mid\phi\in \wedge^dS\subset \wedge^dE^*\}$.
In the setup of \S~\ref{subsec: inverse image til}, take $\Spec R=X^{(d)}_a$, $\mathbb A^{N}=\mathbb A^{{n}\choose{d}}$, $V=\Grass$, and $Z=\cZ^{(d)}_a$. We follow the same notation for the maps among these spaces, which are summarized in Diagram~\eqref{diag: general}.

We consider the inverse image of
$\TIL_K=\oplus_{i=0}^{n-1}\wedge^i\cQ^*$, the Kapranov's
tilting bundle over $\Grass(1,E)$, by $p':\cZ^{(d)}_a\to\Grass$.
We can show
\begin{lemma}
For all $k>0$, and $i$, $j=0,\dots,n-1$, $H^k(\cZ^{(d)}_a p'^*\sHom_{\sO_{\Grass}}(\wedge^i\cQ^{*},
\wedge^j\cQ^{*}))=0$.
\end{lemma}
\begin{proof}
Use the Cauchy-Littlewood and the Littlewod-Richardson formulas, we
can express the sheaf \[p'^*\sHom_{\sO_{\Grass}}(\wedge^i\cQ^{*},
\wedge^j\cQ^{*})\cong\sHom_{\sO_{\Grass}}(\wedge^i\cQ^{*},
\wedge^j\cQ^{*}\otimes_{\sO_{\Grass}} \Sym (\wedge^d\cQ))\] in the form
$\oplus (\SCH_\gamma\cQ^{*})^{\oplus C_\gamma}$ for some
coefficients $C_\gamma$. By  Theorem~\ref{Bott}, all the
higher cohomology of $$\sHom_{\sO_{\Grass}}(\wedge^i\cQ^{*},
\wedge^j\cQ^{*}\otimes_{\sO_{\Grass}}\Sym(\wedge^d\cQ))$$ vanishes.
\end{proof}
\par

Take the Kapranov's tilting bundle $\TIL_K(1,n)$ over
$\Grass(1,E)$. The rank $n-1$ subvariety $X^{(d)}_a$ of $d$-th anti-symmetric tensors and its desingularization
$\cZ^{(d)}_a$ are as above. From the lemma above, using Theorem~\ref{til-resol criterion}, we
get the following.
\begin{prop}
The bundle $p'^*\TIL_K$ is a tilting bundle
over $\cZ^{(d)}_a$.
\end{prop}

Since $\cZ^{(d)}_a$ is smooth,
$\End_{\cZ^{(d)}_a}(p'^*\TIL_K)$ has finite global dimension.

\begin{lemma}\label{lem:anti_tensor_mcm}
Notation as above, $\End_{\cZ^{(d)}_a}(p'^*\TIL_K)$ is maximal Cohen-Macaulay iff ${{n-2}\choose{d-1}}-n-1\geq0$.
\end{lemma}
\begin{remark}
As can be easily checked, ${{n-2}\choose{d-1}}-n-1\geq0$ if and only if $n\geq6$ and $3\leq d\leq n-3$.
\end{remark}
\begin{proof}[Proof of Lemma~\ref{lem:anti_tensor_mcm}]
By Theorem~\ref{thm:dual_mcm}, it suffices to compute \[H^k(\Grass,\wedge^i\cQ\otimes
\wedge^j\cQ^{*}\otimes\omega_{\Grass}\otimes\wedge^{top}\xi^*\otimes\Sym(\wedge^d\cQ))\] for $k>0$ and $i$, $j=0,\dots,n-1$, here $\xi=\cR\otimes\wedge^{d-1}\cQ$ according to  \cite[7.3.1]{W03}. The sheaf $\omega_{\Grass}=\wedge^{n-1}\cQ^{*}\otimes\cR^{n-1}$. Using Cauchy-Littlewood and Littlewood-Richardson formulas, \[\wedge^i\cQ\otimes
\wedge^j\cQ^{*}\otimes\omega_{\Grass}\otimes\wedge^{top}\xi^*\otimes\Sym(\wedge^d\cQ)\cong\wedge^{j-i}\cQ^*\otimes\Sym\wedge^d\cQ\otimes\wedge^{n-1}\cQ^{{{n-2}\choose{d-1}}-(n-1)-1}\otimes \wedge^nE^{*{{n-1}\choose{d-1}}-(n-1)}.\] The conclusion now follows from the Bott's Theorem.
\end{proof}

We know that $\sO_{\Grass}$ is a direct summand of $\TIL_K$. In the case ${{n-2}\choose{d-1}}-n-1\geq0$, the module $q_*'p'^*\TIL_K$ is maximal Cohen-Macaulay. In particular, it is reflexive. According to Proposition~\ref{til-resol criterion} and Lemma~\ref{lem:mcm_nccr}, using Lemma~\ref{lem:anti_tensor_mcm}, we get the following Proposition.
\begin{prop}
Notation as above, in the case ${{n-2}\choose{d-1}}-n-1\geq0$, the map $$\End_{\cZ^{(d)}_a}(p'^*\TIL_K)\to \End_{A^{(d)}_a}(q_*'p'^*\TIL_K)$$ is an
isomorphism of rings and $\End_{\cZ^{(d)}_a}(p'^*\TIL_K)$ is a non-commutative crepant desingularization of $X^{(d)}_a$.
\end{prop}

\begin{remark}
In this case, the tilting bundle over $\cZ^{(d)}_a$ is the inverse image of an exceptional collection $\{\wedge^i\cQ^*\mid i=0,\dots,n-1\}$ over $\Grass$. As can be checked using definition, the dual collection is given by $\wedge^{n-1}\cQ^*\otimes \SCH_{n-i}\cR[n-1-i]$.
\end{remark}
Let us take n=6 and d=3, for an example. One can prove as in Proposition~\ref{prop: Exts Sym}, that for any $\alpha$, $\beta=0,\dots,6$, $\Ext^t(S_\alpha,S_\beta)=H^{t-s-\beta+\alpha}(\Grass,\wedge^s\wedge^3\cQ^*\otimes \SCH_{(6-\beta)}\cR\otimes \SCH_{(6-\alpha)}\cR^*)$. It is known that $\wedge^2\wedge^3\mathbb{C}^6\cong \SCH_{2,2,1,1,0,0}\mathbb{C}^6\oplus \SCH_{1,1,1,1,1,1}\mathbb{C}^6$. In this example, calculation of $\Ext^1$ and $\Ext^2$ gives the quiver 
$$\xymatrix{
\bullet_0\ar@/^1pc/@<1ex>[rrr]^{\beta(\mathbb{C})} &
\bullet_1\ar@/^/@<1ex>[l]^{\alpha(E)}\ar@/^1pc/@<1ex>[rrr]^{\beta(\mathbb{C})} &
\bullet_2\ar@/^/@<1ex>[l]^{\alpha(E)}\ar@/^1pc/@<1ex>[rrr]^{\beta(\mathbb{C})} &
\bullet_3\ar@/^/@<1ex>[l]^{\alpha(E)}\ar@/^1pc/@<1ex>[rrr]^{\beta(\mathbb{C})} &
\bullet_4\ar@/^/@<1ex>[l]^{\alpha(E)}&
\bullet_5\ar@/^/@<1ex>[l]^{\alpha(E)} &
\bullet_6\ar@/^/@<1ex>[l]^{\alpha(E)}
}$$
and the relations:

$$\alpha_i\alpha_{i+1}(\wedge^2E);$$
$$\beta\alpha\beta\alpha+\alpha\beta\alpha\beta(\wedge^2E).$$

\subsection{Cone over a rational normal curve}
Let $E$ be a vector space over $k$ of dimension $n$ and $\mathbb A^{{n+d-1}\choose{d}}=\SCH_dE^*$ be the
affine space consisting of symmetric tensors of power $d$. Upon choosing a set of
basis for $E$, its coordinate ring is identified with $A^{(d)}_s=\Sym(\SCH_dE)$. Let $X^{(d)}_s\subset \SCH_dE^*$ be the rank variety consisting of symmetric tensors of rank $\leq n-1$. We find a non-commutative desingularization for $X^{(d)}_s$ and describe it in certain cases. \par
Let us review a commutative desingularization. Let $$0\to \cR\to E\times\Grass\to \cQ\to 0$$ be the tautological
sequence over $\Grass$ where $\Grass=\Grass(1,E)$ is the Grassmannian of
lines in $E$. Let $\eta^{(d)}_s$ be the vector bundle $\SCH_d\cQ$. As can be found in  \cite[Section 7.3]{W03}, the total space $\cZ^{(d)}_s$ of the vector bundle $(\eta^{(d)}_s)^*$ is a commutative desingularization, i.e., $\cZ^{(d)}_s=\underline{\Spec}_{\Grass}(\Sym \SCH_d\cQ)$. Alternatively, it can also be defined
as the incidence variety $\cZ^{(d)}_s=\{(S,\phi)\in\Grass\times \SCH_dE^*\mid\phi\in \SCH_dS\subset \SCH_dE^*\}$.
In the setup of \S~\ref{subsec: inverse image til}, take $\Spec R=X^{(d)}_s$, $\mathbb A^N=\mathbb A^{{n+d-1}\choose{d}}$, $V=\Grass$, and $Z=\cZ^{(d)}_s$. We follow the same notation for the maps among these spaces, which are summarized in Diagram~\eqref{diag: general}.

\begin{lemma}
For all $k>0$, and $i$, $j=0,\dots,n-1$, $H^k(\cZ^{(d)}_s p'^*\sHom_{\sO_{\Grass}}(\wedge^i\cQ^{*},
\wedge^j\cQ^{*}))=0$.
\end{lemma}

\par

Take the Kapranov's tilting bundle $\TIL_K(1,n)$ over
$\Grass(1,E)$. The rank $n-1$ subvariety $X^{(d)}_s$ of $d$-th symmetric tensors and its desingularization
$\cZ^{(d)}_s$ are as above. From the Lemma above, using Theorem~\ref{til-resol criterion}, we
get the following.
\begin{prop}
The bundle $p'^*\TIL_K$ is a tilting bundle
over $\cZ^{(d)}_s$.
\end{prop}

Since $\cZ^{(d)}_s$ is smooth,
$\End_{\cZ^{(d)}_s}(p'^*\TIL_K)$ has finite global dimension.

\begin{lemma}
Notation as above, $\End_{\cZ^{(d)}_s}(p'^*\TIL_K)$ is maximal Cohen-Macaulay.

\end{lemma}
\begin{proof}
By Theorem~\ref{thm:dual_mcm}, it suffices to compute \[H^k(\Grass,\wedge^i\cQ\otimes
\wedge^j\cQ^{*}\otimes\omega_{\Grass}\otimes\wedge^{top}\xi^*\otimes\Sym(\SCH_d\cQ))\] for $k>0$ and $i$, $j=0,\dots,n-1$, here $\xi=\cR\otimes \SCH_{d-1}E$ according to  \cite[7.2.1]{W03}. The sheaf $\omega_{\Grass}=\wedge^{n-1}\cQ^{*}\otimes\cR^{n-1}$. Using Cauchy-Littlewood and Littlewood-Richardson formulas, \[\wedge^i\cQ\otimes
\wedge^j\cQ^{*}\otimes\omega_{\Grass}\otimes\wedge^{top}\xi^*\otimes\Sym(\SCH_d\cQ)\cong\wedge^{j-i}\cQ^*\otimes\Sym \SCH_d\cQ\otimes\wedge^{n-1}\cQ^{{{n+d-1}\choose{n-1}}-(n-1)-1}\otimes \SCH_\mu E\] for some $\mu$. Now it follows easily from the Bott's Theorem that $\End_{\cZ^{(d)}_s}(p'^*\TIL_K)$ is maximal Cohen-Macaulay iff ${{n+d-1}\choose{n-1}}-n-1\geq0$. But as can be easily seen, ${{n+d-1}\choose{n-1}}-n-1$ is always positive.
\end{proof}

In the case that $n=2$,  $X^{(d)}_s$ is the cone over a rational normal curve. In this case, $\End_{\cZ^{(d)}_s}(p'^*\TIL_K)$ is always maximal Cohen-Macaulay grant that $d\geq2$.

With similar calculation as in Proposition~\ref{prop: Exts Sym}, we get the following:
\begin{prop}
For any two simple objects $S_i$, $S_j$, $i$, $j=0,1$, we have $$\Ext^t(S_i,S_j)=\oplus_{s=0}^1H^{t-s-j+i}(\mathbb{P}^1,(\cQ^*)^{ds}\otimes(\cR^*)^{j-i}).$$
\end{prop}
Plug-in $t=1,2$ and $i$, $j=0,1$, we get the quiver with relations for the non-commutative desingularization.
$$\xymatrix{
\bullet_0\ar@/^/@<1ex>[r]^{\beta(\SCH_{d-1}E)} &
\bullet_1\ar@/^/@<1ex>[l]^{\alpha(E)}
 }$$
and the relations:

$$\alpha\beta(\SCH_{(d-1,1)}E);$$
$$\beta\alpha(\SCH_{(d-1,1)}E).$$

Other examples of representations with finitely many orbits will be analyzed in subsequent papers.

\newcommand{\arxiv}[1]
{\texttt{\href{http://arxiv.org/abs/#1}{arXiv:#1}}}
\newcommand{\doi}[1]
{\texttt{\href{http://dx.doi.org/#1}{doi:#1}}}
\renewcommand{\MR}[1]
{\href{http://www.ams.org/mathscinet-getitem?mr=#1}{MR#1}}

\end{document}